\documentclass[11pt,a4paper]{report}
\usepackage{amssymb}
\usepackage{latexsym}
\usepackage{graphicx}
\usepackage{amsfonts}
\usepackage{amsthm}
\usepackage{mathtools} 
\usepackage[margin=10pt,font=small,labelfont=bf]{caption} 
\usepackage[margin=1.5in]{geometry}
\usepackage{cleveref} 
\usepackage{enumerate}

\newtheorem{theorem}{Theorem}
\newtheorem{proposition}[theorem]{Proposition}
\newtheorem{corollary}[theorem]{Corollary}
\newtheorem{lemma}[theorem]{Lemma}
\newtheorem{conjecture}[theorem]{Conjecture}

\theoremstyle{definition} 
\newtheorem*{definition}{Definition}

\newtheorem*{claim}{Claim}

\theoremstyle{plain}

\providecommand{\norm}[1]{\left\lVert#1\right\rVert}

\title{The Regularity Method\\ for Graphs and Digraphs}
\author{Amelia Taylor}
\date{March 26, 2013}
\begin{document}
\maketitle

\setcounter{tocdepth}{1}
\tableofcontents


\chapter{Introduction}
\label{chap:intro}
Extremal graph theory concerns whether we can embed a given graph into a graph $G$. In this project the subgraph of interest will often be a Hamilton cycle, that is a cycle containing all of the vertices of the graph. The problem of finding a Hamilton cycle in a graph is NP-complete and so it is unlikely that we can find a complete classification of those graphs that do contain a Hamilton cycle. Therefore, we instead try to find sufficient conditions that will ensure a Hamilton cycle.

Initially in this project we will focus on undirected graphs but in later chapters we introduce analogues of many of the definitions and results for directed graphs. We say that a bipartite graph is $\varepsilon$-regular if its edges are fairly evenly distributed and we begin \Cref{chap:regularity} by giving formal definitions of $\varepsilon$-regularity and $(\varepsilon,d)$-superregularity.

Central to this project will be Szem\'eredi's Regularity Lemma, \Cref{reglemma}. Informally, it says that we can partition the vertices of any sufficiently large graph into a bounded number of clusters so that the edges are fairly evenly distributed between these clusters. This lemma will be a powerful tool when combined with the Key Lemma, \Cref{keylem}, and Blow-up Lemma, \Cref{blowup}, allowing us to embed structures into a graph. In \Cref{chap:applications} we give some applications of the Regularity Lemma, including a proof of the well-known Erd\H{o}s-Stone theorem.

In \Cref{chap:digraphs}, we investigate the robust outexpansion property for digraphs. By showing that a sufficiently large digraph satisfying certain degree conditions is a robust outexpander, we are able to prove an approximate version of a conjecture of Nash-Williams.

We conclude the project by considering what conditions are sufficient to ensure any orientation of a Hamilton cycle in a digraph. That is, we no longer require that the edges are oriented consistently as a directed cycle, the edges may change direction along the cycle.

It should be noted that throughout this project we have omitted floors and ceilings where this does not affect the argument given.

\chapter{Notation}
\label{chap:notation}

\section{Notation for Graphs}
Let $G=(V,E)$ be a graph. We will write $|G|$ for the order of $G$, that is, the number of vertices $G$ has, and $e(G)$ for the number of edges. For a vertex $v\in V(G)$, the neighbourhood of $v$ is the set of all vertices adjacent to $v$ in $G$ and is indicated by $N_G(v)$. The degree of $v$ is defined to be $d_G(v)=|N_G(v)|$. Where it is clear which graph we are considering, we may omit the subscript, writing just $N(v)$ and $d(v)$. We write $\delta(G)$ for the minimum degree of any vertex in $G$ and $\Delta(G)$ for the maximum degree.

For a set of vertices $A$, we will write $N_G(A)$ for the neighbourhood of $A$ which we understand to be $\bigcup_{a\in A} N_G(a)$. We write $G[A]$ to indicate that subgraph of $G$ with vertex set $A$ and all edges of $G$ which have both end vertices in $A$ and we call this graph an induced subgraph. We will write $G \setminus A$ to indicate the graph obtained from the graph $G$ by deleting $A$ and any edges incident to $A$.

For disjoint sets of vertices $A,B \subseteq V(G)$, we write $(A,B)_G$ to denote the bipartite graph induced by $G$, that is, the bipartite graph with vertex classes $A$ and $B$ and all edges in $G$ from a vertex in $A$ to a vertex in $B$. We write $e_G(A,B)$ for the number of edges between $A$ and $B$. 

The complete graph on $n$ vertices has all possible edges and will be denoted $K_n$. We say that the graph $P=u_0u_1\ldots u_k$ is a path if $u_i\neq u_j$ for all $i\neq j$ and if $u_iu_{i+1} \in E(G)$ for all $0\leq i < k$. The length of the path $P$, denoted $\ell(P)$, is equal to $|P|-1$. If $u_i, u_j \in V(P)$ then we denote by $u_i P u_j$ the subgraph of $P$ which is a path from $u_i$ to $u_j$. If $|P| \geq 3$ and $u_0u_k\in E(G)$ then $C=u_0u_1\ldots u_ku_0$ is a cycle. We will write $C_n$ for a cycle on $n$ vertices.

We will write $\overline{G}$ for the graph which has the same vertex set as $G$ and $e$ is an edge in $\overline{G}$ if and only if it is not an edge in $G$, we call this graph the complement of $G$.

For any $k \in \mathbb{N}$, we write $[k]$ for the set $\{1, 2, \ldots, k\}$.

\section{Notation for Digraphs}
Suppose that $G$ is a digraph and let $x \in V(G)$. We write $N^+_G(x)$ for the outneighbourhood of $x$, that is, the set $\{ y \in V(G) : xy \in E(G)\}$, where $xy$ denotes the edge directed from $x$ to $y$ in $G$. The outdegree of $x$ is $d^+(x)= |N^+_G(x)|$ and the minimum outdegree is the $\min_{x \in V(G)} d^+_G(x) \eqqcolon \delta^+(G)$. Similarly, we define the inneighbourhood of $x$, $N^-_G(x)$, the indegree of $x$, $d^-_G(x)$, and the minimum indegree, $\delta^-(G)$. Again, we will often omit the subscript $G$ when it is clear to which graph we refer. We define the minimum semidegree, $\delta^0(G)$, to be the minimum of $\delta^+(G)$ and $\delta^-(G)$.

For sets of vertices $A, B\subseteq V(G)$, we write $(A,B)_G$ to denote the oriented bipartite graph with vertex classes $A$ and $B$ and all edges in $G$ from a vertex in $A$ to a vertex in $B$.  We write $e_{G}(A,B)$ for the number of edges directed from a vertex in $A$ to a vertex in $B$.

In \Cref{chap:digraphs}, we assume that all paths and cycles are directed paths and cycles. As in the undirected case, we will write $uPv$ to indicate the section of the path $P$ starting at the vertex $u$ and ending at $v$. If $C$ is a cycle then $uCv$ indicates the path along the cycle $C$ from the vertex $u$ to the vertex $v$. For a vertex $v$ on a path (or cycle) $P$ we will often write $v^-$ and $v^+$ to indicate its predecessor and successor on the path (or cycle).

An \emph{oriented graph} is a graph $G$ which can be obtained by orienting a simple undirected graph. That is, for all vertices $x$ and $y$, if $xy \in E(G)$ then $yx \notin E(G)$. A \emph{tournament} is an oriented complete graph.

\chapter{Regularity}
\label{chap:regularity}

Throughout these first sections we will restrict ourselves to undirected graphs, although later we will introduce analogues of many of the definitions and results covered which can be applied to directed graphs.

\section{Density and Regularity}\label{subsec:reg}

We will require some definitions before introducing the Regularity Lemma; the first of these is that of density. Density is a measure of the proportion of the maximum possible number of edges which are present in a bipartite graph and it takes a value between $0$ and $1$. A graph with density $0$ would have no edges whilst a graph with density $1$ has all possible edges, that is, it is a complete bipartite graph.

\begin{definition}
The \emph{density} of a bipartite graph $G=(A,B)$ with vertex classes $A$ and $B$ is defined to be 
$$d_G(A,B) \coloneqq \frac{e_G(A,B)}{|A||B|}.$$
\end{definition}

We will sometimes omit the subscript $G$, writing instead $d(A,B)$, when it is clear to which graph we refer.

Another important definition is that of $\varepsilon$-regularity. If a bipartite graph $G=(A,B)$ is $\varepsilon$-regular, this indicates that the edges between the vertex classes are fairly evenly distributed. The smaller the value of $\varepsilon$ the more uniform the distribution. We also define what it means for $G$ to be $(\varepsilon,d)$-superregular. Superregularity introduces minimum bounds on the degree of every vertex in $G$ and also on the density of the bipartite subgraphs induced by (sufficiently large) subsets of the vertex classes $A$ and $B$.

\begin{definition}\label{def:reg}
Given $\varepsilon>0$, we say that the bipartite graph $G=(A,B)$ is \emph{$\varepsilon$-regular} if for all sets $X \subseteq A$ and $Y \subseteq B$ with $|X|\geq \varepsilon |A|$ and $|Y| \geq \varepsilon |B|$ we have $$|d_G(A,B) - d_G(X,Y)| < \varepsilon.$$

Given $d \in [0,1)$, we say that $G$ is \emph{$(\varepsilon,d)$-superregular} if all sets $X \subseteq A$ and $Y \subseteq B$ with $|X|\geq \varepsilon |A|$ and $|Y| \geq \varepsilon |B|$ satisfy $d(X,Y)>d$ and, furthermore, if $d_G(a)>d|B|$ for all $a \in A$ and $d_G(b)>d|A|$ for all $b \in B$.
\end{definition}

We will now use these definitions to prove some simple results which will be used in various applications of the Regularity Lemma in \Cref{chap:applications}. The first of these propositions concerns the degrees of the vertices in $A$. We will obtain a bound on the number of vertices in $A$ which have few neighbours in a sufficiently large subset of $B$. This result will be useful when we come to embed structures in graphs, for example, in the proofs of \Cref{keylem} and \Cref{c6packing}.

\begin{proposition}\label{prop:neighbours}
Let (A,B) be an $\varepsilon$-regular pair of density d. Suppose that $Y\subseteq B$ with $|Y| \geq \varepsilon |B|$. Let $X \subseteq A$ be a set of vertices each having at most $(d-\varepsilon)|Y|$ neighbours in Y. Then $|X| < \varepsilon |A|$.
\end{proposition}

\begin{proof}
Suppose that $X \subseteq A$ is a set of vertices each having at most $(d-\varepsilon)|Y|$ neighbours in $Y$. Then $e(X,Y) \leq (d- \varepsilon) |Y| |X|$ which means that $$d(X,Y) = \frac{e(X,Y)}{|X||Y|} \leq d- \varepsilon.$$
Then, since $(A,B)$ is $\varepsilon$-regular, we must have that $|X|< \varepsilon |A|$.
\end{proof}

If a bipartite graph is $\varepsilon$-regular, then we can easily see that its complement must also be $\varepsilon$-regular and we verify this in the following proposition.

\begin{proposition}\label{prop:complement}
Let $G$ be a graph, $A,B\subseteq V(G)$ be disjoint sets and suppose that $(A,B)_G$ is $\varepsilon$-regular. Then $(A,B)_{\overline{G}}$ is also $\varepsilon$-regular.
\end{proposition}

\begin{proof}
Suppose that $(A,B)_G$ is $\varepsilon$-regular. Consider sets $X\subseteq A$ and $Y \subseteq B$ with $|X| \geq \varepsilon |A|$ and $|Y| \geq \varepsilon |B|$. 
We observe that $$d_{\overline{G}}(X,Y) = \frac{|X||Y|-d_G(X,Y)|X||Y|}{|X||Y|} = 1-d_G(X,Y).$$
So we see that
\begin{equation*}
\begin{split}
|d_{\overline{G}}(A,B)-d_{\overline{G}}(X,Y)| &= |(1-d_G(A,B))-(1-d_G(X,Y))|\\
&= |d_G(X,Y)-d_G(A,B)|\\
& <\varepsilon.
\end{split}
\end{equation*}
Hence, $(A,B)_{\overline{G}}$ is also $\varepsilon$-regular.
\end{proof}

Another result, again following directly from the definitions, states that if we choose sufficiently large subsets of $A$ and $B$ then the subgraph induced by these sets will be $\varepsilon'$-regular for some $\varepsilon' \geq 2\varepsilon$ and gives a minimum bound for the density. So this tells us that we can preserve regularity when removing a small number of vertices from an $\varepsilon$-regular bipartite graph.

\begin{proposition}\label{prop:subsets}
Suppose that $0< \varepsilon \leq \alpha \leq 1/2$. Let $(A,B)$ be a $\varepsilon$-regular pair of density $d$. If $A' \subseteq A, B' \subseteq B$ with $|A'| \geq \alpha |A|$ and $|B'| \geq \alpha |B|$ then $(A',B')$ is $\varepsilon / \alpha$-regular and has density greater than $d-\varepsilon$.
\end{proposition}

\begin{proof}
Since $(A,B)$ is $\varepsilon$-regular we have that $|d - d(A',B')| < \varepsilon$ and so $$d(A',B')>d- \varepsilon.$$
\\Now suppose that $X \subseteq A'$ and $Y \subseteq B'$ with $|X|\geq \varepsilon |A'|/\alpha \geq \varepsilon |A|$ and $|Y| \geq \varepsilon |B'|/\alpha \geq \varepsilon |B|$. We have that 
\begin{equation*}
\begin{split}
|d(A',B') - d(X,Y)| &= |d(A',B') - d + d - d(X,Y)| \\
&\leq |d(A',B') - d| + |d - d(X,Y)| \\
&< 2\varepsilon \leq \varepsilon / \alpha
\end{split}
\end{equation*}
since  $(A,B)$ is $\varepsilon$-regular. Hence $(A',B')$ is $\varepsilon / \alpha$-regular.
\end{proof}

We can also add a small number of vertices to a regular bipartite graph and maintain regularity, as demonstrated in the following proposition.

\begin{proposition}\label{prop:supersets}
Suppose that $0< \varepsilon < d \leq 1$ with $2\sqrt{\varepsilon}<d$. Let $(A,B)$ be a $\varepsilon$-regular pair of density $d$. Suppose that at most $\sqrt{\varepsilon}|A|$ vertices are added to $A$ to obtain $A'$ and at most $\sqrt{\varepsilon}|B|$ vertices are added to $B$ to obtain $B'$. Then the graph $(A',B')$ is $5\sqrt[4]{\varepsilon}$-regular with density at least $d-2\sqrt[4]{\varepsilon}$.
\end{proposition}

\begin{proof}
Consider subsets $X \subseteq A'$ and $Y \subseteq B'$ with sizes $|X| \geq 5 \sqrt[4]{\varepsilon}|A'| \geq \sqrt[4]{\varepsilon}|A|$ and $|Y| \geq 5 \sqrt[4]{\varepsilon}|B'| \geq \sqrt[4]{\varepsilon}|B|$. Let $X_0=X\cap A, Y_0=Y\cap B$ and $X_1=X\setminus X_0, Y_1=Y\setminus Y_0$, so $X_1$ and $Y_1$ are the new vertices contained in $X$ and $Y$. We will obtain lower and upper bounds on the density of $(X,Y)$.

First we consider a lower bound. The lowest density is obtained if we have no edges between the new vertices and the original graph. We find that 

\begin{equation*}
\begin{split}
d(X,Y) &\geq \frac{e(X_0,Y_0)}{|X||Y|} = \frac{d(X_0,Y_0)|X_0||Y_0|}{|X||Y|}\\
&>\frac{(d-\varepsilon)(|X|-\sqrt{\varepsilon}|A|)(|Y|-\sqrt{\varepsilon}|B|)}{|X||Y|}\\
&= (d-\varepsilon)\left(1-\sqrt{\varepsilon}\frac{|A|}{|X|}\right)\left(1-\sqrt{\varepsilon}\frac{|B|}{|Y|}\right)\\
& \geq (d-\varepsilon)(1-\sqrt[4]{\varepsilon})(1-\sqrt[4]{\varepsilon})\\
& = d-\varepsilon -2\sqrt[4]{\varepsilon}(d-\varepsilon)+ \sqrt{\varepsilon}(d-\varepsilon)\\
& \geq d-\varepsilon-2\sqrt[4]{\varepsilon}+\sqrt{\varepsilon}(2\sqrt{\varepsilon}-\varepsilon)\\
&\geq d-2\sqrt[4]{\varepsilon}.
\end{split}
\end{equation*}

We obtain an upper bound on the density by considering the case where the vertices $X_1, Y_1$ are joined to all vertices in $X,Y$ respectively. We see that

\begin{equation*}
\begin{split}
d(X,Y) &\leq \frac{e(X_0,Y_0)+|X||Y_1| + |X_1||Y|}{|X||Y|} \\
&< d+\varepsilon+ \frac{|X|\sqrt{\varepsilon}|B| + \sqrt{\varepsilon}|A|}{|X||Y|}\\
& \leq d+\varepsilon + \frac{\sqrt{\varepsilon}|B|}{|Y|} + \frac{\sqrt{\varepsilon}|A|}{|X|}\\
& \leq d+\varepsilon + \frac{\sqrt{\varepsilon}|B|}{\sqrt[4]{\varepsilon}|B|} + \frac{\sqrt{\varepsilon}|A|}{\sqrt[4]{\varepsilon}|A|}\\
&= d+\varepsilon+2\sqrt[4]{\varepsilon}\\
&\leq d+3\sqrt[4]{\varepsilon}.
\end{split}
\end{equation*}

As $(d+3\sqrt[4]{\varepsilon})-(d-2\sqrt[4]{\varepsilon})=5\sqrt[4]{\varepsilon}$, we can conclude that $(A',B')$ is $5\sqrt[4]{\varepsilon}$-regular and has density at least $d-2\sqrt[4]{\varepsilon}$.
\end{proof}

We obtain similar results when we consider the superregularity of a graph instead. When we are embedding a structure in a graph, we may wish to alter the clusters by removing some vertices or adding some vertices which satisfy certain degree properties. The following two propositions show that we can maintain the superregularity of a pair when removing vertices and when adding new vertices, provided that the new vertices have sufficiently many neighbours in the pair.

\begin{proposition}\label{removing}
Suppose that $0<\varepsilon\leq 1/9$ and  $\varepsilon^2<d\leq 1$.
Let $G=(A,B)$ be an $(\varepsilon,d)$-superregular pair. If $A' \subseteq A, B' \subseteq B$ with $|A'| \geq (1-\sqrt{\varepsilon}) |A|$ and $|B'| \geq (1-\sqrt{\varepsilon})|B|$ then $H=(A',B')$ is $(\sqrt{\varepsilon}, d-\sqrt{\varepsilon})$-superregular.
\end{proposition}

\begin{proof}
For any two sets $X \subseteq A', Y\subseteq B'$ with $|X| \geq \sqrt{\varepsilon}|A'|$ and $|Y| \geq \sqrt{\varepsilon}|B'|$ we have that $d_G(X,Y) > d$, since $(A,B)$ was $(\varepsilon,d)$-superregular.

We also have that for all $a \in A'$, $$d_H(a) \geq (d-\sqrt{\varepsilon})|A| \geq (d-\sqrt{\varepsilon})|A'|$$ and for all $b \in B'$, $$d_H(b) \geq (d-\sqrt{\varepsilon})|B| \geq (d-\sqrt{\varepsilon})|B'|.$$
Therefore, $H$ is $(\sqrt{\varepsilon}, d-\sqrt{\varepsilon})$-superregular.
\end{proof}


\begin{proposition}\label{adding}
Let $(A,B)$ be an $(\varepsilon,d)$-superregular pair with $|A|=|B|=m$. Suppose that $A'$ and $B'$ are disjoint sets of vertices of size $|A'|,|B'| \leq \sqrt{\varepsilon}m$ satisfying $|N(a) \cap B| \geq dm/3$ for every vertex $a \in A'$ and $|N(b) \cap A| \geq dm/3$ for every vertex $b \in B'$. Then the graph $H=(A \cup A', B \cup B')$ is $(2 \sqrt{\varepsilon}, d/6)$-superregular.
\end{proposition}

\begin{proof}
Let $A^*=A \cup A'$ and $B^*=B \cup B'$. For any vertex $a \in A^*$ we have that $$d_H(a) \geq dm/3 \geq d|B^*|/6$$ and for any $b \in B^*$ we have $$d_H(b) \geq dm/3 \geq d|A^*|/6.$$ So $H$ satisfies the minimum degree conditions for superregularity.

Now consider any sets $X \subseteq A^*$ and $Y \subseteq B^*$ with $|X| \geq 2 \sqrt{\varepsilon} |A^*| \geq 2\sqrt{\varepsilon} m$ and $|Y| \geq 2\sqrt{\varepsilon}|B^*| \geq 2 \sqrt{\varepsilon} m$. Then we can find sets $X_0 \subseteq A \cap X, Y_0 \subseteq B \cap Y$ with $|X_0|\geq |X|/2 \geq \varepsilon m, |Y_0| \geq |Y|/2 \geq \varepsilon m$. Then we have
$$d(X,Y) = \frac{e(X,Y)}{|X||Y|} \geq \frac{e(X_0,Y_0)}{2|X_0|2|Y_0|} = \frac{d(X_0,Y_0)}{4} \geq \frac{d}{4} > \frac{d}{6}.$$
Hence, $H$ is $(2 \sqrt{\varepsilon}, d/6)$-superregular.
\end{proof}

Let us now consider what structures we can find in an $\varepsilon$-regular graph. We might be interested in finding a matching, defined below, in $G$.
\begin{definition}
A \emph{matching} in a graph $G$ is an independent set of edges of $G$, that is, a set of edges such that no two edges in the set share an endvertex. We say that a matching is \emph{perfect} if all vertices of $G$ are covered.
\end{definition}
The following proposition shows that we can use $\varepsilon$-regularity to prove that a graph contains a perfect matching. It will be useful to recall Hall's theorem.

\begin{theorem}[Hall]\label{hall}
Let $G$ be a bipartite graph with vertex classes $A$ and $B$. $G$ has a matching that covers all the vertices of $A$ if and only if for all subsets $S \subseteq A, |N(S)| \geq |S|$.
\end{theorem}

In particular, if $|A|=|B|$ and $G$ is a bipartite graph satisfying Hall's condition then $G$ must have a perfect matching.

\begin{proposition}\label{matchingundir}
Let $0<d\leq 1$ and $0<3\varepsilon \leq d^2$. Suppose that $G=(A,B)$ is an $\varepsilon$-regular bipartite graph with $|A|=|B|=n$, density $d$ and $\delta(G)\geq (d-\varepsilon) n$. Then $G$ contains a perfect matching.
\end{proposition}

\begin{proof} Let $S \subseteq A$. First we suppose $0<|S| \leq (d-\varepsilon) n$. Let $v \in S$. We know that $d(v) \geq (d-\varepsilon) n$ so $$|N(S)| \geq |N(v)| \geq (d-\varepsilon) n \geq |S|.$$

Let us now suppose that $|S| > (1-(d-\varepsilon))n$. Then, since $|A \setminus S| < (d-\varepsilon) n$, we have that for every $v \in B, N(v) \cap S \neq \emptyset$ as $d(v) \geq (d-\varepsilon)n$. So $N(S) = B$ and therefore $N(S) \geq |S|$.

It remains to check that Hall's condition is satisfied for $S$ where $$\varepsilon n \leq (d-\varepsilon) n \leq |S| \leq (1-(d-\varepsilon))n.$$ Note that $|N(S)| \geq (d-\varepsilon) n \geq \varepsilon n$. We will assume, for the sake of contradiction, that $|N(S)| < |S| \leq (1-(d-\varepsilon))n$. Since for every $v \in S$ we have that $d(v) \geq (d-\varepsilon) n$ we get that $$e(S, N(S)) \geq (d-\varepsilon) n|S|.$$ Hence
\begin{equation*}
\begin{split}
d(S, N(S)) &= \frac{e(S, N(S))}{|S||N(S)|} \geq \frac{(d-\varepsilon) n}{|N(S)|} \\&> \frac{(d-\varepsilon) n}{(1-(d-\varepsilon))n} = d + \frac{d^2-\varepsilon d -\varepsilon}{1-(d-\varepsilon)} \\&\geq d+(d^2-2\varepsilon) \geq d+ \varepsilon.
\end{split}
\end{equation*}
But this contradicts the $\varepsilon$-regularity of $G$. Hence $|N(S)| \geq |S|$.

Therefore, $G$ satisfies the condition of Hall's theorem and, since $|A|=|B|$, has a perfect matching. 
\end{proof}


\section{The Regularity Lemma}

Now that we have all of the necessary definitions, we will introduce the Regularity Lemma. Informally, this lemma states that if we have a sufficiently large graph then we can partition its vertices into a bounded number of sets, or clusters, in such a way that most of the pairs of clusters induce $\varepsilon$-regular bipartite graphs. Importantly, the maximum number of clusters needed does not depend on the number of vertices, only on the values $\varepsilon$ and $k_0$.

\begin{lemma}[Regularity Lemma, Szemer\'{e}di \cite{szem}]\label{reglemma}
For all $\varepsilon>0$ and all integers $k_0$ there is an $N = N(\varepsilon, k_0)$ such that for every $G$ on $n \geq N$ vertices there exists a partition of V(G) into $V_0, V_1, \ldots, V_k$ such that the following hold:
	\begin{enumerate}[(i)]
	\item $k_0 \leq k \leq N$ and $|V_0| \leq \varepsilon n$,
	\item $|V_1| = \cdots = |V_k| \eqqcolon m$,
	\item for all but $\varepsilon k^2$ pairs $1 \leq i < j \leq k$ the graph $(V_i,V_j)_G$ is $\varepsilon$-regular.
	\end{enumerate}
\end{lemma}

The sets $V_i$, for $1\leq i\leq k$, are called \emph{clusters} and the set $V_0$ is called the \emph{exceptional set}. Formally, a partition consists of disjoint, non-empty sets but in this case we will allow the set $V_0$ to be empty. We call a partition of the vertices of $G$ satisfying $(i)$--$(iii)$ an $\varepsilon$\emph{-regular partition}. We will give a proof of the Regularity Lemma based on that given by Alexander Schrijver in \cite{schr}. This proof uses ideas from Euclidean geometry and we will require some preliminary results, definitions and notation.

We first state some definitions concerning partitions of the vertices of a graph $G$ on $n$ vertices. Given a partition $\mathcal{P}=\{P_1,\ldots,P_k\}$ of $V(G)$, we say that a partition $\mathcal{Q}=\{Q_1,\ldots,Q_{k'}\}$ of $V(G)$ is a \emph{refinement} of $\mathcal{P}$ if, for every $Q_i$ in $\mathcal{Q}$, there is a $P_j$ in $\mathcal{P}$ containing $Q_i$. For each $\varepsilon>0$, we say that a partition $\mathcal{P}$ of $V(G)$ is $\varepsilon$\emph{-balanced} if it has a subset $\mathcal{C}\subseteq \mathcal{P}$ such that all classes in $\mathcal{C}$ are the same size and $n-|\bigcup \mathcal{C}| \leq \varepsilon n$. We will call such a subset $\mathcal{C}$ a \emph{balancing subset}.	We say the partition $\mathcal{P}$ is \emph{good} if it has a balancing subset $\mathcal{C}$ in which all but at most $\varepsilon|\mathcal{C}|^2$ pairs are $\varepsilon$-regular.

\begin{lemma}\label{lemma1}
Let $0<\varepsilon<1/4$  and $k>0$. Let $G$ be a graph on $n\geq \varepsilon^{-1}k$ vertices. Suppose that $\mathcal{P}$ is a partition of $V(G)$ with $|\mathcal{P}|\leq k$. Then there exists an $\varepsilon$-balanced refinement $\mathcal{Q}$ of $\mathcal{P}$ such that $|\mathcal{Q}| \leq (1+\varepsilon^{-1})|\mathcal{P}|$ and if $\mathcal{C}$ is any balancing subset of $\mathcal{Q}$ then $|\mathcal{C}| \geq |\mathcal{P}|$. 
\end{lemma}

\begin{proof}
Let $$t \coloneqq \left\lceil \frac{\varepsilon n}{|\mathcal{P}|} \right\rceil .$$
Divide each class in $\mathcal{P}$ into classes of size $t$ and at most one class of size less than $t$ to get a refinement $\mathcal{Q}$. We have that $|\mathcal{Q}| \leq |\mathcal{P}| + n/t \leq (1+ \varepsilon^{-1}) |\mathcal{P}|$ and the number of vertices contained in classes of size less than $t$ is at most $|\mathcal{P}| (\varepsilon n/|\mathcal{P}|) = \varepsilon n$, so $\mathcal{Q}$ is $\varepsilon$-balanced.

Suppose that $\mathcal{C} \subseteq \mathcal{Q}$ is a balancing subset. Then 
$$|\mathcal{C}| \geq \frac{(1-\varepsilon)n}{t}\geq \frac{(1-\varepsilon)n}{\varepsilon n/|\mathcal{P}| +1} \geq \frac{(1-\varepsilon)n}{2\varepsilon n/|\mathcal{P}|} = \frac{1-\varepsilon}{2\varepsilon}|\mathcal{P}|\geq |\mathcal{P}|.$$
\end{proof}

Let $G$ be a graph on $n$ vertices. We will work in the matrix space $\mathbb{R}^{V(G) \times V(G)}$ with the inner product defined by $$\langle M, N \rangle= Tr(M^TN) = \sum_{i,j\in V(G)} a_{i,j}b_{i,j}$$ and Frobenius norm given by
$$\norm{M} = Tr(M^TM)^{1/2} = \left(\sum_{i,j\in V(G)} a_{i,j}^2\right)^{1/2}$$ for all $M=(a_{i,j}), N=(b_{i,j}) \in \mathbb{R}^{V(G)\times V(G)}$.

If $I, J \subseteq V(G)$ are non-empty sets of vertices, let $L_{I,J}$ be the one dimensional subspace of $\mathbb{R}^{V(G)\times V(G)}$ consisting of all matrices which are constant on $I \times J$ and $0$ elsewhere. For each $M \in \mathbb{R}^{V(G)\times V(G)}$, define $M_{I,J}$ to be the orthogonal projection of $M$ onto $L_{I,J}$. Let $e$ be the unit vector generating $L_{I,J}$ which is equal to $1/|I||J|$ on $I \times J$ and $0$ elsewhere. We have that $M_{I,J} = \langle M,e \rangle e$ and so on $I \times J$ the entries of $M_{I,J}$ are equal to the average value of $M$ on $I \times J$ and outside $I \times J$ the entries are $0$.

For any partition $\mathcal{P}$ of $V(G)$, let $L_\mathcal{P}$ be the sum of the spaces $L_{I,J}$ with $I,J \in \mathcal{P}$. We define $M_\mathcal{P}$ to be the orthogonal projection of $M$ onto $L_\mathcal{P}$. Then $$M_\mathcal{P} = \sum_{I,J \in \mathcal{P}} M_{I,J}.$$
Observe that if $\mathcal{Q}$ is a refinement of $\mathcal{P}$ then $L_\mathcal{P} \subseteq L_\mathcal{Q}$ and so \begin{equation}\label{eq:refinement}
\norm{M_\mathcal{P}} \leq \norm{M_\mathcal{Q}}.
\end{equation}

We will require Pythagoras' theorem and a consequence of the Cauchy-Schwartz inequality.

\begin{theorem}[Pythagoras]\label{pythag}
Let $X$ be an inner product space. Suppose that $x, y \in X$ are orthogonal vectors. Then $$\norm{x+y}^2 = \norm{x}^2+ \norm{y}^2.$$
\end{theorem}

\begin{lemma}[Cauchy-Schwartz Inequality]\label{c-s}
Let $a_i,b_i$ be real numbers for $i=1, \ldots, n$. Then
$$\left(\sum_{i=1}^n a_ib_i\right)^2 \leq \sum_{i=1}^n a_i^2 \sum_{i=1}^n b_i^2.$$
\end{lemma}

\begin{proposition}\label{coroll:c-s}
Suppose that $A \in \mathbb{R}^{V(G) \times V(G)}$ and $\mathcal{P}$ is a partition of $V(G)$. Let $(I_1,J_1), \ldots ,(I_r,J_r)$ be distinct pairs of classes in $\mathcal{P}$. Suppose that $X_k \subseteq I_k$, $Y_k \subseteq J_k$ for all $1 \leq k \leq r$. Then
$$\norm{A}^2 \geq \sum_{k=1}^r\norm{A_{X_k,Y_k}}^2.$$
\end{proposition}

\begin{proof}
Let $A=(a_{i,j})$ and $A_{X_k,Y_k} = (a^{(k)}_{i,j})$ for each $1 \leq k \leq r$. Recall that $$a^{(k)}_{i,j} = \frac{1}{|X_k||Y_k|}\sum\limits_{i\in X_k}\sum\limits_{j\in Y_k} a_{i,j}$$ if $i \in X_k$ and $j \in Y_k$ and $0$ otherwise. So we can apply the Cauchy-Schwartz inequality to see that, for each $1\leq k \leq r$,
$$\norm{A_{X_k,Y_k}}^2 =|X_k||Y_k| \left(\sum\limits_{i\in X_k}\sum\limits_{j\in Y_k} a_{i,j} \cdot \frac{1}{|X_k||Y_k|}\right)^2\leq \sum\limits_{i\in X_k}\sum\limits_{j\in Y_k} a_{i,j}^2.$$
So we obtain that
$$\norm{A}^2 = \sum_{i,j\in V(G)} a_{i,j}^2 \geq \sum_{k=1}^r\sum_{i\in X_k}\sum_{j\in Y_k} a_{i,j}^2 \geq \sum_{k=1}^r \norm{A_{X_k,Y_k}}^2.$$
\end{proof}

Given a graph $G$ on $n$ vertices, we define the adjacency matrix of $G$ to be the matrix $A= (a_{i,j})\in \mathbb{R}^{V(G) \times V(G)}$ with $a_{i,j} = 1$ if $ij \in E(G)$ and $0$ otherwise. We will consider the adjacency matrix of $G$ and its orthogonal projection onto subspaces of $\mathbb{R}^{V(G) \times V(G)}$ defined by partitions of $V(G)$. We will show that, by refining a partition which is not good, we can increase the value of $\norm{A_\mathcal{P}}^2$ by some fixed amount. Note that if $\mathcal{Q}$ is any partition, the partition consisting of entirely of singletons is a refinement of $\mathcal{Q}$ and so we have that
$$\norm{A_\mathcal{Q}}^2 \leq \norm{A}^2 \leq n^2.$$ So, after a finite number of steps we will show that we can obtain a good partition.

\begin{lemma}\label{lemma2}
Let $0<\varepsilon<1/2$ and let $G$ be a graph on $n$ vertices with adjacency matrix $A$. Suppose that $\mathcal{P}$ is an $\varepsilon$-balanced partition of $V(G)$ which is not good. Then $\mathcal{P}$ has a refinement $\mathcal{Q}$ with
$$|\mathcal{Q}| \leq |\mathcal{P}|2^{|\mathcal{P}|} \text{ and } \norm{A_\mathcal{Q}}^2 > \norm{A_\mathcal{P}}^2 + \varepsilon^5n^2/4.$$
\end{lemma}

\begin{proof}
Let $(I_1,J_1), (I_2,J_2), \ldots ,(I_r,J_r)$ be the pairs of classes in $\mathcal{P}$ which are not $\varepsilon$-regular.  By the definition of $\varepsilon$-regularity, for each $i=1, \ldots, r$ we can choose sets $X_i \subseteq I_i$ and $Y_i \subseteq J_i$ with $|X_i| \geq \varepsilon |I_i|$ and $|Y_i| \geq \varepsilon |J_i|$ such that
$$|d(X_i, Y_i) - d(I_i,J_i)| \geq \varepsilon.$$

For each $K \in \mathcal{P}$, we will define a partition $\mathcal{Q}_K$ of $K$. Consider the set $K' \coloneqq \{X_i : I_i=K\} \cup \{Y_i : J_i =K\}.$ We put two vertices of $K$ in the same class in $\mathcal{Q}_K$ if and only if they lie in exactly the same elements of $K'$. So $1 \leq |\mathcal{Q}_K| \leq 2^{|\mathcal{P}|}$. Define $\mathcal{Q}\coloneqq \bigcup_{K \in \mathcal{P}} \mathcal{Q}_K.$ Then $\mathcal{Q}$ is a refinement of $\mathcal{P}$ and $$|\mathcal{Q}| \leq |\mathcal{P}|2^{|\mathcal{P}|}.$$

Now, for each $i=1, \ldots, r$, the sets $X_i$ and $Y_i$ are the union of classes of $\mathcal{Q}$, so $L_{X_i,Y_i} \subseteq L_\mathcal{Q}$ giving that $(A_\mathcal{Q})_{X_i,Y_i} = (\sum_{I,J \in \mathcal{Q}}A_{I,J})_{X_i,Y_i} = A_{X_i,Y_i}$. Also, $A_{X_i,Y_i}$ and $A_{\mathcal{P}}$ are constant on $X_i \times Y_i$ with values $d(X_i,Y_i)$ and $d(I_i,J_i)$ respectively. So we get that
\begin{equation}\label{eq:IJbound}
\begin{split}
\norm{(A_\mathcal{Q}-A_\mathcal{P})_{X_i,Y_i}}^2 & = \norm{A_{X_i,Y_i} - (A_\mathcal{P})_{X_i,Y_i}}^2\\
& = |X_i||Y_i| (d(X_i,Y_i) - d(I_i,J_i))^2 \\
& \geq (\varepsilon|I_i|)(\varepsilon |J_i|) \varepsilon^2 = \varepsilon^4|I_i||J_i|.
\end{split}
\end{equation}

Recall that $\mathcal{Q}$ is a refinement of $\mathcal{P}$, so $L_{\mathcal{P}} \subseteq L_{\mathcal{Q}}$ and hence $A_\mathcal{P}$ is orthogonal to $(A_\mathcal{Q}-A_\mathcal{P})$. We also know that the vectors $(A_\mathcal{Q}- A_\mathcal{P})_{X_i,Y_i}$ are pairwise orthogonal. So we see that
\begin{align*}
\norm{A_\mathcal{Q}}^2-\norm{A_\mathcal{P}}^2 & = \norm{A_\mathcal{Q}-A_\mathcal{P}}^2 & \text{ (by \Cref{pythag})}\\
& \geq \norm{\sum_{i=1}^r (A_\mathcal{Q}- A_\mathcal{P})_{X_i,Y_i}}^2 & \text{ (by \Cref{coroll:c-s})}\\
&=\sum_{i=1}^r \norm{(A_\mathcal{Q}- A_\mathcal{P})_{X_i,Y_i}}^2 & \text{ (by \Cref{pythag})}\\
& \geq \sum_{i=1}^r \varepsilon^4 |I_i||J_i| & \text{ (by \eqref{eq:IJbound})}\\
& \geq \varepsilon^5(1-\varepsilon)^2 n^2 > \varepsilon^5n^2/4.&
\end{align*}
Therefore, $\norm{A_\mathcal{Q}}^2 > \norm{A_\mathcal{P}}^2 + \varepsilon^5n^2/4$, as required.
\end{proof}

We now combine the previous two results to show that we can obtain an $\varepsilon$-balanced partition of bounded size resulting in a significant increase in $\norm{A_\mathcal{P}}^2$.

\begin{lemma}\label{combined}
Let $0<\varepsilon<1/4$ and $k>0$. Let $G$ be a graph on $n \geq k2^k$ vertices with adjacency matrix $A$. Suppose that $\mathcal{P}$ is an $\varepsilon$-balanced partition of $V(G)$ which is not good and $|\mathcal{P}|\leq k$. Then $\mathcal{P}$ has an $\varepsilon$-balanced refinement $\mathcal{Q}$ such that
$$|\mathcal{Q}| \leq (1+\varepsilon^{-1})|\mathcal{P}|2^{|\mathcal{P}|} \text{ and } \norm{A_\mathcal{Q}}^2 > \norm{A_\mathcal{P}}^2 + \varepsilon^5n^2/4$$ and if $\mathcal{C}$ is any balancing subset of $\mathcal{Q}$ then $|\mathcal{C}| \geq |\mathcal{P}|$.
\end{lemma}

\begin{proof}
First apply \Cref{lemma2} to the partition $\mathcal{P}$ to obtain a refinement $\mathcal{Q}'$ of $\mathcal{P}$ with
$$|\mathcal{Q}'| \leq |\mathcal{P}|2^{|\mathcal{P}|} \text{ and } \norm{A_{\mathcal{Q}'}}^2 > \norm{A_\mathcal{P}}^2 + \varepsilon^5n^2/4.$$

We now apply \Cref{lemma1} to the partition $\mathcal{Q}'$ to obtain an $\varepsilon$-balanced partition $\mathcal{Q}$ with 
$$|\mathcal{Q}| \leq (1+\varepsilon^{-1})|\mathcal{Q}'|\leq (1+\varepsilon^{-1})|\mathcal{P}|2^{|\mathcal{P}|}$$ such that if $\mathcal{C}$ is any balancing subset of $\mathcal{Q}$ then
$$|\mathcal{C}| \geq |\mathcal{Q}'| \geq |\mathcal{P}|.$$
Since $\mathcal{Q}$ is a refinement of $\mathcal{Q}'$, we have, by \eqref{eq:refinement}, that $$\norm{A_\mathcal{Q}}^2 \geq \norm{A_{\mathcal{Q}'}}^2> \norm{A_\mathcal{P}}^2 + \varepsilon^5n^2/4.$$
\end{proof}

We are now in a position to use this result to prove the Regularity Lemma.

\begin{proof}[Proof (of \Cref{reglemma}).]
Suppose $\varepsilon>0$ and $k_0\geq 1$ are given. We may assume, without loss of generality, that $\varepsilon< 1/4$. Define $$s \coloneqq \left\lfloor 4/\varepsilon^5\right\rfloor.$$  We see that we will need to apply \Cref{combined} to an $\varepsilon$-balanced partition which is not good at most $s$ times before obtaining a good partition.

Define $f(x)=(1+\varepsilon^{-1})x2^x$ and let $$N \coloneqq f^s(k_0+1).$$

Now, let $G$ be any graph of order $n\geq N$. We choose an initial partition by letting $C_0 \subseteq V(G)$ be a set of vertices of minimum size such that $n-|C_0|$ is divisible by $k_0$ and then partition the remaining vertices into $k_0$ clusters of equal size. Let $\mathcal{P}_0$ denote the initial partition. We have that $|C_0|<k_0 \leq \varepsilon n$ so this partition is $\varepsilon$-balanced. If this partition is $\varepsilon$-regular then we are done. Otherwise, apply \Cref{combined} to the partition $\mathcal{P}_0$ to obtain a new $\varepsilon$-balanced partition $\mathcal{P}_1$ satisfying the properties given in \Cref{combined}. 

If the resulting partition is good then we are done. Otherwise, repeat this process. Let us denote the partition obtained after $i\geq 1$ applications of \Cref{combined} by $\mathcal{P}_i$. We note that we always have that $|\mathcal{P}_i| \leq (1+\varepsilon^{-1})|\mathcal{P}_{i-1}|2^{|\mathcal{P}_{i-1}|} \leq N$, $\norm{A_{\mathcal{P}_{i}}}^2 > \norm{A_{\mathcal{P}_{i-1}}}^2+\varepsilon^5n^2/4> \norm{A_{\mathcal{P}_0}}^2+i\varepsilon^5n^2/4$ and if $\mathcal{C}_i$ is any balancing subset of $\mathcal{P}_i$ then $|\mathcal{C}_i| \geq |\mathcal{P}_{i-1}| \geq k_0$.

We continue in this way until we obtain a good partition, $\mathcal{P}$, this will take at most $s$ steps and we note that the size of the partition will be at most $N$. $\mathcal{P}$ is a good partition so we can find a balancing subset $\mathcal{C}\subseteq \mathcal{P}$ containing at most $\varepsilon|\mathcal{C}|^2$ pairs which are not $\varepsilon$-regular. All classes in $\mathcal{C}$ are the same size and $k_0 \leq |\mathcal{C}| \leq N$. If we set $V_0 = V(G) \setminus \bigcup\mathcal{C}$, we have that $|V_0| \leq \varepsilon n$. Let $\mathcal{P}'$ be the partition whose sets are $V_0$ together with the sets in $\mathcal{C}$, then this partition satisfies properties $(i)$--$(iii)$.
\end{proof}

\section{The Degree Form of the Regularity Lemma}

We will often find it more convenient to work with the following Degree form of the Regularity Lemma. This alternative form follows from \Cref{reglemma} and we derive it below.

\begin{lemma}[Degree form of the Regularity Lemma]\label{degreeform}
For all $\varepsilon>0$ and all integers $k_0$ there is an $N = N(\varepsilon, k_0)$ such that for every number $d \in [0,1)$ and for every graph $G$ on $n \geq N$ vertices there exist a partition of V(G) into $V_0, V_1, \ldots, V_k$ and a spanning subgraph $G'$ of $G$ such that the following hold:
	\begin{enumerate}[(i)]
	\item $k_0 \leq k \leq N$ and $|V_0| \leq \varepsilon n$,
	\item $|V_1| = \cdots = |V_k| \eqqcolon m$,
	\item $d_{G'}(x) > d_G(x) - (d + \varepsilon)n$ for all vertices $x \in V(G)$,
	\item for all $i \geq 1$ the graph $G'[V_i]$ is empty,
	\item for all $1 \leq i < j \leq k$ the graph $(V_i,V_j)_{G'}$ is $\varepsilon$-regular and has density either 0 or $>d$.
	\end{enumerate}
\end{lemma}

In the proof of this lemma we will make use of the following notation: $$a \ll b.$$ This means that we can find an increasing function $f$ for which all of the conditions in the proof are satisfied whenever $a \leq f(b)$. It is equivalent to setting $a = \min \{f_1(b), f_2(b), \ldots, f_k(b)\}$ where each $f_i(b)$ corresponds to the maximum value of $a$ allowed in order that the corresponding argument in the proof holds. However, in order to simplify the presentation, we will not determine these functions explicitly.

\begin{proof}
Let $\varepsilon>0$, $k_0 \in \mathbb{N}$ and $d \in [0,1)$. We may assume that $\varepsilon\leq 1$. We choose further positive constants $\varepsilon'$ and $k'_0 \in \mathbb{N}$ satisfying $$\frac{1}{k'_0}, \varepsilon' \ll \varepsilon, d, \frac{1}{k_0}.$$

By \Cref{reglemma}, there exists $N'=N'(\varepsilon', k'_0)$ such that, if we let $N\coloneqq \lceil 4N'/\varepsilon\rceil \geq N'$ and $G$ is a graph on $n \geq N$ vertices, $G$ has a partition of its vertices into clusters $$V'_0, V'_1, \ldots, V'_{k'}$$ such that:
\begin{enumerate}[(a)]
\item $k'_0 \leq k' \leq N'$ and $|V'_0| \leq \varepsilon' n$;
\item $|V'_1|= \ldots = |V'_{k'}| \eqqcolon m'$ and
\item for all but $\varepsilon' k'^2$ pairs $1 \leq i < j \leq k'$ the graph $(V'_i,V'_j)_G$ is $\varepsilon'$-regular.
\end{enumerate}

We will remove some edges from the graph $G$ to obtain a graph $G'$ and a partition of its vertices, $V_0,V_1, \ldots, V_k$, which satisfies properties $(i)$--$(v)$ of \Cref{degreeform} by carrying out the following steps:

\begin{enumerate}
\item[1.] For each pair of clusters $V'_i, V'_j$ where $1 \leq i<j \leq k'$, if $(V'_i, V'_j)_G$ is not $\varepsilon'$-regular, then colour all edges between $V'_i$ and $V'_j$ red.
For any $v \in V(G)$, if $v$ is incident to at least $\varepsilon n/10$ red edges then move $v$ to the exceptional set, $V'_0$. Then delete all red edges that do not have an endvertex in $V'_0$
\end{enumerate}
After deleting these edges, we observe that the degree of any vertex $v\in V(G)$ is greater than $d_G(v)-\varepsilon n/10$.

We have at most $$\varepsilon' k'^2 m'^2 \leq \varepsilon'n^2$$ red edges, by (c), and so we have moved at most $$\frac{2\varepsilon' n^2}{\varepsilon n/10}= \frac{20\varepsilon'n}{\varepsilon} \leq \frac{\varepsilon n}{4}$$ vertices to the exceptional set $V'_0$.

\begin{enumerate}
\item[2.] Next, consider each pair of clusters $V'_i, V'_j$ where $1 \leq i<j \leq k'$, and $d_G(V'_i, V'_j)\leq d+\varepsilon'$. Colour all remaining edges between these clusters blue. For each $v \in V'_i$ such that $v$ sends more than $(d+2\varepsilon')m'$ edges to $V'_j$, mark all but $(d+2\varepsilon')m'$ of these edges. Similarly, for each $v \in V'_j$, if $v$ sends more than $(d+2\varepsilon')m'$ edges to $V'_i$, then mark all but $(d+2\varepsilon')m'$ of these edges.
\end{enumerate}
Since $(V'_i,V'_j)_G$ is $\varepsilon'$-regular, we observe that, if $X$ is the set of vertices in $V'_i$ having more that $(d+2\varepsilon')m'$ neighbours in $V'_j$ then, as $$d_G(X, V'_j) > \frac{(d+2\varepsilon')m'|X|}{m'|X|} = d+2\varepsilon',$$ we have that $|X|<\varepsilon'm'$. Similarly, $V'_j$ contains at most $\varepsilon'm'$ vertices having more than $(d+2\varepsilon')m'$ neighbours in $V'_i$. So we mark at most $2\varepsilon'm'^2$ edges between the clusters $V'_i$ and $V'_j$.

We carry out this process for all $\varepsilon'$-regular pairs of clusters with density at most $d+\varepsilon'$. There are at most $\binom{k'}{2}$ such pairs, so in total we mark at most $$\binom{k'}{2} \varepsilon'm'^2 \leq \varepsilon'n^2$$ edges.

\begin{enumerate}
\item[3.] Move any vertex that was adjacent to at least $\varepsilon n/10$ marked edges to the exceptional set $V'_0$ and delete all blue edges that do not have an endvertex in $V'_0$. 
\end{enumerate}
Every vertex loses fewer than $(d+2\varepsilon')m'k'+\varepsilon n/10$ incident edges in this step.
We marked at most $\varepsilon'n^2$ edges, so, in total, we move at most $$\frac{2 \varepsilon'n^2}{\varepsilon n/10} = \frac{20 \varepsilon'n}{\varepsilon}\leq \frac{\varepsilon n}{4}$$ vertices to $V'_0$.

\begin{enumerate}
\item[4.] Delete any edges inside clusters $V'_i, 1\leq i \leq k'$.
\end{enumerate}
So each vertex may lose a further, at most, $m'\leq n/k'\leq n/k'_0 \leq \varepsilon n/5$ incident edges.

\begin{enumerate}
\item[5.] Finally, we ensure that all clusters have same size by splitting each cluster into smaller subclusters of size $\lceil\varepsilon n/(4k')\rceil$. Move the vertices that are leftover in each cluster after this process into the set $V'_0$. Call the new  exceptional set $V_0$ and the new clusters $V_1, V_2, \ldots V_k$.
\end{enumerate}

We have at most $\varepsilon n/(4k')$ vertices left in each of the clusters $V'_i$, $1\leq i \leq k'$, after splitting them and so add at most $$\frac{\varepsilon n}{4k'}k' = \frac{\varepsilon n}{4}$$ further vertices to the exceptional set in this step.

We now check that the graph, $G'$ obtained, together with the vertex partition, satisfies the properties of the lemma; $(ii)$ and $(iv)$ are clear. Let us consider property $(i)$. We have that $$k_0 \leq k'_0 \leq k' \leq k$$ and also $$k \leq \frac{m'}{\varepsilon n/(4k')}k' \leq \frac{4k'}{\varepsilon} \leq \frac{4}{\varepsilon}N' \leq N.$$ So we see that $k_0\leq k\leq N$. Using (c) and that we have added at most $\varepsilon n/4$ vertices to the exceptional set in each of steps 1, 3 and 5, we have that $$|V_0| \leq \varepsilon'n+3\varepsilon n/4 \leq \varepsilon n.$$ So property $(i)$ is satisfied.

For property $(iii)$ we combine our previous observations to see that we have removed fewer than $$\varepsilon n/10 + ((d+2\varepsilon')m'k'+\varepsilon n/10) + \varepsilon n/5 \leq (d+2\varepsilon'+2\varepsilon/5)n \leq (d+\varepsilon)n$$ edges incident at any vertex (in steps 1, 3 and 4). Hence, for every $v \in V(G)$ we have that $$d_{G'}(v) > d_G(v)-(d+\varepsilon)n.$$

Finally we check that property $(v)$ is satisfied. If $V_r, V_s$ are clusters of $G'$, then either $(V_r,V_s)_{G'}$ has density $0$, or it is the subgraph of an $\varepsilon'$-regular pair of clusters $(V'_i,V'_j)_G$ of density at least $d+\varepsilon'$. Let us assume that $d(V_r,V_s)_{G'}\neq 0$. Then, we can apply \Cref{prop:subsets}, since $|V_r|=|V_s| \geq \varepsilon m'/4$, to see that $(V_i,V_j)_{G'}$ is $\varepsilon'/(\varepsilon/4)$-regular with density $>(d+\varepsilon')-\varepsilon'=d$. Hence, we see that, since $\varepsilon'$ is sufficiently small, $(V_r,V_s)_{G'}$ is $\varepsilon$-regular and has density greater than $d$, as required.
\end{proof}
	
The graph $G'$ is referred to as the \emph{pure graph}. We define another graph,  the \emph{reduced graph} $R$, as follows. $R$ has vertices $V(R) = \{V_1,\ldots,V_k\}$ and, for each $V_i,V_j \in V(R)$, $V_iV_j$ is an edge of $R$ if the subgraph $(V_i,V_j)_{G'}$ is $\varepsilon$-regular and has density greater than $d$. The following proposition shows that there is a close relationship between the minimum degree of $G$ and the minimum degree of $R$.

\begin{proposition}\label{mindeg}
Suppose that $0<2\varepsilon \leq d \leq c/2$ and let $G$ be a graph with $\delta(G) \geq cn$. Let $R$ be the reduced graph of $G$ with parameters $\varepsilon,d$. Then $$\delta(R) \geq (c-2d)|R|.$$
\end{proposition}

\begin{proof}
Consider any $V_i \in V(R)$ and let $x \in V_i$ in $G$. We observe that $x$ has neighbours in at least $(d_{G'}(x)-|V_0|)/m$ different clusters $V_j$ in $G'$. By part $(v)$ of \Cref{degreeform} and the definition of $R$, $V_i$ is a neighbour of each of these clusters $V_j$ in $R$ so we have $$d_R(V_i) \geq (d_{G'}(x)-|V_0|)/m \geq (d_{G'}(x)-\varepsilon n)/m.$$ From part $(iii)$ of \Cref{degreeform}, we also have that $$d_{G'}(x) > d_G(x)-(d+\varepsilon)n \geq (c-(d+\varepsilon))n.$$ 
Combining these inequalities, we obtain that $$d_R(V_i) \geq (c-(d+2\varepsilon))n/m \geq (c-2d)|R|$$ and hence $\delta(R) \geq (c-2d)|R|$.
\end{proof}

The next result shows that if we have a Hamilton path in the reduced graph $R$ then we are able to find large subclusters of each of the $V_i$ so that the graphs induced by the pairs of subclusters corresponding to edges in the path are superregular. In fact, we could obtain a similar result for any subgraph of $R$ with bounded maximum degree.

\begin{proposition}\label{superreg}
Suppose that $4\varepsilon < d \leq 1$ and that P is a Hamilton path in R. Then every cluster $V_i$ contains a subcluster $V'_i \subseteq V_i$ of size $m-2\varepsilon m$ such that $(V'_i,V'_j)_{G'}$ is $(2\varepsilon, d-3 \varepsilon)$-superregular for every edge $V_iV_j \in E(P)$.
\end{proposition}

\begin{proof}
By relabelling if necessary, we may assume that $$P=V_1V_2 \ldots V_k.$$ Consider any $i<k$. Then, since $(V_i,V_{i+1})_{G'}$ is $\varepsilon$-regular, we may apply \Cref{prop:neighbours} to see that $V_i$ contains at most $\varepsilon m$ vertices $x$ such that $|N_{G'}(x) \cap V_{i+1}| \leq (d-\varepsilon)m$. Similarly, for all $i>1$, since $(V_{i-1},V_{i})_{G'}$ is $\varepsilon$-regular, we have that there are at most $\varepsilon m$ vertices $x \in V_{i}$ with $|N_{G'}(x) \cap V_{i-1}| \leq (d-\varepsilon)m$.
So, for each $i=1,\ldots,k$, we may choose a set of vertices $V'_i \subseteq V_i$ of size $m-2 \varepsilon m \eqqcolon m'$ which does not contain any of these vertices.

Now we need to check the conditions for $(2\varepsilon, d-3 \varepsilon)$-superregularity. Let $i<k$ and consider $X \subseteq V'_i, Y \subseteq V'_{i+1}$ with $|X|,|Y| \geq 2 \varepsilon m'$. By \Cref{prop:subsets}, we have that $H_i \coloneqq (V'_i,V'_{i+1})_{G'}$ is $2\varepsilon$-regular and has density greater than $d- \varepsilon$, and so $$d(X,Y) > d(V'_i, V'_{i+1}) - 2\varepsilon > d- 3 \varepsilon.$$
Also, $$d_{H_i}(a) > (d-\varepsilon)m - 2 \varepsilon m \geq (d - 3 \varepsilon) m', \hspace{6pt} \forall a \in V'_{i}$$ and 
$$d_{H_i}(b) > (d-\varepsilon)m - 2 \varepsilon m \geq (d - 3 \varepsilon) m', \hspace{6pt} \forall b \in V'_{i+1}.$$
So we have that $H_i=(V'_i,V'_{i+1})_{G'}$ is $(2\varepsilon, d-3 \varepsilon)$-superregular, as required.
\end{proof}

\chapter{Applications of the Regularity Lemma}
\label{chap:applications}

\section{How to Apply the Regularity Lemma}

In this section we will show how we can use the Regularity Lemma to embed a structure into a graph $G$. Our procedure will often take the following form. First we obtain an $\varepsilon$-regular partition of the graph $G$ using the Regularity Lemma and from this we obtain the reduced graph. We then look to embed a simpler structure into the reduced graph. If we are able to do this, we can then apply a result such as the Key Lemma (\Cref{keylem}) or the Blow-up Lemma (\Cref{blowup}) to embed the desired structure in the graph $G$. This process is made more complicated when the structure we wish to find is spanning in $G$, for example a Hamilton cycle, as we must then ensure that all of the exceptional vertices are also incorporated. 

Given a graph $G$ which admits an $\varepsilon$-regular partition,	we define the \emph{regularity graph}, $R$, with parameters $\varepsilon$ and $d$ to be the graph with vertex set $\{V_1, \ldots V_k\}$ and an edge from $V_i$ to $V_j$ if $(V_i,V_j)_G$ is $\varepsilon$-regular with density at least $d$. The regularity graph is almost identical to the reduced graph we defined previously, the only difference being that the regularity graph also has an edge between $V_i$ and $V_j$ if $d_G(V_i,V_j)=d$.

Given any graph $R$, we define the graph $R^s$ to be the graph formed by replacing each vertex of $R$ by a set of $s$ vertices and replacing the edges of $R$ by complete bipartite graphs. We illustrate this in \Cref{fig:R2}.
\begin{figure}[h]
\centering
\includegraphics[scale=0.45]{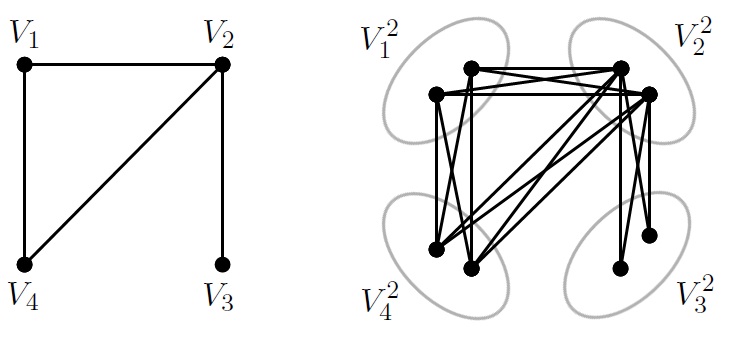}
\caption{The graph $R$ (left) and $R^s$ (right).}\label{fig:R2}
\end{figure} 
The Key Lemma will be an important tool as it allows us to conclude that if $R$ is the regularity graph of $G$ and we can find a structure in the graph $R^s$ then we are also able to embed it in the graph $G$.

\begin{lemma}[Key Lemma] \label{keylem}
Let $d\in (0,1]$, $\Delta\geq 1$. Then there exists an $\varepsilon_0>0$ such that, given graphs $G$ and $H$, with $\Delta(H)\leq \Delta$, and $s \in \mathbb{N}$, if $R$ is a regularity graph of $G$ with parameters $\varepsilon \leq \varepsilon_0$ and $d$ and each vertex of $R$ is a cluster of size $m\geq 2s/d^\Delta$ in $G$, then
$$H \subseteq R^s \Rightarrow H \subseteq G.$$
\end{lemma}

\begin{proof}
Choose $0<\varepsilon_0<d$ satisfying
\begin{equation}\label{epsilon0}
(d-\varepsilon_0)^\Delta-\Delta\varepsilon_0\geq d^\Delta/2.
\end{equation}
We are able to do this since $(d-\varepsilon)^\Delta-\Delta\varepsilon \rightarrow d^\Delta$ as $\varepsilon \rightarrow 0$.

Suppose that we have a graph $G$ which admits an $\varepsilon$-regular	partition, with parameters $\varepsilon\leq \varepsilon_0$, $m>2s/d^\Delta$ and $d$, into an exceptional set $V_0$ and clusters $\{V_1, \ldots V_k\}$  satisfying those properties in \Cref{degreeform}. Let $R$ be its regularity graph.  Suppose that $H \subseteq R^s$, with $$V(H)=\{u_1,u_2,\ldots , u_h\}.$$
Each of the vertices $u_i$ of $H$ is contained in one of the sets $V^s_j$ of $R^s$. This defines a mapping $\sigma$, where $\sigma(u_i)=j$ if $u_i \in V^s_j$. Our aim is to embed $H$ in $G$ by defining a mapping which takes each $u_i$ to a distinct $v_i$ in $V_{\sigma(u_i)}$ such that the edge $v_iv_j \in E(G)$ if $u_iu_j \in E(H)$. We will select these vertices $v_i$ one at a time, starting with $v_1$.

For each $1\leq i\leq h$ let \begin{itemize}
\item $Y^0(u_i)=V_{\sigma(u_i)}$
\item $Y^\ell(u_i)$ be the set of candidates for $v_i$ at the $\ell^{th}$ step, where $1\leq \ell \leq i$.
\end{itemize}
At the $j^{th}$ step, we select the vertex $v_j$, so we have that $Y^j(u_j)=\{v_j\}$. For each $i>j$, if $u_iu_j \in E(H)$, then we remove any vertices from $Y^{j-1}(u_i)$ that are not adjacent to $v_j$, that is, $$Y^j(u_i)=Y^{j-1}(u_i) \cap N_G(v_i).$$

We want to select each vertex $v_j$ so that, for all $i>j$ with $u_iu_j \in E(H)$, the sets $Y^j(u_i)$ are not too small so as to ensure that we can find a copy of $H$ in $G$. For each such $u_i$ we recall that the graph $(V_{\sigma(u_j)}, V_{\sigma(u_i)})$ is $\varepsilon$-regular and so, by \Cref{prop:neighbours}, all but at most $\varepsilon m$ vertices in $Y^{j-1}(u_j) \subseteq V_{\sigma(u_j)}$ have at least $(d-\varepsilon)|Y^{j-1}(u_i)|$ neighbours in $Y^{j-1}(u_i)$, provided that $|Y^{j-1}(u_i)|\geq \varepsilon m$. We must consider at most $\Delta$ neighbours of $u_j$ and so we find that, by avoiding at most $\Delta\varepsilon m$ vertices in $Y^{j-1}(u_j)$, we can ensure that
\begin{equation}\label{size}
|Y^j(u_i)| \geq (d-\varepsilon)|Y^{j-1}(u_i)|, \; \forall i>j.
\end{equation}
Since at most $s$ vertices of $H$ can lie in each set $V^s_{\sigma(u_j)}$ of $R^s$, as long as we have that $$|Y^j(u_i)|\geq s+\Delta\varepsilon m$$ we will be able to find a vertex $v_j$ which satisfies \eqref{size}.

Now, for each $i>j$ we know that
$$|Y^j(u_i)|\geq (d-\varepsilon)^\Delta m \geq (d-\varepsilon_0)^\Delta m$$
by repeatedly applying \eqref{size}, since we delete vertices from the set $Y^{\ell}(u_i)$ only when $u_iu_\ell \in E(H)$ and this is the case for at most $\Delta$ vertices $u_\ell$ with $\ell \leq j$. By our choice of $\varepsilon_0$ in \eqref{epsilon0} we have that $(d-\varepsilon_0)^\Delta\geq d^\Delta/2+\Delta\varepsilon$. Then, recalling that $m\geq 2s/d^\Delta$, we obtain that $$|Y^j(u_i)|\geq (d-\varepsilon)^\Delta m \geq (d^\Delta/2+\Delta\varepsilon)m \geq s+\Delta\varepsilon m.$$ So we can choose suitable, distinct vertices for each $u_i$. Therefore, we are able to embed $H$ in $G$. 
\end{proof}

It is worth noting that if the reduced graph of $G$ with parameters $\varepsilon$, $d$ satisfies the conditions of the lemma, then so does the regularity graph with the same parameters. In particular, we may also apply the lemma if we know that $H$ is a subgraph of the graph $R^s$, where $R$ denotes the reduced graph, to conclude that $H$ is a subgraph of $G$.

Sometimes we might require a stronger result when we wish to embed a structure in a graph, in this case we will apply Koml\'os, S\'ark\"ozy and Szemer\'edi's Blow-up Lemma,  \cite{ksosz}.  If we compare this lemma to the Key Lemma (\Cref{keylem}), we find that the Blow-up lemma is actually much more powerful than the Key Lemma. Whilst the Key Lemma allows us to embed a graph $H$ whose order is small relative to $G$, the Blow-up Lemma will let us embed any spanning subgraph $H$ of $G$ with bounded maximum degree. Informally, the Blow-up Lemma tells us that superregular graphs behave like complete bipartite graphs if we want to embed a bipartite subgraph of bounded maximum degree. The proof of \Cref{c6packing} will use a special case of the Blow-up Lemma for bipartite graphs.

\begin{lemma}[Blow-up Lemma (bipartite form), Koml\'os, S\'ark\"ozy and Szemer\'edi, \cite{ksosz}]\label{blowup}
Given $d>0$ and $\Delta \in \mathbb{N}$, there is a positive constant $\varepsilon_0 = \varepsilon_0(d,\Delta)$ such that the following holds for every $\varepsilon < \varepsilon_0$. Given $m \in \mathbb{N}$, let $G^*$ be an $(\varepsilon,d)$-superregular bipartite graph with vertex classes of size m. Then $G^*$ contains a copy of every subgraph H of $K_{m,m}$ with $\Delta(H) \leq \Delta$. 
\end{lemma}

\noindent In \Cref{chap:arbitrary}, we will require the following, more general, $r$-partite version of the lemma.

\begin{lemma}[Blow-up Lemma (r-partite form), Koml\'os, S\'ark\"ozy and Szemer\'edi, \cite{ksosz}]\label{rpartblowup}
Suppose that $F$ is a graph on $[k]$, let $d>0$ and let $\Delta$ be a positive integer. Then there exists a positive constant $\varepsilon_0 = \varepsilon_0(d,\Delta,k)$ such that the following holds for all positive integers $\ell_1, \ldots, \ell_k$ and all $0< \varepsilon\leq \varepsilon_0$.

Let $F'$ be the graph obtained from $F$ by replacing each vertex $i \in V(F)$ by a set $V_i$ of $\ell_i$ vertices and adding all $V_i$-$V_j$ edges whenever $ij \in E(F)$. Let $G'$ be a spanning subgraph of $F'$ such that for every edge $ij \in E(F)$, the graph $(V_i,V_j)_{G'}$ is $(\varepsilon, d)$-superregular.  Then $G'$ contains a copy of $H$ for every $H \subseteq F'$ with $\Delta(H) \leq \Delta$ such that, for each vertex $v \in V(H)$, if $v \in V_i$ in $F'$ then $v$ is also mapped to $V_i$ by the copy of $H$ in $G'$. 
\end{lemma}

The three applications of the Regularity Lemma which we will consider in this section are: a proof of the Erd\H{o}s-Stone theorem, a result in Ramsey theory and a very specific use of the lemma to find a perfect $C_6$-packing in a graph. In the final application, we will have to confront the problem, mentioned earlier, of incorporating the exceptional vertices since a perfect $C_6$-packing is a spanning subgraph of $G$.

\section{The Erd\H{o}s-Stone Theorem}\label{subsec:erdos-stone}

Given a graph $H$, a natural question to ask is how many edges can a graph $G$ on $n$ vertices have without containing $H$ as a subgraph. An important corollary of the Erd\H{o}s-Stone Theorem, \Cref{coroll:e-s} stated later in this section, will help us to go some way towards answering this question.
\begin{definition}
Let $H$ be a graph and $n\in \mathbb{N}$. Then $$ex(n,H)=\max\{e(G):G \text{ is a graph on }n\text{ vertices and } H \nsubseteq G\}.$$ 
\end{definition}
Another way to think about this is that if $G$ is any graph on $n$ vertices with more than $ex(n,H)$ edges then we know that $H$ must be a subgraph of $G$. If $G$ is a graph on $n$ vertices, $H\nsubseteq G$ and $e(G)=ex(n,H)$ then we say that $G$ is \emph{extremal}.

An important graph is the Tur\'an graph, $T_{r-1}(n)$, where $r,n$ are positive integers and $r\geq 2$. This graph is formed by partitioning $n$ vertices into $r-1$ sets, or vertex classes, which have size as equal as possible, differing by at most $1$. So we have that the sets have size either $\lfloor\frac{n}{r-1}\rfloor$ or $\lceil\frac{n}{r-1}\rceil$. We add all possible edges between these sets. We illustrate this for the graph $T_5(9)$ in \Cref{fig:t59}.
\begin{figure}[h]
\centering
\includegraphics[scale=0.45]{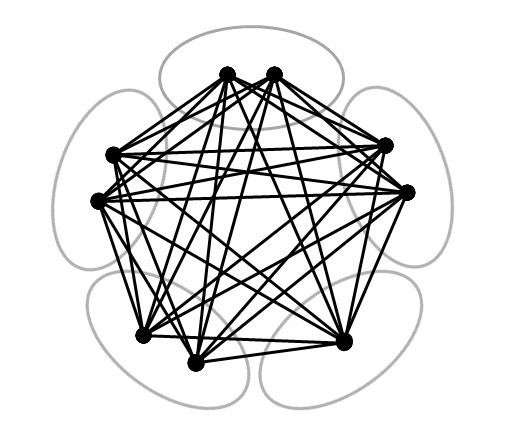}
\caption{The graph $T_5(9)$.}\label{fig:t59}
\end{figure}
Some of the sets may be empty and if $n\leq r-1$ then we simply have that $T_{r-1}(n)=K_n$.
We see that this graph cannot possibly contain a copy of $K_r$ as a subgraph. Suppose that it did. Then two vertices of the $K_r$ subgraph would have to lie in the same vertex class but these sets are independent.

We write $t_{r-1}(n)$ for the number of edges of $T_{r-1}(n)$. If we write $a$ for the number of vertex classes of size $\lceil\frac{n}{r-1}\rceil$ then we find that
$$t_{r-1}(n) = \frac{1}{2}\left( \frac{r-2}{r-1}n^2-\frac{(r-1)-a}{r-1}a \right)$$
which is maximised when $a=r-1$, that is, when $r-1$ divides $n$. So we see that
\begin{equation}\label{eq:turan}
t_{r-1}(n) \leq \left(\frac{r-2}{r-1}\right)\frac{n^2}{2}.
\end{equation}

The following theorem states that the Tur\'an graph, $T_{r-1}(n)$, contains the maximum number of edges without having a $K_r$ subgraph, that is, $t_{r-1}(n)=ex(n,K_r)$. Further, if $G$ is any graph on $n$ vertices with $ex(n,K_r)$ edges and $K_r \nsubseteq G$ we have that $G=T_{r-1}(n)$, so the Tur\'an graph is the unique extremal graph. 

\begin{theorem}[Tur\'an, 1941]\label{turan}
Let $r,n$ be integers, $r>1$. Suppose $G$ is a graph on $n$ vertices which does not contain $K_r$ as a subgraph. If $e(G)=ex(n,K_r)$, then $G=T_{r-1}(n)$.
\end{theorem}

The following proposition gives us that the value $t_{r-1}(n)\binom{n}{2}^{-1}$ converges to

\begin{proposition}\label{turanlimit}
$$\lim_{n\rightarrow\infty} t_{r-1}(n) \binom{n}{2}^{-1} = \frac{r-2}{r-1}.$$
\end{proposition}

The graph $K^s_r$ is the complete $r$-partite graph where every vertex class has $s$ vertices. By requesting that $G$ has only $\gamma n^2$ more edges than the Tur\'an graph $T_{r-1}(n)$, for given $\gamma$, $r$ and $s$ and sufficiently large $n$, the Erd\H{o}s-Stone Theorem states that we can guarantee, not only that $K_r$ is contained in $G$ as a subgraph, but something even stronger: $G$ contains a copy of the graph $K^s_r$. We will use the Regularity Lemma together with the Key Lemma to prove this theorem.

\begin{theorem}[Erd\H{o}s and Stone, 1946]\label{erdos-stone}
Suppose that $r \geq 2$ and $s\geq 1$ are integers and let $\gamma >0$, then there exists an integer $n_0$ such that every graph with $n\geq n_0$ vertices and at least $t_{r-1}(n)+\gamma n^2$ edges contains $K^s_r$ as a subgraph.
\end{theorem}

\begin{proof}
Suppose that $r\geq 2$, $s \geq 1$ and $\gamma >0$ are given and let $G$ be a graph on $n$ vertices with $$e(G) \geq t_{r-1}(n)+\gamma n^2.$$ We see that we must have $\gamma<1$ for such a graph to exist. 

We apply \Cref{keylem} with $d=\gamma$ and $\Delta=\Delta(K^s_r)$ to obtain an $\varepsilon_0>0$ and (since the result holds for all $\varepsilon \leq \varepsilon_0$) we may assume that $\varepsilon_0 <\gamma/4$. Choose a positive constant $\varepsilon\leq \varepsilon_0$ and let $n_0$ be a positive integer satisfying
$$1/n_0 \ll \varepsilon, 1/r, 1/s.$$

Suppose that $G$ is a graph on $n \geq n_0$ vertices and apply the degree form of the Regularity Lemma, \Cref{degreeform}, with the parameters $\varepsilon$, $d = \gamma$ and $k_0\coloneqq \lceil 1/\gamma\rceil$. We obtain clusters $V_1, \ldots V_k$ with $|V_1|=\ldots=|V_k| \eqqcolon m$, an exceptional set $V_0$, a pure graph $G'$ and a reduced graph $R$. We check that $$m=\frac{n-|V_0|}{k} \geq \frac{n_0(1-\varepsilon)}{k}\geq \frac{2s}{\gamma^\Delta}.$$

We will proceed to show that $K_r \subseteq R$ implying that $K^s_r \subseteq R^s$. Then we will be able to apply \Cref{keylem} to show that $K^s_r \subseteq G$. In order to do this, we will estimate the number of edges in $R$. Recall that we have an edge in $G'$ between a pair of clusters only if they are $\varepsilon$-regular with density greater than $\gamma$. These edges in $G'$ all contribute to $e(R)$. We must remember to subtract from the edges in $G'$ any edges which have an endvertex in $V_0$ since these do not contribute to $e(R)$. Also, each of the edges in $R$ can correspond to at most $m^2$ such edges in $G'$. We recall that $d_{G'}(v) > d_G(v) - (\gamma + \varepsilon)n$ for all vertices $v \in V(G)$ and so we see that
\begin{equation*}
\begin{split}
e(R) &\geq \frac{1}{m^2}\left(\frac{1}{2}\sum_{v\in V(G)}(d_{G'}(v)-|V_0|)-|V_0|n\right)\\
& > \frac{1}{2m^2}\left(\sum_{v\in V(G)}(d_{G}(v)-(\gamma+\varepsilon)n-\varepsilon n)-2\varepsilon	n^2\right)\\
& = \frac{1}{2m^2}\left(2e(G) - \gamma n^2- 4\varepsilon n^2\right)\\
& \geq \frac{k^2}{2}\left(\frac{2t_{r-1}(n)}{n^2}+2\gamma - \gamma- 4\varepsilon \right)\\
& = \frac{k^2}{2}\left( t_{r-1}(n)\binom{n}{2}^{-1}\frac{n-1}{n}+\gamma- 4\varepsilon\right).
\end{split}
\end{equation*}

Now we know, by our choice of $\varepsilon$, that $\gamma-4\varepsilon>0$. So we can apply \Cref{turanlimit} and \eqref{eq:turan} to see that, for sufficiently large $n$, we have 
$$e(R) >\frac{k^2}{2}\left(\frac{r-2}{r-1}\right) \geq t_{r-1}(k).$$
We conclude that $K_r \subseteq R$ by \Cref{turan} and hence $K^s_r \subseteq R^s$. Therefore, we can apply \Cref{keylem} to see that $K^s_r \subseteq G$.
\end{proof}

We can now return to our original question of finding a copy of any graph $H$ in our graph $G$. We must first introduce the concept of a vertex colouring as well as the chromatic number of a graph.

\begin{definition}
A \emph{vertex colouring} of a graph $G$ assigns a colour to each vertex in such a way that no pair of adjacent vertices receive the same colour. We call a vertex colouring which uses $k$ colours a \emph{$k$-colouring}.
The \emph{chromatic number}, $\chi(G)$, is the smallest $k$ such that $G$ has a $k$-colouring.
\end{definition}
It is easy to find an upper bound for the chromatic number of a graph $G$ by colouring the vertices of $G$ greedily. Order the vertices of $G$ arbitrarily as $v_1,v_2, \ldots, v_n$. Assign to each vertex in turn a colour that has not already been used amongst its neighbours of lower index. Since each vertex has at most $\Delta(G)$ neighbours, it will always be able to do this using at most $\Delta(G)+1$ colours. Therefore,
$$\chi(G) \leq \Delta(G)+1.$$

The chromatic number is central to an interesting corollary of the Erd\H{o}s-Stone theorem. This corollary determines, asymptotically, for any non-bipartite graph $H$, the number of edges required to force a copy of $H$ in $G$.

\begin{corollary}\label{coroll:e-s}
Let $H$ be a graph with $\chi(H) \geq 2$. Then
$$\lim_{n \rightarrow \infty} ex(n,H)\binom{n}{2}^{-1}= \frac{\chi(H)-2}{\chi(H)-1}.$$
\end{corollary}

\begin{proof}
Let $\varepsilon>0$ and define $r\coloneqq \chi(H)$, $s \coloneqq |H|$.

We have that $H \nsubseteq T_{r-1}(n)$ since $\chi(H)=r > \chi(T_{r-1}(n))$. This gives that
$$ex(n,H) \geq t_{r-1}(n),$$
and thus 
\begin{equation}\label{eq:cor1}
\liminf_n ex(n,H) \binom{n}{2}^{-1} \geq \lim_{n \rightarrow \infty} t_{r-1}(n) \binom{n}{2}^{-1}=\frac{r-2}{r-1}
\end{equation}
by \Cref{turanlimit}.

By \Cref{erdos-stone}, there exists an $n_0$ such that every graph on $n \geq n_0$ vertices with $$e(G) \geq t_{r-1}(n) + \varepsilon n^2$$ has $K^s_r$ as a subgraph. By the definitions of $r$ and $s$, we observe that $$H \subseteq K^s_r \subseteq G.$$
Hence, we see that whenever $n \geq n_0$,
$$ex(n,H) < t_{r-1}(n) + \varepsilon n^2.$$ We again apply \Cref{turanlimit} to see that
\begin{equation}\label{eq:cor2}
\limsup_n ex(n,H) \binom{n}{2}^{-1} \leq \lim_{n \rightarrow \infty} (t_{r-1}(n) + \varepsilon n^2) \binom{n}{2}^{-1}=\frac{r-2}{r-1}+\varepsilon.
\end{equation}
Now, this equation holds for all $\varepsilon>0$ and so, together, \eqref{eq:cor1} and \eqref{eq:cor2} give that
$$\lim_{n \rightarrow \infty} ex(n,H)\binom{n}{2}^{-1}= \frac{\chi(H)-2}{\chi(H)-1}.$$
\end{proof}

This corollary means that for any non-bipartite graph $H$ and any $\varepsilon>0$ there exists an integer $n_0$ such that if $G$ is a graph on $n \geq n_0$ vertices and $$e(G) \geq \left( \frac{\chi(H)-2}{\chi(H)-1}+ \varepsilon \right) \binom{n}{2} $$ then $H \subseteq G$.

\section{Ramsey Theory}

Ramsey Theory focusses on finding structure in large graphs. A well known result, which is easily verified, is that in any group of six people there will be three acquaintances or three strangers. More generally, the theory roughly states that whenever we partition a large graph into a small number of subsets, in one of those subsets there will be a large substructure, for example a large complete graph or a large independent set. Ramsey's theorem tells us that, given any sufficiently large graph, we are guaranteed to find a large complete graph or a large independent set.

\begin{theorem}[Ramsey, 1930]\label{thm:ramsey}
For every $k\in \mathbb{N}$, there exists an $n\in \mathbb{N}$ such that every graph on at least $n$ vertices contains $K_k$ or $\overline{K_k}$ as an induced subgraph. 
\end{theorem}

We might also think of this to mean that if we colour the edges of a $K_n$ with two colours: red and blue, then any such colouring yields a monochromatic $K_k$. We define the Ramsey number as follows.

\begin{definition}
For any $k \in \mathbb{N}$, we define the \emph{Ramsey number}, $R(k)$, to be the smallest positive integer $n$ such that any colouring of the edges of $K_n$ using two colours yields a monochromatic $K_k$.

Given any graph $H$, we define $R(H)$ to be the smallest positive integer such that any colouring of the edges of $K_n$ using two colours yields a monochromatic copy of $H$.
\end{definition}

\begin{proof}[Proof of \Cref{thm:ramsey}]
The result is clear for $k =1$ so let us assume that $k \geq 2$.  Let $n \coloneqq 2^{2k-3}$ and suppose that $G$ is a graph of order at least $n$. Choose $V_1\subseteq V(G)$ be any set of $n$ vertices and let $v_1 \in V_1$ be any vertex. We will define a sequence of sets of vertices $V_1 \supset V_2 \supset \ldots \supset V_{2k-2}$, and vertices $v_i \in V_i$, such that for all $2 \leq i \leq 2k-2$:
\begin{enumerate}[$(i)$]
\item $|V_i| =|V_{i-1}|/2^{i-1}$;
\item $V_i \not\ni v_{i-1}$;
\item $V_i \cap N(v_{i-1})= V_i$ or $\emptyset$.
\end{enumerate} 

Let $1<j \leq 2k-2$ and suppose that we have already chosen sets $V_i$ and vertices $v_i$ for $1 \leq i \leq j-1$ satisfying $(i)$--$(iii)$. We note that $|V_{j-1} \setminus \{v_{j-1}\}| = n/2^{j-2}-1 >0$ is odd, so we can find a subset $V_j$ which satisfies $(i)$--$(iii)$. Choose any $v_j \in V_j$.

Now, amongst the $2k-3$ vertices $v_1, v_2, \ldots, v_{2k-3}$, we can find a set of $k-1$ vertices, $V$, such that either: $N(v_{i-1}) \cap V_i = V_i$ for all $v_{i-1} \in V$ or $N(v_{i-1}) \cap V_i = \emptyset$ for all $v_{i-1} \in V$. In the first case, the vertices $V \cup \{v_{2k-2}\}$ induce a $K_k$ in $G$ and in the second the vertices $V \cup \{v_{2k-2}\}$ induce a $\overline{K_k}$.
\end{proof}

Ramsey numbers are very difficult to calculate, in general, and very few are known. We have shown, in the proof of \Cref{thm:ramsey}, that $R(k) \leq 2^{2k-3}$ for all $k \geq 2$, giving us an exponential bound on the Ramsey number for complete graphs. We will now show that, by considering only the Ramsey numbers of graphs $H$ of bounded maximum degree we can greatly improve on this bound. In fact, we are able to obtain a bound which is linear in $|H|$.

\begin{theorem}[Chv\'atal, R\"odl, Szemer\'edi and Trotter, 1983]\label{ramseyresult}
Suppose that $\Delta$ is a positive integer. Then there exists a constant $c$ such that
$$R(H) \leq c|H|$$ for every graph $H$ with $\Delta(H) \leq \Delta$.
\end{theorem}

\begin{proof}
Apply \Cref{keylem} with inputs $d=1/2$ and $\Delta$ to obtain $\varepsilon_0$, as in the statement of the lemma. Let $k_0 = R(\Delta+1)$ and choose a positive constant $\varepsilon \leq \varepsilon_0$ satisfying $\varepsilon \ll 1/k_0.$ Let $c$ be a positive integer satisfying $1/c\ll \varepsilon,1/\Delta$.

Now, let $H$ be a graph with $\Delta(H) \leq \Delta$ and let $|H|\eqqcolon h$. Suppose that $G$ is a graph on $n \geq ch$ vertices. Apply the Regularity Lemma, \Cref{reglemma}, with the parameters $\varepsilon$ and $k_0$ to obtain an $\varepsilon$-regular partition into clusters $V_1, \ldots, V_k$, with $|V_1| = \ldots = |V_k| \eqqcolon m$, and exceptional set $V_0$. We aim to prove that $G$ has $H$ or $\overline{H}$ as a subgraph. Equivalently, we will show that $H\subseteq G$ or $H \subseteq \overline{G}$.

We check that
$$m=\frac{n-|V_0|}{k} \geq \frac{ch(1-\varepsilon)}{k}\geq \frac{2h}{d^\Delta},$$ so we will be able to use \Cref{keylem}.

Let $R$ be the graph with vertices $\{V_1, \ldots V_k\}$ and an edge between two vertices if the corresponding pair of clusters is $\varepsilon$-regular. We have that $|R|=k$ and there are at most $\varepsilon k^2$ pairs which are not $\varepsilon$-regular so
\begin{eqnarray*}
e(R) &\geq &\frac{k(k-1)}{2}-\varepsilon k^2 = \frac{k^2}{2} \left(1-\frac{1}{k}-2\varepsilon \right)\\
& \geq &\frac{k^2}{2} \left(1-\frac{1}{k_0}-\frac{1}{k_0(k_0-1)} \right) = \frac{k^2}{2}\frac{k_0-2}{k_0-1}\\
& \stackrel{\eqref{eq:turan}}{\geq}& t_{k_0-1}(k).
\end{eqnarray*}
Then we have that $K=K_{k_0} \subseteq R$, by \Cref{turan}.

Let us now colour the edges of $R$ as follows:
\begin{itemize}
\item Colour the edge $V_iV_j$ red if $d_G(V_i,V_j)\geq 1/2$;
\item Colour the edge $V_iV_j$ blue if $d_G(V_i,V_j)< 1/2$.
\end{itemize}

We define graphs $R'$ and $R''$ both having vertex sets $V(R)$. The graph $R'$ has all red edges and the graph $R''$ has all blue edges. We see that $R'$ is in fact the regularity graph corresponding to this partition of $G$ with parameters $\varepsilon$ and $d=1/2$. By recalling \Cref{prop:complement}, we also see that the graph $R''$ is the reduced graph of $\overline{G}$ with the same parameters.

Recall that we defined $k_0=R(\Delta+1)$. So $K$ must contain a red $K_{\Delta+1}$ or a blue $K_{\Delta+1}$. Then, since $\chi(H)\leq \Delta(H)+1 \leq \Delta+1$, we have that $H \subseteq K_{\Delta+1}^h$ and so $H \subseteq (R')^h$ or $H \subseteq (R'')^h$. We can apply \Cref{keylem} to see that, in the first case, $H \subseteq G$ and, in the second, $H \subseteq \overline{G}$. Therefore, $R(H) \leq c|H|$.
\end{proof}

\section{Finding a Perfect $C_6$-Packing}

Let us now consider a particular example, where the Regularity Lemma is used to find a spanning subgraph of a graph $G$ consisting entirely of disjoint copies cycles of length $6$. Such a subgraph is called a perfect $C_6$-packing and we formally define an $F$-packing below.

\begin{definition}
Given two graphs $F$ and $G$, an \emph{$F$-packing} in $G$ is a collection of vertex-disjoint copies of $F$ in $G$. An $F$-packing is said to be \emph{perfect} if it covers all of the vertices of $G$.
\end{definition}

We will prove the following theorem.

\begin{theorem}\label{c6packing}
For every $0<\eta<1/2$ there exists an integer $n_0$ such that every graph $G$ with order $n\geq n_0$ divisible by $6$ and $\delta(G) \geq n(1/2+\eta)$ contains a perfect $C_6$-packing.
\end{theorem}

Note that the bound on the minimum degree in \Cref{c6packing} is close to best possible. Indeed, suppose that $n$ is divisible by $6$  and consider the graph $G$ on $n$ vertices consisting of disjoint copies of $K_{n/2+1}$ and $K_{n/2-1}$. We have that $\delta(G) = n/2-2$. In order to contain a perfect $C_6$-packing, the two components must have perfect $C_6$-packings but this is not possible since their orders are not divisible by $6$.

\begin{proof}
Choose positive constants $\varepsilon$ and $d$ and $n_0 \in \mathbb{N}$ such that
$$1/n_0 \ll \varepsilon \ll d \ll \eta < 1/2.$$ 
Let $k_0 \coloneqq 1/ \varepsilon$ and let $G$ be a graph on $n \geq n_0$ vertices. The first step is to apply the degree form of the Regularity Lemma (\Cref{degreeform}) with  parameters $\varepsilon, d$ and $k_0$ to the graph $G$. We obtain: clusters $V_1, V_2, \ldots, V_k$; an exceptional set, $V_0$; a pure graph, $G'$ and a reduced graph, $R$.

\noindent Note that $|R|=k$. We are given that $\delta(G) \geq n(1/2+\eta)$ and so we may apply \Cref{mindeg} to see that
\begin{itemize}
\item[(a)]$\delta(R) \geq (1/2+\eta-2d)k \geq (1+\eta)k/2 > k/2.$
\end{itemize}
Then Dirac's theorem implies that $R$ contains a Hamilton path, $P$, and we may assume that $P=V_1V_2 \ldots V_k$ by relabelling if necessary.

We use that $(V_i,V_{i+1})_{G'}$ is $\varepsilon$-regular and has density $>d$ for each $1\leq i\leq k-1$ and apply \Cref{superreg} in order to obtain subclusters $V'_i \subseteq V_i$ of size $m' \coloneqq m-2\varepsilon m$ such that
\begin{enumerate}
\item[(b)]$(V'_i,V'_{i+1})_{G'}$ is $(2\varepsilon, d/2)$-superregular for every edge $V_iV_{i+1} \in E(P)$.
\end{enumerate}
For each $i=1,\ldots,k$ we add the vertices in $V_i\setminus V'_i$ to the exceptional set $V_0$, in total we add $k(2\varepsilon m) \leq 2\varepsilon n$ vertices to $V_0$. We also add the vertices in $V_k$ to the exceptional set if $k$ is odd, adding at most $m'\leq n/k\leq n/k_0=\varepsilon n$ vertices. We continue to refer to the reduced graph as R, its number of vertices as $k$ and to call the exceptional set $V_0$ and we now have
$$|V_0| \leq 4 \varepsilon n.$$
Now that $k$ is even, we can find a perfect matching $M=\{V_1V_2,V_3V_4,...V_{k-1}V_k\}$ in $P$.

We will now set aside some vertices from the graph - these will be put back at a later stage. Consider any odd $i$. By (a), we know that there exists a vertex $V_j \in (N_R(V_i) \cap N_R(V_{i+1}))$. Recall that $(V_i,V_j)_{G'}$ and $(V_{i+1}, V_j)_{G'}$ are $\varepsilon$-regular. So by \Cref{prop:subsets}, $(V'_i,V'_j)_{G'}$ and $(V'_{i+1}, V'_j)_{G'}$ are $2\varepsilon$-regular and have density at least $d-\varepsilon\geq d/2$. Then, by \Cref{prop:neighbours}, we have that there are at least $(1-4\varepsilon)m'$ vertices in $V'_j$ having at least $(d/2-2\varepsilon)m'$ neighbours in both $V'_i$ and $V'_{i+1}$. Let $X_i \in V'_j$ be a set of $11$ of these vertices. For each odd $i$ we choose a set $X_i$ and we choose these in such a way that the $X_i$s are disjoint. We are able to do this since we have $n$ large enough such that $11k/2 \leq (1-4\varepsilon)m'$.

Let $X \coloneqq X_1 \cup X_3 \cup \cdots \cup X_{k-1}$. We remove the vertices in $X$ from their  clusters but do not add them to $V_0$. We have $|X| <11k$ and so if we remove at most $|X|(k-1) < 11k^2 \leq \varepsilon n$ additional vertices (and add these to the exceptional set) we may assume that the resulting subclusters $V''_i \subseteq V'_i$ all have the same size, which we shall define to be $m''$. The new exceptional set has size $$|V_0| \leq 5\varepsilon n.$$

We want to assign each element $x \in V_0$ to a cluster and in order to this we will define an odd index $i$ to be \emph{good} for $x$ if $|N_{G'}(x) \cap V''_i| \geq \eta^2 m''$ and $|N_{G'}(x) \cap V''_i| \geq \eta^2 m''$. Denote the number of good indices by $g_x$. We find that the number of neighbours of $x$ in $G$ belonging to clusters $V''_i$ is $$d_{G'}(x) -|V_0|-|X| \leq 2g_xm''+(\eta^2+1)(k/2-g_x)m''$$ since if $i$ is good all $2m''$ vertices in $V_i$ and $V_{i+1}$ may be neighbours and if $i$ is not good then $x$ has fewer than $\eta^2m''$ neighbours in at least one of $V_i, V_{i+1}$.
Now, $|V_0|+|X| \leq 6\varepsilon n \leq \eta n/2$ and $\delta(G) \geq n(1/2+\eta)$ so we have that $(1+\eta)n/2 \leq d_{G'}(x) -|V_0|-|X|$. We also note that $(k/2-g_x)m''\leq km''/2 \leq n/2$. So, combining these observations, we see that $$(1+\eta)n/2 \leq 2g_xm''+(\eta^2+1)n/2$$ and hence $$g_x \geq \frac{\eta n}{4m''}(1-\eta) \geq \frac{\eta k}{4}\frac{1}{2} = \frac{\eta|M|}{4}.$$
We also have that $$\frac{|V_0|}{\sqrt\varepsilon m''} \leq \frac{5\sqrt{\varepsilon}n}{m''} \leq \frac{\eta|M|}{4}.$$ This means that we can choose a good odd index $i$ for each vertex in the exceptional set so that no index is assigned more than $\sqrt\varepsilon m''$ vertices.

For each odd index $i$, consider those vertices which have been assigned to $i$. Distribute these as evenly as possible between the sets $V''_i$ and $V''_{i+1}$ forming new sets $V^*_i$ and $V^*_{i+1}$ which differ in size by at most $1$.

We claim that the graph $(V^*_i,V^*_{i+1})_{G'}$ is $(4\sqrt[4]\varepsilon, d/18)$-superregular for each odd $i$. This follows from (b) and applying first \Cref{removing} to see that the graph $(V''_i, V''_{i+1})_{G'}$ is $(\sqrt{2\varepsilon}, d/2-\sqrt{2\varepsilon})$-superregular. We then apply \Cref{adding} to see that after adding the exceptional vertices assigned to each cluster, at most $\sqrt\varepsilon m''$, 
\begin{enumerate}
\item[(c)]$(V^*_i,V^*_{i+1})_{G'}$ is $(4\sqrt[4]{\varepsilon}, d/18)$-superregular for each odd $i$.
\end{enumerate}

We must now show that we can make $|V^*_i \cup V^*_{i+1} \cup X_i|$ divisible by $6$ for every odd $i$. First let's consider $i=1$. Suppose that $|V^*_1 \cup V^*_2 \cup X_1| \equiv a$ mod $6$ for some $0 \leq a < 6$. We can apply \Cref{prop:supersets} and \Cref{prop:subsets} to see that the graphs $(V^*_2,V^*_3)_{G'}$ and $(V^*_3,V^*_4)_{G'}$ are $\sqrt[5]{\varepsilon}\geq 10\sqrt[4]{\varepsilon}$-regular and has density at least $d/4\geq d-7\sqrt[4]{\varepsilon}$ to choose $a$ disjoint copies of $C_6$, each having $1$ vertex in $V^*_2$, $3$ vertices in $V^*_3$ and $2$ vertices in $V^*_4$. First we pick a vertex $x \in V^*_2$ which has at least $(d/4-\sqrt[5]{\varepsilon})|V^*_3|$ neighbours in $V^*_3$ and a vertex $y \in V^*_3$ which has at least $(d/4-\sqrt[5]{\varepsilon})|V^*_4|$ neighbours in $V^*_4$. There are many such vertices we could choose by \Cref{prop:neighbours}. Now, by \Cref{prop:subsets}, $((N_{G'}(x)\setminus \{y\}) \cap V^*_3 , N_{G'}(y) \cap V^*_4)_{G'}$ is $2 \sqrt[5]{\varepsilon}$-regular and has density at least $d/4-\sqrt[5]{\varepsilon}$. So we can choose $2$ vertices $y_1$ and $y_2$ in $N_{G'}(y) \cap V^*_4$ with (distinct) neighbours $z_1$ and $z_2$ respectively, in $N_{G'}(x) \cap V^*_3$. Together, these vertices form a copy of $C_6$. 
\begin{figure}[h]
\centering
\includegraphics[scale=0.4]{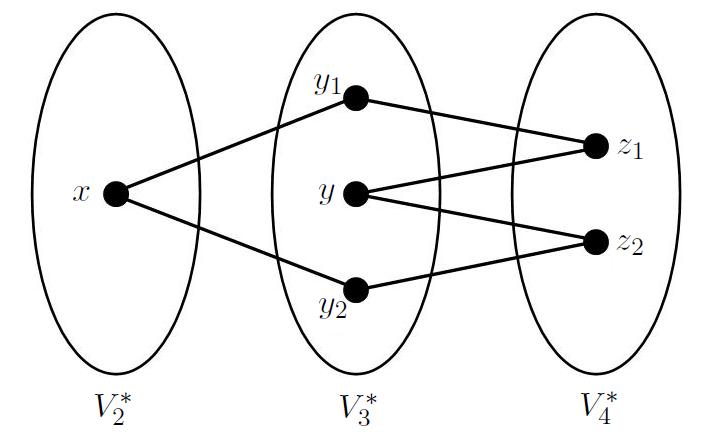}
\caption{A copy of $C_6$.}\label{fig:c6}
\end{figure}
Continue in this way until we have removed $a$ copies of $C_6$ and then $|V^*_1 \cup V^*_2 \cup X_1|$ is divisible by $6$. We are able to do this since we only remove a small number of vertices in each copy of $C_6$ and so we can apply \Cref{prop:subsets} to see that the graph is still regular. We repeat this process for each odd $i$ in turn and since $n$ is divisible by $6$ we can ensure that $|V^*_i \cup V^*_{i+1} \cup X_i|$ is divisible by $6$ for every odd $i$.

Before we removed the copies of $C_6$, $|V^*_i|$ and $|V^*_{i+1}|$ differed by at most $1$. Now they can differ by at most $1+5+5=11$. We return to the sets $X_i$ which we set aside earlier. For each odd $i$, add each $x \in X_i$ to either $V^*_i$ or $V^*_{i+1}$ so that the new sets $V^\diamond_i \supseteq V^*_i$ and $V^\diamond_{i+1} \supseteq V^*_{i+1}$ are equal in size. 
Recall these clusters were formed after removing at most 15 vertices from $V^*_i$ and $V^*_{i+1}$ in copies of $C_6$ (we have removed at most $5$ sets of $3$ vertices from each $V^*_i$ if $i$ is odd and at most $5$ single vertices and then $5$ sets of $2$ vertices if $i$ is even). We have now added the vertices from $X_i$ which were originally chosen so that they had at least $(d/2-2\varepsilon)m'$ neighbours in both $V'_i \supseteq V^*_i$ and $V'_{i+1} \supseteq V^*_{i+1}$. Then, using (c), we can apply \Cref{adding} and \Cref{removing} to see that $(V^\diamond_i,V^\diamond_{i+1})_{G'}$ is $(2 \varepsilon^{16}, d/150)$-superregular. Finally, since $|V^*_i \cup V^*_{i+1} \cup X_i|$ is divisible by $6$, we have that $V^\diamond_i$ and $V^\diamond_{i+1}$ are divisible by $3$ and we can apply \Cref{blowup} to the graph $H^\diamond_i$ to find a perfect $C_6$-packing. We can do this for each odd $i$. Together with the copies of $C_6$ we removed earlier, these $C_6$-packings combine to form a perfect $C_6$-packing in $G$.
\end{proof}

\chapter{Hamilton Cycles}
\label{chap:hamilton}

We now turn our attention to Hamilton cycles. The decision problem of whether a graph contains a Hamilton cycle in NP-complete, so it is unlikely that it is possible to completely characterise those graphs which are Hamiltonian. Instead, we look for sufficient conditions which will guarantee a Hamilton cycle.

One of the most well-known results is Dirac's theorem \cite{dirac} which states that if $G$ is a graph on $n \geq 3$ vertices with $\delta(G) \geq n/2$ then $G$ contains a Hamilton cycle. Dirac's theorem can be strengthened by allowing some vertices in $G$ to have a degree much smaller than $n/2$ and we will look at some \textquoteleft degree sequence' conditions in the next section. Ghouila-Houri proved an analogue of Dirac's theorem for digraphs in \cite{ghouila} and we will consider some other conditions which ensure that a digraph is Hamiltonian.

\section{Degree Sequence Conditions}

We define the \emph{degree sequence} of $G$ to be $d_1, d_2, \ldots, d_n$, where $d_i$ are the degrees of the vertices in $G$ and $d_1\leq d_2\leq \ldots\leq d_n$.  P\'osa's theorem \cite{posa} states that if $d_i \geq i+1$ for all $i<(n-1)/2$ and, if $n$ is odd, $d_{\lceil n/2 \rceil}\geq \lceil n/2 \rceil$, then $G$ contains a Hamilton cycle. Chv\'atal's theorem generalises P\'osa's theorem still further describing those degree sequences which ensure that a graph is Hamiltonian.

\begin{theorem}[Chv\'atal, 1972]\label{chvatal}
Let $G$ be a graph on $n \geq 3$ vertices with degree sequence $d_1\leq d_2\leq \ldots \leq d_n$ satisfying
$$d_i \geq i+1 \text{ or } d_{n-i} \geq n-i$$
for all $i<n/2$. Then $G$ has a Hamilton cycle.
\end{theorem}

\begin{proof}
Suppose that the theorem is not true. Then we can choose a graph $G$ on $n\geq 3$ vertices with degree sequence satisfying the condition of the theorem and the maximum number of edges such that $G$ does not contain a Hamilton cycle. Label the vertices $v_1, v_2, \ldots, v_n$ so that $d_G(v_i)=d_i$ for all $i=1, \ldots, n$.

Let $v_j,v_k \in V(G)$ be non-adjacent vertices with $j<k$ such that $d_j+d_k$ is maximal. Consider the graph $$G' \coloneqq G \cup \{v_jv_k\}.$$ Now $d_{G'}(v_i) \geq d_G(v_i)$ for all $v_i \in V(G)$, so the degree sequence of $G'$ satisfies the condition of the theorem. Since $G$ was edge maximal, we have that $v_jv_k$ lies on a Hamilton cycle $C$ in $G'$. Then $C \setminus \{v_jv_k\}$ is a Hamilton path in $G$. Let us denote this path by $$P=x_1x_2 \ldots x_n$$ where $x_1=v_j$ and $x_n=v_k$. Let
$$S \coloneqq \{ x_{i-1} : x_1x_i \in E(G)\} \text{ and } T \coloneqq N_G(x_n).$$ Observe that $S \cup T \subseteq \{x_1, x_2, \ldots, x_{n-1}\}$, $|S|=d_j$ and $|T|=d_k$.

If there exists $x_{i-1} \in S\cap T$ then $x_1Px_{i-1}x_nPx_ix$ forms a Hamilton cycle in $G$. Hence, the sets $S$ and $T$ are disjoint. Therefore
$$d_j+d_k = |S|+|T| \leq n-1.$$
Recall that $d_j \leq d_k$ and so $d_j < n/2$.

Since $S \cap T = \emptyset$, we know that all vertices in $S$ are not adjacent to $x_n=v_k$. Since we chose $v_j$ to maximise $d_j+d_k$, we know  that $d_G(x_i) \leq d_j$ for all $x_i \in S$ which implies that $d_{d_j} \leq d_j$. Then, by the condition in the theorem, we see that $d_{n-d_j} \geq n-d_j$ which means that the vertices $v_{n-d_j}, \ldots, v_n$ must all have degree at least $n-d_j$. Since this list contains $d_j+1$ vertices and $d(v_j)=d_j$, we know that at least one of these vertices is not adjacent to $v_j$ in $G$, say $u$. Now,
$$d_G(v_j)+d_G(x) \geq d_j + (n-d_j) = n > d_j+d_k$$
which contradicts the choice of $v_j$ and $v_k$. So the assumption that $G$ does not contain a Hamilton cycle was false.
\end{proof}

The condition on the degree sequence in this theorem is best possible, that is, we can always find a graph $G$ with degree sequence $d_1, \ldots, d_n$ and $d_r = r$ and $d_{n-r} = n-r-1$ for some $1 \leq r < n/2$ such that $G$ does not contain a Hamilton cycle. Fix $n$ and $1 \leq r <n/2$, we will define the graph $G$ on $n$ vertices as follows. Label the vertices of $G$ by $v_1, \ldots, v_n$ and join two vertices $v_i$ and $v_j$ if:
\begin{itemize}
\item $i,j \geq r+1$ or
\item $i\leq r$ and $j \geq n-r+1$.
\end{itemize}  
We check that $G$ has $r$ vertices of degree $r$, $n-2r$ vertices of degree $n-r-1 \geq r$ and $r$ vertices of degree $n-1 \geq n-r-1$. So we do indeed have $d_r= r$ and $d_{n-r}=n-r-1$. We illustrate this in \Cref{fig:chvatal} for the case $n=8$, $r=3$.

\begin{figure}[h]
\centering
\includegraphics[scale=0.4]{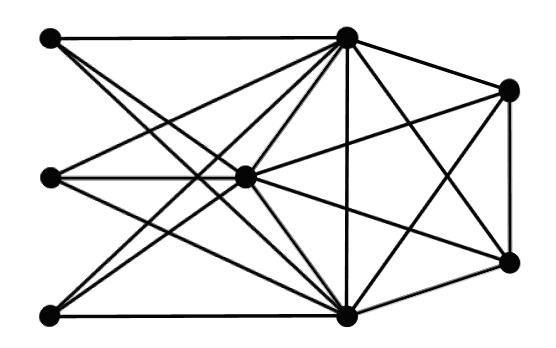}
\caption{The graph $G$ for $n=8$, $r=3$. This graph is the union of a $K_{3,3}$ and a $K_5$. Its degree sequence is $3,3,3,4,4,7,7,7$.}\label{fig:chvatal}
\end{figure}

Now the graph $G$ consists of a $K_{r,r}$ on the vertices $\{v_1, \ldots, v_r, v_{n-r+1} \ldots, v_n\}$ and a $K_r$ on $\{v_{r+1} \ldots, v_n\}$. A Hamilton cycle would have to visit each of the vertices in the $K_{r,r}$ but the only way to do this is with a $C_{2r}$ leaving the rest of the vertices in the graph unvisited. So $G$ does not contain a Hamilton cycle.

Shortly after the proof of Chv\'atal's theorem, \Cref{chvatal}, Nash-Williams conjectured a digraph analogue of the theorem. If $G$ is a digraph on $n$ vertices then we can define its degree sequences. The \emph{outdegree sequence} of $G$ is $d^+_1, d^+_2, \ldots, d^+_n$, where $d^+_i$ are the outdegrees of the vertices in $G$ and $d^+_1\leq d^+_2\leq \ldots\leq d^+_n$. In a similar way, we define the \emph{indegree sequence} $d^-_1, d^-_2, \ldots, d^-_n$ with $d^-_1\leq d^-_2\leq \ldots\leq d^-_n$. Note that $d^+_i$ and $d^-_i$ may not refer to the degrees of the same vertex. 
 
\begin{conjecture}[Nash-Williams, \cite{nash}]\label{nashconjecture}
Suppose that $G$ is a strongly connected digraph on $n \geq 3$ vertices such that 
\begin{enumerate}[(i)]
	\item $d^+_i \geq i+1$ or $d^-_{n-i} \geq n-i$ and
	\item $d^-_i \geq i+1$ or $d^+_{n-i} \geq n-i$
	\end{enumerate}
for all $i<n/2$. Then $G$ contains a Hamilton cycle.
\end{conjecture}

In \Cref{chap:digraphs}, we will prove an approximate version of this conjecture for large digraphs.

\section{Hamilton Cycles in Oriented Graphs}

We define an oriented graph to be a digraph which can be obtained by orienting an undirected simple graph. So an oriented graph does not contain any cycles of length two. In \cite{keevko}, Keevash, K\"uhn and Osthus give a bound on the minimum semidegree which ensures a Hamilton cycle of standard orientation in any sufficiently large oriented graph.

\begin{theorem}[Keevash, K\"uhn and Osthus, \cite{keevko}]\label{thm:stdhamcycle}
There exists $n_0$ such that every oriented graph $G$ on $n \geq n_0$ vertices with $\delta^0(G) \geq (3n-4)/8$ contains a directed Hamilton cycle.
\end{theorem}

This result is actually best possible.

\begin{proposition}\label{prop:bestforham}
For any $n \geq 3$ there is an oriented graph on $n$ vertices with $\delta^0(G) = \lceil (3n-4)/8 \rceil -1$ which does not contain a directed Hamilton cycle.
\end{proposition}

We will prove this proposition using a construction given by H\"aggkvist for the special case where $n=8k-1$ for some $k$. (A proof covering all cases is given in \cite{ckko}.)

\begin{proof}
Suppose $n=4m+3$ for some odd $m$. We will define an oriented graph $G$ on $n$ vertices with $\delta^0(G) = (3n-5)/8$ which has no $1$-factor and hence no Hamilton cycle. We illustrate this graph in \Cref{fig:noantidir}. Let $A$ and $C$ be regular tournaments on $m$ vertices and let $B$ and $D$ be sets of vertices of size $m+2$ and $m+1$ respectively. Then $G$ is the disjoint union of $A$, $B$, $C$ and $D$ together with:
\begin{itemize}
\item all edges from $A$ to $B$, $B$ to $C$, $C$ to $D$ and $D$ to $A$;
\item all edges between $B$ and $D$, oriented to form a bipartite graph which is as regular as possible, so that the indegree and outdegree of each vertex differ by at most one.
\end{itemize} 
We can check that $\delta^0(G) = (m-1)/2+(m+1) = (3n-5)/8$. We will show that this graph does not contain a $1$-factor.

Now, every path connecting two vertices in $B$ must use a vertex from $D$. So any cycle in $G$ will use at least one vertex from $D$ for each vertex it visits in $B$. Since $|B|>|D|$, we see that $G$ cannot contain a $1$-factor. Therefore, $G$ has no Hamilton cycle. 
\end{proof}

\begin{figure}[h]
\centering
\includegraphics[scale=0.4]{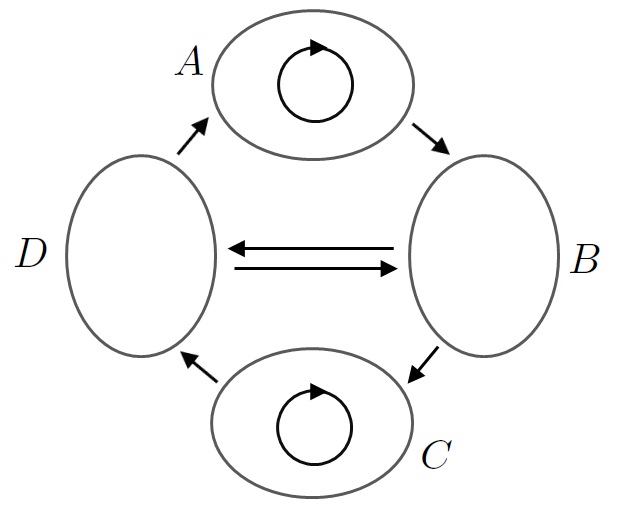}
\caption{The oriented graph constructed in \Cref{prop:bestforham} and \Cref{prop:noantidir}.}\label{fig:noantidir}
\end{figure}


\chapter{Digraphs}
\label{chap:digraphs}

We have seen the ideas of regularity and superregularity and have stated and applied the Regularity Lemma for undirected graphs. From now on, we will consider directed graphs, or digraphs, and many of the definitions and results we have met so far will follow through with little change. We will see an analogue of the Regularity Lemma for digraphs but first we will define a new concept, that of robust outexpansion.

\section{Robust Outexpansion}
Robust outexpansion has formed a key feature in many recent results involving Hamilton cycles. The concept was introduced by K\"uhn, Osthus and Treglown in \cite{kot}.

We say that a graph $G$ is a robust $(\nu, \tau)$-outexpander if, when we consider any subset $S$ of the vertices of $G$ which is neither too small or too large, the set of vertices having at least $\nu n$ inneighbours in $S$ has size at least $|S|+\nu n$. The precise definitions of a robust outexpander and the, weaker, outexpander are given below. We also include here definitions for a robust $(\nu,\tau)$-inexpander and a robust $(\nu,\tau)$-diexpander which we will require in \Cref{chap:arbitrary}.

\begin{definition}
Let $0< \nu \leq \tau <1$. Given any digraph $G$ on $n$ vertices and $S \subseteq V(G)$, the \emph{$\nu$-robust outneighbourhood} $RN_{\nu,G}^+(S)$ of $S$ is the set of all those vertices $x \in V(G)$ which have at least $\nu n$ inneighbours in $S$. We define the \emph{$\nu$-robust inneighbourhood} $RN_{\nu,G}^-(S)$ of $S$ is the set of all those vertices $x \in V(G)$ which have at least $\nu n$ outneighbours in $S$.

$G$ is called a \emph{robust $(\nu,\tau)$-outexpander} if $|RN_{\nu,G}^+(S)| \geq |S|+\nu n$ for all $S \subseteq V(G)$ with $\tau n < |S| < (1- \tau)n$. We define a \emph{robust $(\nu,\tau)$-inexpander} similarly. If $G$ is both a robust $(\nu,\tau)$-outexpander and a robust $(\nu,\tau)$-inexpander, we will say that $G$ is a \emph{robust $(\nu,\tau)$-diexpander}.

$G$ is called a \emph{$(\nu,\tau)$-outexpander} if $|N^+(S)| \geq |S|+\nu n$ for all $S \subseteq V(G)$ with $\tau n < |S| < (1- \tau)n$.
\end{definition}

Such graphs are interesting because they occur frequently, for example, any sufficiently large oriented graph $G$ with $\delta^0(G) \geq (3/8+\alpha) n$ is a robust outexpander. 

\begin{lemma}[\cite{kellyexact78}]\label{lem:3/8}
Let $0<1/n \ll \nu \ll \tau \leq \alpha/2 \leq 1$ and suppose that $G$ is an oriented graph on $n$ vertices with $\delta^0(G) \geq (3/8+\alpha) n$. Then $G$ is a robust $(\nu,\tau)$-outexpander.
\end{lemma}

The following lemma shows that if we have a sufficiently large graph whose degrees sequences satisfy the given conditions then this graph is also robust outexpander. Notice that these degree sequence conditions closely resemble those of Conjecture \ref{nashconjecture}.

\begin{lemma}\label{robdegseq}
Let $n_0$ be a positive integer and $\tau, \eta$ be constants such that
$$1/n_0 \ll \tau \ll \eta <1.$$ Suppose that $G$ is a digraph on $n \geq n_0$ vertices satisfying
	\begin{enumerate}[(i)]
	\item $d^+_i \geq i+ \eta n$ or $d^-_{n-i-\eta n} \geq n-i$ and
	\item $d^-_i \geq i+ \eta n$ or $d^+_{n-i-\eta n} \geq n-i$
	\end{enumerate}
for all $i<n/2$.
Then $\delta^0(G) \geq \eta n$ and $G$ is a robust $(\tau^2,\tau)$-outexpander.
\end{lemma}

\begin{proof}
First we will show that $\delta^0(G) \geq \eta n$. Notice that $\delta^+(G) = d^+_1$, so if $d^+_1 \geq \eta n$ then $\delta^+(G) \geq \eta n$. So we may assume that $d^+_1 < \eta n<1+\eta n$. Then by condition $(i)$, we have that $d^-_{n-1-\eta n} \geq n-1$. This means that $G$ contains at least $\eta n+1$ vertices with indegree at least $n-1$. Now, every vertex in $G$ must send an edge to at least $\eta n$ of these and so $\delta^+(G) \geq \eta n$.

By considering $d^-_1$ and proceeding in a similar fashion, we show that $\delta^-(G) \geq \eta n$. Therefore, $\delta^0(G) \geq \eta n$.

Now suppose that $S \subseteq V(G)$ with $\tau n<|S|<(1-\tau)n$. We consider the following cases:

\vspace{6pt}
\noindent\textbf{Case 1:} $d^+_{|S|-\lfloor\tau n\rfloor} \geq |S|- \lfloor\tau n\rfloor+\eta n \geq |S|+\eta n/2$.

\noindent We know that the degrees of least $\lfloor\tau n\rfloor$ vertices in $S$ appear after $d^+_{|S|-\lfloor\tau n\rfloor}$ in the outdegree sequence of $G$. So we can consider a set $X\subseteq S$ of size $\lfloor\tau n\rfloor$ such that each vertex in $X$ has outdegree at least $|S|+\eta n/2$. Consider the set $Y=\{ y \in V(G): |N^-(y) \cap X| \geq \tau^2n\}\subseteq |RN^+_{\tau^2,G}(S)|$. We see that
$$|X|(|S|+\eta n/2) \leq \sum_{x\in X} d^+(x) \leq |Y||X| + (n-|Y|)\tau^2n \leq |Y||X|+\tau^2n^2.$$
This implies that $|Y| \geq |S|+\eta n/2-\tau^2n^2/|X|\geq |S|+2\tau^2n$ and so
$$|RN^+_{\tau^2,G}(S)| \geq |Y| \geq |S|+2\tau^2n.$$

\vspace{6pt}
\noindent\textbf{Case 2:} $|S| \neq n/2+ \lfloor\tau n\rfloor$ and $d^+_{|S|-\lfloor\tau n \rfloor}<|S| - \lfloor\tau n\rfloor + \eta n$.

\noindent If $|S| > (1- \eta+ \tau^2)n$ then $|G\setminus S| < \eta n- \tau^2n $. We have shown that $\delta^-(G) \geq \eta n$. Then, for all $x \in V(G)$ we have that $|N^-(x) \cap S| \geq \tau^2n$ and so $x \in RN^+_{\tau^2,G}(S)$ giving $RN^+_{\tau^2,G}(S) = V(G)$ and we are done. So we may assume that $|S| \leq (1- \eta+ \tau^2)n$.

If $|S|- \lfloor\tau n \rfloor < n/2$ then by $(i)$ we see that $d^-_{n-|S|+\lfloor\tau n\rfloor - \eta n} \geq n-|S|+\lfloor\tau n\rfloor$ Otherwise, we have $n-(|S|- \lfloor\tau n \rfloor) < n/2$ and, applying $(ii)$, we again see that $d^-_{n-|S|+\lfloor\tau n\rfloor - \eta n}\geq n-|S|+\lfloor\tau n\rfloor$. So 
$$d^-_{n-|S|+\lfloor\tau n\rfloor - \eta n} \geq n-|S|+\lfloor\tau n\rfloor \geq n-|S|+\tau^2n.$$
Then $G$ must contain at least $|S|-\lfloor\tau n\rfloor+\eta n \geq |S| + \eta n/2$ vertices each having indegree at least $n-|S|+\tau^2 n$, let $U$ denote the set of all such vertices. Observe that any vertex $x \in U$ has at least $\tau^2n$ inneighbours in $S$ and so $U \subseteq RN^+_{\tau^2,G}(S)$. So $$|RN^+_{\tau^2,G}(S)| \geq |U| \geq |S| +\eta n/2 \geq |S|+2\tau^2n.$$

\vspace{6pt}
\noindent\textbf{Case 3:} $|S| = n/2+ \lfloor\tau n\rfloor$

\noindent Consider a subset $S' \subset S$ with $|S'|=|S|-1$. Then, by the previous arguments, we see that $|RN^+_{\tau^2,G}(S')| \geq |S'|+2\tau^2n = |S| -1 +2\tau^2 \geq |S| +\tau^2$. Hence, $$|RN^+_{\tau^2,G}(S)| \geq |RN^+_{\tau^2,G}(S')| \geq |S|+\tau^2n.$$

Together, these cases show that $|RN^+_{\tau^2,G}(S)| \geq |S|+\tau^2n$ for all $S \subseteq V(G)$ with $\tau n<|S|<(1-\tau)n$. Therefore, $G$ is a robust $(\tau^2,\tau)$-outexpander.
\end{proof}

The property of robust outexpansion is resilient, by this we mean that it can not be destroyed by removing just a small number of vertices or edges. Likewise, we can also add a small number of vertices.

\begin{proposition}\label{prop:removingrobexp}
Let $0 \leq \nu \leq \tau \ll 1$ and suppose that $G$ is a robust $(\nu, \tau)$-outexpander on $n$ vertices. Let $V_0\subseteq V(G)$ be any set of at most $\nu n/4$ vertices. Then the graph $G'\coloneqq G\setminus V_0$ is a robust $(\nu/2, 2\tau)$-outexpander.
\end{proposition}

\begin{proof}
Let $n' \coloneqq |G'|$. Then $(1-\nu/4)n \leq n' \leq n$. Consider a set $S \subseteq V(G')$ of size $\tau n \leq 2\tau n' \leq |S| \leq (1-2\tau)n' \leq (1-\tau)n$.
We have lost at most $\nu n/4$ vertices so $|N^-_{G'}(x) \cap S| \geq \nu n - \nu n/4 \geq \nu n'/2$ for all $x \in RN^+_{\nu, G}(S)\setminus V_0$. Hence $|RN^+_{\nu/2, G'}(S)| \geq |RN^+_{\nu, G}(S)|- \nu n/4$. Then, since $G$ is a robust $(\nu,\tau)$-outexpander, we have that
$$|RN^+_{\nu/2, G'}(S)| \geq |RN^+_{\nu, G}(S)|- \nu n/4 \geq |S| + \nu n - \nu n/4\geq |S| + \nu n'/2.$$
Therefore, $G'$ is a robust $(\nu/2, 2\tau)$-outexpander.
\end{proof}

\begin{proposition}\label{prop:addingrobexp}
Let $0 \leq \nu \leq \tau \ll 1$ and suppose that $G$ is a robust $(\nu, \tau)$-outexpander on $n$ vertices. Let $V_0$ be a set of at most $\nu^2n$ vertices. Then the graph $G'\coloneqq G \cup V_0$ is a robust $(\nu/2, 2\tau)$-outexpander.
\end{proposition}

\begin{proof}
Let $n' \coloneqq |G'|$, so $n \leq n' \leq (1+\nu^2)n$. Consider a set $S \subseteq V(G')$ of size $2\tau n' \leq |S| \leq (1-2\tau)n'$. We know that the set $S$ can contain at most $\nu^2n$ new vertices so
$$\tau n \leq |S \cap V(G)| \leq (1-\tau)n.$$
Observe that $RN^+_{\nu, G}(S) \subseteq RN^+_{\nu/2, G'}(S)$ and so we can use that $G$ is a robust $(\nu,\tau)$-outexpander to see that
$$|RN^+_{\nu/2, G'}(S)| \geq |S \cap V(G)| + \nu n \geq |S| - \nu^2n + \nu n \geq |S| + \nu n'/2.$$
Therefore, $G'$ is a robust $(\nu/2, 2\tau)$-outexpander.
\end{proof}

It is clear from the proofs of these results that we can replace \textquoteleft outexpander' by \textquoteleft inexpander' or \textquoteleft diexpander' in the statements of \Cref{prop:removingrobexp} and \Cref{prop:addingrobexp}.

\section{Regularity and the Diregularity Lemma}

We will now define what it means for a digraph, $G$, (which is not necessarily bipartite) to be $\varepsilon$-regular. The density of the pair $(X,Y)_G$ is defined as in the undirected case, recall that $d_G(X,Y)=\frac{e_G(X,Y)}{|X||Y|}$. Note that the order of $X$ and $Y$ matters now, in general it will not be the case that $d_G(X,Y)=d_G(Y,X)$.

\begin{definition}\label{def:Diregularity}
Let $\varepsilon>0, d\in[0,1]$ and suppose that $G$ is a graph on $n$ vertices. We say that $G$ is \emph{$\varepsilon$-regular} with density $d$ if for all sets $X, Y \subseteq V(G)$ with $|X|, |Y| \geq \varepsilon n$ we have that $$|d_G(X,Y)-d| < \varepsilon.$$
We say that $G$ is \emph{$[\varepsilon,d]$-superregular} if it is $\varepsilon$-regular and $\delta^0(G) \geq dn$.
\end{definition}

Our definition of superregularity for digraphs differs slightly from our earlier definition for undirected graphs. We now require that the graph is $\varepsilon$-regular in order to be $[\varepsilon,d]$-superregular as this definition will be more convenient in subsequent statements of lemmas and proofs. For clarity, we will always write $(\varepsilon,d)$-superregular when we wish to apply the definition to the undirected graph and $[\varepsilon,d]$-superregular when we wish to refer to the directed graph.

In \Cref{matchingundir} we showed that a regular undirected graph meeting minimum degree conditions contains a perfect matching. We will now prove a similar result for superregular digraphs.

\begin{definition}
A $k$-factor in a digraph $G$ is a $k$-regular spanning subgraph of $G$.
\end{definition}

\Cref{superhalls} shows that we can use superregularity to guarantee that $G$ contains a $1$-factor, that is, a set of vertex disjoint cycles covering all of the vertices of $G$. 

\begin{proposition}\label{superhalls}
Suppose that $0<\varepsilon \ll d < 1$ and $G$ is an $[\varepsilon,d-\varepsilon]$-superregular digraph on $n$ vertices with density $d$. Then $G$ contains a $1$-factor.
\end{proposition}

\begin{proof} Let us define an auxiliary bipartite graph $G^*$ with vertex classes $A=V(G)$ and $B=V(G)$. For every $a \in A$ and $b \in B$, $ab \in E(G^*)$ if and only if the directed edge $ab \in E(G)$.

We will first show that $G^*$ contains a perfect matching. Let $S \subseteq A$.

Suppose $0<|S|\leq (d-\varepsilon) n$. Let $v \in S$ and note that $d^+(v) \geq (d-\varepsilon)n$ since $G$ is $[\varepsilon, d-\varepsilon]$-superregular. Then $$|N_{G^*}(S)| \geq |N_{G^*}(v)| = |N^+_G(v)| \geq (d-\varepsilon)n \geq |S|.$$

Let us now suppose that $|S| > (1-(d-\varepsilon))n$. Then, since $|A \setminus S| < (d-\varepsilon)n$, we have that for every $v \in B, N_{G^*}(v) \cap S \neq \emptyset$ as $d_{G^*}(v) = d^-_G(v) \geq (d-\varepsilon)n$. So $N_{G^*}(S) = B$ and therefore $|S| \leq |N_{G^*}(S)|$.

It remains to show that Hall's condition is satisfied for $\varepsilon n \leq (d-\varepsilon)n < |S| \leq (1-(d-\varepsilon))n$, we assume that $3\varepsilon \leq d^2$. Note that $|N_{G^*}(S)| \geq (d-\varepsilon)n \geq \varepsilon n$. We will assume, for the sake of contradiction, that $|N_{G^*}(S)| < |S| \leq (1-(d-\varepsilon))n$. Since for every $v \in S$ we have that $d_{G^*}(v) \geq (d-\varepsilon)n$ we get that $e_{G^*}(S, N_{G^*}(S)) \geq (d-\varepsilon)n|S|$. Hence 
\begin{equation*}
\begin{split}
d_{G^*}(S, N_{G}(S)) &= \frac{e(S, N_{G^*}(S))}{|S||N_{G^*}(S)|} \geq \frac{(d-\varepsilon)n}{|N_{G^*}(S)|} > \frac{(d-\varepsilon)n}{(1-(d-\varepsilon))n} \\&= d + \frac{d^2-\varepsilon d-\varepsilon}{1-(d-\varepsilon)} \geq d+ (d^2-\varepsilon d- \varepsilon)\\ &\geq d+\varepsilon.
\end{split}
\end{equation*}
But this contradicts the $\varepsilon$-regularity of $G$. Hence $|N_{G^*}(S)| \geq |S|$.

Therefore, $G^*$ satisfies the condition of Hall's theorem and, since $|A|=|B|$, has a perfect matching. This matching corresponds to a $1$-factor in $G$.
\end{proof}

If $G$ is an $\varepsilon$-regular digraph and we define the graph $G^*$ as above then we can see a correspondence between our definitions of regularity. That is, $G^*$ is also $\varepsilon$-regular.

We will use \Cref{superhalls} to find a Hamilton cycle in the following lemma which is a special case of a result of Frieze and Krivelevich, see \cite{frieze}.


\begin{lemma}\label{superham}
Suppose that $1/n_0 \ll \varepsilon \ll d \ll 1$ and $G$ is an $[\varepsilon,d-\varepsilon]$-superregular digraph on $n\geq n_0$ vertices with density $d$. Then $G$ contains a Hamilton cycle.
\end{lemma}

\begin{proof}
By \Cref{superhalls} we can consider a $1$-factor, $F$, in $G$. Choose any cycle in $F$, remove an edge from this cycle and call the resulting path $P = u_0u_1\ldots u_k$. If the final vertex, $u_k$, of $P$ has an outneighbour $x$ which does not lie on $P$ then extend the path $P$ by removing the edge $x^- x$ from the cycle in $F$ on which $x$ lies (where $x^-$ denotes the predecessor of $x$ on this cycle) and joining the two paths by the edge $u_k x$.

Similarly if the initial vertex, $u_0$, of $P$ has an inneighbour $x$ that does not lie on $P$ then we can extend $P$ to contain the vertex $x$ and all other vertices on the cycle in $F$ containing $x$.

By repeating this process as required, we may assume that all of the inneighbours of $u_0$ and the outneighbours of $u_k$ lie on $P$. Note that this implies that $|P| \geq \delta^0(G)+1 \geq (d-\varepsilon)n+1 > \varepsilon n$.


\begin{claim}
There exists a cycle $C$ with $V(C)=V(P)$.
\end{claim}

Let $\ell$ be the largest possible integer such that $\ell \leq (d-\varepsilon)n \leq \delta^0(G)$ and $\ell$ is divisible by $4$. We have $\ell \geq dn/2 \geq 8\varepsilon n$. Write
$$N^-(u_0) = \{x_1, x_2, \ldots, x_\ell, \ldots \} \text{ and } N^+(u_k) = \{y_1, y_2, \ldots, y_\ell, \ldots \}$$
where the vertices are listed according to their order of appearance on the path.
We will consider two cases:
\begin{enumerate}
\item $x_{\ell/2}$ appears after $y_{\ell/2}$ on $P$, or $x_{\ell/2}=y_{\ell/2}$.
\item $x_{\ell/2}$ appears before $y_{\ell/2}$ on $P$.
\end{enumerate}

\noindent\textbf{Case 1:} $x_{\ell/2}$ appears after $y_{\ell/2}$ on $P$, or $x_{\ell/2}=y_{\ell/2}$.

\noindent Consider the disjoint sets:
$$X_1=\{x_{\ell/2+1}, \ldots, x_\ell\} \subseteq N^-(u_0) \text{ and } Y_1=\{y_1, \ldots, y_{\ell/2}\} \subseteq N^+(u_k).$$

Note that $|X_1|,|Y_1| = \ell/2 \geq dn/4 \geq \varepsilon n$. If $u_ku_0 \in E(G)$ then we are done, so we assume that $u_ku_0 \notin E(G)$. We define two further sets: $X^+_1$ and $Y^-_1$ which are the sets of successors of $X_1$ and the predecessors of $Y_1$ on $P$ respectively (for each $v \in V(P)$ we will also write $v^+$ and $v^-$ for the successor and predecessor of $v$ on $P$). It follows that $|X^+_1|, |Y^-_1| \geq \varepsilon n$. Then, since $G$ is $\varepsilon$-regular, there is an edge $y^-_ix^+_j \in E(G)$ for some $y^-_i \in Y^-_1$ and $x^+_j \in X^+_j$.

Then $$C= u_0 P y^-_i x^+_j P u_k y_i P x_j u_0$$ is a cycle in $G$ with $V(C)=V(P)$, see \Cref{fig:case1} for an illustration.

\begin{figure}[h]
\centering
\includegraphics[scale=0.55]{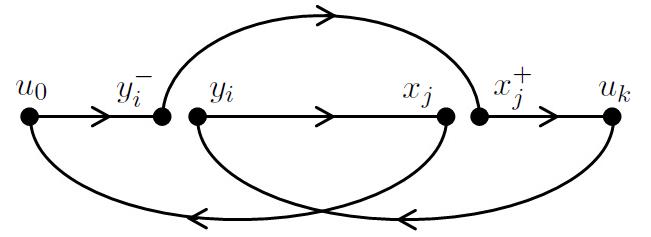}
\caption{The cycle obtained in Case 1}\label{fig:case1}
\end{figure}

\vspace{11pt}
\noindent\textbf{Case 2:} $x_{\ell/2}$ appears before $y_{\ell/2}$ on $P$. 

\noindent In this case, we will consider the disjoint sets:
$$X_1=\{x_1, \ldots, x_{\ell/2}\} \subseteq N^-(u_0) \text{ and } Y_1=\{y_{\ell/2}, \ldots, y_\ell\} \subseteq N^+(u_k).$$
Let $X^+_1$ and $Y^-_1$ be defined as previously. We will consider the subsets
$$X^+_2=\{x^+_{\ell/4+1}, \ldots, x^+_{\ell/2}\} \subseteq X^+_1 \text{ and } Y^-_2=\{y^-_{\ell/2}, \ldots, y^-_{3\ell/4-1}\} \subseteq Y^-_1.$$
We have that $|X^+_2|=|Y^-_2| \geq \ell/4 \geq dn/8$.

We let $X_3 = V(u_0 P x_{\ell/4})$ and $Y_3 = V(y^+_{3\ell/4} P u_{k})$. We have $|X_3|,|Y_3| \geq \ell/4 \geq dn/8$. Then we can apply \Cref{prop:subsets}, with $\alpha=d/8$, to the $\varepsilon$-regular bipartite graph $(V(G),V(G))_G$, to see that $(X_3, X^+_2)_G$ and $(Y^-_2, Y_3)_G$ are $\sqrt{\varepsilon}$-regular pairs. Now, by \Cref{prop:neighbours}, we can find subsets $X'_3 \subseteq X_3$ and $Y'_3 \subseteq Y_3$ of size at least $(1-\sqrt{\varepsilon}) \ell/4 = \ell/8 \geq \varepsilon n$ so that each vertex in $X'_3, Y'_3,$ has at least one neighbour in $X^+_2, Y^-_2$.

Now the sets $X'^+_3$ and $Y'^-_3$ have size at least $\varepsilon n$ so the $\varepsilon$-regularity of $G$ implies that there exist $u_{i+1} \in X'^+_3$ and $u_{j-1} \in Y'^-_3$ such that $u_{j-1}u_{i+1} \in E(G)$. We defined the sets $X'_3$ and $Y'_3$ in such a way that we can now find $x^+_i \in X^+_2$ such that $u_i x^+_i \in E(G)$ and $y^-_j \in Y^-_2$ such that $y^-_j u_j \in E(G)$.

We obtain a cycle $$C= u_0 P u_i x^+_i P y^-_j u_j P u_k y_j P u^-_j u^+_i P x_i u_0$$ with $V(C)=V(P)$ as shown in \Cref{fig:case2}. This proves the claim.

\begin{figure}[h]
\centering
\includegraphics[scale=0.55]{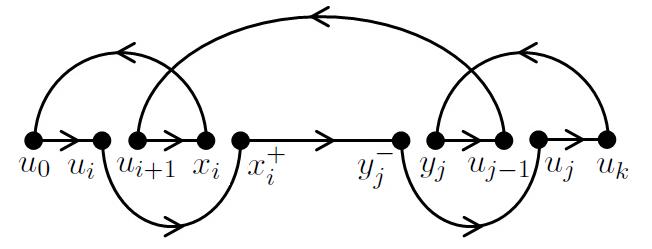}
\caption{The cycle obtained in Case 2}\label{fig:case2}
\end{figure}


We will now show that this cycle can be extended to form a Hamilton cycle. So let us suppose that $V(C) \neq V(G)$. Then we may consider a vertex $x \in V(G)\setminus V(C)$ lying on the cycle $C_x$ in the $1$-factor $F$. Note that $V(C) \cap V(C_x) = \emptyset$ since each time we extended the path $P$ to include a vertex $v$ we also added all other vertices lying on the same cycle in $F$.

Suppose that $N^+(x) \cap V(C) \neq \emptyset$, that is, there exists a vertex $y \in N^+(x) \cap V(C)$. Then we may consider the new, longer, path $P'=x^+C_x xyCy^-$. We can then carry out the same extension process we performed earlier in the proof so that all inneighbours of the initial vertex and outneighbours of the final vertex of $P'$ lie on $P'$. We then find a cycle $C'$ with $V(C')=V(P')$. 

Similarly, if $N^-(x) \cap V(C) \neq \emptyset$ then we may obtain a longer path and a cycle which has the same vertex set as this new path.

So we may assume that for all vertices $x \in V(G) \setminus V(C)$ we have $N^+(x) \cup N^-(x) \subseteq V(G) \setminus V(C)$. Then $$|V(G) \setminus V(C)| \geq \delta^0(G) \geq (d-\varepsilon) n \geq \varepsilon n.$$ But we have that $d_G(V(C),V(G)\setminus V(C)) = 0 < d-\varepsilon$ contradicting the $\varepsilon$-regularity of $G$. Therefore, $C$ is in fact a Hamilton cycle.
\end{proof}

We have seen an undirected form of the Regularity Lemma and there is also a directed form, the Diregularity Lemma, due to Alon and Shapira, see \cite{alon}. We state the degree form of the Diregularity Lemma below. It follows from the Diregularity Lemma in a similar way to its undirected counterpart.

\begin{lemma}[Degree form of the Diregularity Lemma]\label{direg}
For all $\varepsilon>0$ and all integers $k_0$ there is an $N = N(\varepsilon, k_0)$ such that for every number $d \in [0,1)$ and for every digraph $G$ on $n \geq N$ vertices there exist a partition of V(G) into $V_0, V_1, \ldots, V_k$ and a spanning subdigraph $G'$ of $G$ such that the following hold:
	\begin{enumerate}[(i)]
	\item $k_0 \leq k \leq N$ and $|V_0| \leq \varepsilon n$,
	\item $|V_1| = \cdots = |V_k| \eqqcolon m$,
	\item $d^+_{G'}(x) > d^+_G(x) - (d + \varepsilon)n$ for all vertices $x \in V(G)$,
	\item $d^-_{G'}(x) > d^-_G(x) - (d + \varepsilon)n$ for all vertices $x \in V(G)$,
	\item for all $i \geq 1$ the digraph $G'[V_i]$ is empty,
	\item for all $1 \leq i,j \leq k$ with $i \neq j$ the graph $(V_i,V_j)_{G'}$ is $\varepsilon$-regular and has density either 0 or $>d$.
	\end{enumerate}
\end{lemma}

We refer to $G'$ as the \emph{pure digraph}. We define a \emph{reduced digraph} $R$, as in the undirected case, that is, the vertices of $R$ are $\{V_1, \ldots, V_k\}$ and we have an edge from $V_i$ to $V_j$ in $R$ if the graph $(V_i,V_j)_{G'}$ is $\varepsilon$-regular and $d_{G'}(V_i,V_j)\geq d$. $R$ inherits some of the properties of $G$, for instance, if $G$ is a robust outexpander then we can show that $R$ is as well.

\begin{lemma}\label{Routexpander}
Let $k_0,n_0$ be positive integers and $\varepsilon, d, \eta, \nu, \tau$ be positive constants such that $$1/n_0 \ll 1/k_0,\varepsilon \ll d \ll \nu, \tau, \eta < 1.$$ Suppose that $G$ is a digraph on $n \geq n_0$ vertices with $\delta^0(G) \geq \eta n$ such that $G$ is a robust $(\nu, \tau)$-outexpander. Let $R$ be the reduced digraph of $G$ with parameters $\varepsilon, d$ and $k_0$. Then $\delta^0(R) \geq \eta |R| /2$ and $R$ is a robust $(\nu/2, 2\tau)$-outexpander.
\end{lemma}

\begin{proof}
Apply the Diregularity Lemma (\Cref{direg}) to the digraph $G$ with parameters $\varepsilon, d$ and $k_0$. We obtain a partition of $V(G)$ into clusters $V_1, \ldots, V_k$ with $|V_1|=\ldots=|V_k|=m$ and an exceptional set $V_0$. $G'$ denotes the pure digraph, $R$ the reduced digraph and we note that $|R|=k$.

Let $V_i \in V(R)$ and consider any $x \in V_i$. Observe that $x$ has outneighbours in at least $(\delta^+(G')-|V_0|)/m$ clusters $V_j$ in $G'$. Similarly, $x$ has inneighbours in at least $(\delta^-(G')-|V_0|)/m$ clusters in $G'$. Then, using part $(vi)$ of \Cref{direg} and the definition of $R$, we see that $$ \delta^0(R) \geq (\delta^0(G')-|V_0|)/m \geq ((\delta^0(G)-(d+\varepsilon)n)-\varepsilon  n)/m$$ and so $$\delta^0(R) \geq (\eta-(d+2\varepsilon))n/m \geq \eta k/2.$$

Now suppose that $S \subseteq V(R)=\{V_1, \ldots, V_k\}$ with $2\tau k \leq |S| \leq (1-2\tau)k$ and let $S'$ be the set of vertices inside clusters in $S$, that is, $S'= \{x\in V_i: V_i \in S\}$. Then $$\tau n \leq 2\tau k m \leq |S'| \leq (1-2 \tau)km \leq (1-2\tau)n.$$

Recall that $RN^+_{\nu,G}(S')$ is the set of vertices having at least $\nu n$ inneighbours in $S'$ in the digraph $G$. For any $x \in RN^+_{\nu,G}(S')$, we have that
$$|N^-_{G'}(x) \cap S'| \geq |N^-_G(x) \cap S'| - (d+\varepsilon)n \geq \nu n/2.$$
In the graph $G'$ every vertex $x \in R^+_{\nu/2,G'}(S') \setminus V_0$ is an outneighbour of vertices from at least $\nu k/2$ different clusters $V_i\in S$. This is because
$$|N^-_{G'}(x) \cap S'|/m \geq(\nu n /2)/m \geq \nu k/2.$$ 
Then, by part $(vi)$ of \Cref{direg}, if $V_j$ is the cluster containing $x$, $V_j$ is an outneighbour of the vertices of $R$ corresponding to each of these $\nu k/2$ clusters. Therefore, $V_j \in RN^+_{\nu/2,R}(S)$.

Clearly, $|RN^+_{\nu/2,G'}(S')| \geq |RN^+_{\nu,G'}(S')|$. As $G$ is a robust $(\nu, \tau)$-outexpander, we have that $$|RN^+_{\nu/2,G'}(S')| \geq |S'|+\nu n \geq |S|m+ \nu m k.$$
Then we find that $$|RN^+_{\nu/2,R}(S)| \geq (|RN^+_{\nu/2,G'}(S')|-|V_0|)/m \geq |S|+\nu k -\varepsilon n/m \geq |S|+ \nu k/2.$$
Therefore, $R$ is a robust $(\nu/2, 2\tau)$-outexpander.
\end{proof}

\section{Hamilton Cycles in Robust Outexpanders}

We will now prove that a sufficiently large robust outexpander of linear minimum degree contains a Hamilton cycle. The proof will require the concept of a \emph{shifted walk}.

\begin{definition}\label{def:shifted}
Suppose that $G$ is a digraph and $F$ is a $1$-factor in the reduced digraph $R$. We define a \emph{shifted walk} in $R$ from a cluster $A$ to a cluster $B$, $W(A,B)$, to be a walk of the form
$$W(A,B)=X_1 C_1 X^-_1 X_2 C_2 X^-_2 \ldots X_t C_t X^-_t X_{t+1}$$
where $X_1=A$ and $X_{t+1}=B$ and for each $1 \leq i \leq t$:
\begin{itemize}
	\item $C_i$ is the cycle of $F$ containing $X_i$;
	\item $X^-_i$ is the predecessor of $X_i$ on $C_i$ and
	\item the edge $X^-_iX_{i+1}$ lies in $E(R)$.
\end{itemize}
We say that $W(A,B)$ \emph{traverses} $t$ cycles, even if some cycles are used more than once.
We say that the clusters $\{X_2, X^-_2, \ldots , X_t, X^-_t\}$ are used \emph{internally} by $W(A,B)$.
The clusters $\{X_2, X_3, \ldots, X_{t+1}\}$ are referred to as the \emph{entry clusters} and the clusters $\{X^-_1, X^-_2, \ldots, X^-_t\}$ are the \emph{exit clusters} of $W(A,B)$.
\end{definition}

\begin{figure}[h]
\centering
\includegraphics[scale=0.4]{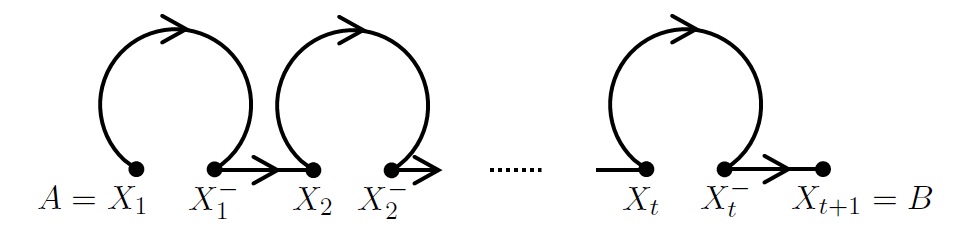}
\caption{A shifted walk, $W(A,B)$, from $A$ to $B$.}\label{fig:shiftedwalk}
\end{figure}

Each time $W(A,B)\setminus B$ visits a cycle it uses all of the clusters on that cycle so we observe that, for any cycle in $F$, $W(A,B)\setminus B$ visits each of its clusters the same number of times. We also note that if we have a closed shifted walk, $W(A,A)$, then this implies that $W(A,A)$  again visits all clusters lying on the same cycle the same number of times.

We may also assume that $W(A,B)$ uses every cluster at most once as an entry cluster since if a cluster $X$ is used multiple times as an entry then we can remove the section of the walk between the first and last appearances of $X$ as an entry cluster to obtain a shorter shifted walk from $A$ to $B$ which only uses $X$ once as an entry. Similarly, we can assume that each cluster is used at most once as an exit cluster. The following result will be used to find short shifted walks in the proof of \Cref{expander}.

\begin{proposition}\label{shortwalk}
Let $0< \nu \leq \tau\ll \eta <1$. Suppose that $R$ is a $(\nu/2, 2\tau)$-outexpander on $k$ vertices with $\delta^0(R) \geq \eta k/2$ and suppose that $R$ has a $1$-factor $F$. Let $Q\subseteq V(R)$ with $|Q| \leq \nu k/8$ and suppose $A,B \in V(R)$. Then there is a shifted walk $W(A,B)$ avoiding $Q$ internally which traverses at most $4/\nu$ cycles.
\end{proposition}

\begin{proof}
Let $A=U_1$ and for each $i > 1$ let $U_i$ be the set of clusters that can be reached from $A$ by a shifted walk traversing $i-1$ cycles which avoids $Q$ internally. We denote by $U^-_i$ the set of predecessors of the clusters in $U_i\setminus Q$, that is, $U^-_i=\bigcup_{A\in U_i\setminus Q} A^-$. Note that $|U^-_i|\geq |U_i|-\nu k/8$ for all $i>1$.

We first note that $|U_2| \geq d_R(A^-) \geq \delta^0(R) \geq \eta k/2$. If $|U^-_2\setminus Q|\leq (1-2\tau)k$ then we can use that $R$ is a $(\nu/2, 2\tau)$-outexpander to see that $$|U_3| \geq |U^-_2\setminus Q|+\nu k/2 \geq (|U_2| - \nu k/4)+ \nu k/2 \geq (2\eta+\nu)k/4.$$

Continuing in this way we see that, as long as $|U^-_t\setminus Q| \leq (1-2\tau)k$, by traversing $t$ cycles we can reach $$|U_{t+1}| \geq (2\eta + (t-1)\nu)k/4 \geq t\nu k/4$$ clusters.

Let $\ell$ be the smallest positive integer such that $|U^-_\ell\setminus Q|>(1-\eta/2)k$. Observe that $\ell \leq 4/\nu$. Now, we know that $d^-_R(B) \geq \delta^0(R) \geq \eta k/2 > |R|-|U^-_\ell\setminus Q|$. Therefore $N^-_R(B) \cap (U^-_\ell\setminus Q) \neq \emptyset$. Hence there is a shifted walk from $A$ to $B$ which traverses at most $4/\nu$ cycles and avoids $Q$ internally.
\end{proof}

We will use the results that we have gathered so far to prove that we can find a Hamilton cycle in a robust outexpander. First we will apply the Diregularity Lemma to the graph $G$ and then we will find a $1$-factor in the reduced graph $R$. Using shifted walks, we will incorporate the exceptional vertices and obtain a closed walk, made up of shifted walks, which visits all of the clusters. Finally, we will show that we can use this walk to construct a $1$-factor in $G$ in such a way that every vertex lies on the same cycle in $G$ - a Hamilton cycle.

\begin{theorem}[K\"uhn, Osthus and Treglown \cite{kot}]\label{expander}
Let $n_0$ be a positive integer and $\nu, \tau, \eta$ be positive constants such that $1/n_0 \ll \nu \leq \tau \ll \eta <1$. Let $G$ be a digraph on $n \geq n_0$ vertices with $\delta^0(G) \geq \eta n$ which is a robust $(\nu, \tau)$-outexpander. Then $G$ contains a Hamilton cycle.
\end{theorem}

\begin{proof}
Choose constants $\varepsilon$ and $d$ satisfying $$1/n_0 \ll \varepsilon \ll d \ll \nu$$ and apply the degree form of the Diregularity Lemma (\Cref{direg}) with parameters $\varepsilon, d$ and $k_0=1/\varepsilon$. We obtain a partition $V_1, \ldots V_k$ with $|V_1|=\ldots=|V_k|\eqqcolon m$; an exceptional set $V_0$ with $|V_0| \leq \varepsilon n$ and a reduced digraph $R$. By \Cref{Routexpander}, we have that $R$ is a $(\nu/2, 2\tau)$-outexpander with $\delta^0(R) \geq \eta k/2$.

\begin{claim}
$R$ contains a 1-factor.
\end{claim}

Define an auxiliary bipartite graph, $R^*$, as in the proof of \Cref{superhalls}, with vertex classes $A=V(R)$ and $B=V(R)$. So, for every $a \in A$ and $b \in B$, $ab \in E(R^*)$ if and only if the directed edge $ab$ lies in $R$. If $S \subseteq A$ with $2\tau k < |S| < (1- 2\tau)k$ then $|N_{R^*}(S)| = |N^+_R(S)| \geq |S|+\nu k/2$. So Hall's condition holds in this case. Suppose now that $|S| < 2\tau k$. Then $|N_{R^*}(S)| \geq \delta^0(R) \geq \eta k/2 > |S|$. If we have that $|S| > (1-2\tau)k$ then we note that  for all $u \in B$, $N_{R^*}(u) \cap S \neq \emptyset$ so $N_{R^*}(S)=B$ which implies that $|N_{R^*}(S)| \geq |S|$. We find that $R^*$ satisfies Hall's condition and so contains a perfect matching. This matching corresponds to a $1$-factor in $R$ which we shall call $F$. 

\vspace{6pt}
For each $A \in V(R)$ we will write $A^+$ and $A^-$ to denote the successor and predecessor of $A$, respectively, on the cycle of $F$ containing $A$. Suppose that $(A, A^+)_G$ has density $d_A>d$. By \Cref{superreg}, we see that by removing $2\varepsilon m$ vertices from each cluster (and adding these to the exceptional set), we can assume that for each cluster $A$ the bipartite graph $(A, A^+)_G$ is $(2\varepsilon,d_A-3\varepsilon)$-superregular. Note that here we ignore the orientations of the edges and consider the underlying undirected graph together with the definition of superregularity given in \Cref{subsec:reg}. We also have, by \Cref{prop:subsets}, that $(A, A^+)_G$ is $2\varepsilon$-regular. So we have that $(A, A^+)_G$ is:
\begin{itemize}
\item[(a)] $(2\varepsilon,d_A-3\varepsilon)$-superregular and
\item[(b)] $2\varepsilon$-regular.
\end{itemize}
Let $m'\coloneqq m-2\varepsilon m$. We will continue to refer to the clusters as $V_1, V_2, \ldots V_k$ and the exceptional set as $V_0$. We have added $2\varepsilon mk \leq 2\varepsilon n$ vertices to the exceptional set and so we now have that $|V_0| \leq 3\varepsilon n$.


We would like to find a closed walk which visits all of the exceptional vertices and all of the clusters. We will begin by assigning each vertex in the exceptional set $V_0=\{a_1, a_2, \ldots, a_s\}$ to clusters in $R$ as follows.

For each $a_i \in V_0$, we say that a cluster $V_j$ is a \emph{good outcluster} for $a_i$ if $a_i$ has many outneighbours in $V_j$. More precisely, if $$|N^+_G(a_i) \cap V_j| \geq \eta m'/2.$$ We want to assign each exceptional vertex $a_i$ to a good outcluster $T_i$.

Let $q$ be the number of good outclusters. We have that $d^+_G(a_i) \geq \eta n$ and so
$$qm'+(k-q)\eta m'/2 + |V_0| \geq d^+_G(a_i) \geq \eta n.$$
This implies that $$q \geq \frac{\eta n -|V_0|}{m'} - \frac{\eta k}{2} \geq \frac{(\eta-3\varepsilon)n}{m'}-\frac{\eta k}{2} \geq \frac{\eta k}{3}.$$
When choosing a cluster to which to assign the vertex $a_i$, we say that a cluster $V_j$ is \emph{full} if it has been chosen for at least $4|V_0|/(\eta k)$ of the vertices $a_1, a_2, \ldots a_{i-1}$. Then we have at most $$|V_0|/(4|V_0|/(\eta k)) = \eta k/4$$ full clusters. Since the  number of full clusters is less than the number of good outclusters, this means that we can assign each exceptional vertex $a_i$ to a good outcluster, $T_i$, in such a way that each cluster is used at most $\sqrt{\varepsilon}m'/4 > 4|V_0|/(\eta k)$ times.

Similarly, for each $a_i \in V_0$ we say that $V_j$ is a \emph{good incluster} if $$|N^-_G(a_i) \cap V_j| \geq \eta m'/2.$$ Then, by similar reasoning, we can assign a good incluster $U_i$ to each $a_i$ so that no cluster is used more than $\sqrt{\varepsilon}m'/4 \geq |V_0|/(\eta^2k)$ times.

\begin{claim}
There exists a closed spanning walk $W$ on $V_0 \cup V(R)$ which visits all clusters on the same cycle in $F$ the same number of times and which does not use any cluster more than $\sqrt{\varepsilon}m'$ times as an entry cluster or more than $\sqrt{\varepsilon}m'$ times as an exit cluster.
\end{claim}

We define $W$ by a series of shifted walks. Starting at $a_1$ we move to $T_1$ and then follow a shifted walk $W(T_1,U^+_2)$ in $R$. The walk then continues along the cycle to $U_2$ from which it can reach the vertex $a_2$. Continuing in this way, $W$ visits all of the exceptional vertices. For convenience, we extend our definition of an entry cluster so as to include the clusters $T_i$ where we \textquoteleft enter' the first cycle on the walk $W(T_i,U^+_{i+1})$. Similarly, we will also consider the cluster $U_{i+1}$ to be an exit cluster. Finally, we add at most $k$ further shifted walks between any clusters that have not already been covered and to return to $a_1$. We will show that we can choose these shifted walks greedily so that each traverses at most $4/\nu$ cycles and no cluster is used more than $\sqrt{\varepsilon}m'$ times as an entry cluster or more than $\sqrt{\varepsilon}m'$ times as an exit cluster.

Suppose that we have already found $i<|V_0|+k$ such walks. Let $Q$ be the set of clusters that have been used at least $\sqrt{\varepsilon}m'/5$ times internally. Now each of these walks uses at most $8/\nu$ clusters internally so
$$|Q| \leq (|V_0|+k)(8/\nu)/(\sqrt{\varepsilon}m'/5)\leq 160\sqrt{\varepsilon}k/\nu \leq \nu k/8.$$ Then, by \Cref{shortwalk}, we can find the next required shifted walk, traversing at most $4/\nu$ cycles and avoiding $Q$ internally. Since we can assume that a cluster is used at most two times internally by a shifted walk (once as an entry and once as an exit), we can ensure that we find the walks so that each cluster is used at most $\sqrt{\varepsilon}m'/5+2 \leq \sqrt{\varepsilon}m'/4$ times internally.

Together, these shifted walks form a closed spanning walk $W$ on $V_0 \cup V(R)$. We have that each cluster $V$ has been used at most $\sqrt{\varepsilon}m'/4$ times internally. $V$ may also appear up to $\sqrt{\varepsilon}m'/4$ times as $T_i$ and up to $\sqrt{\varepsilon}m'/4$ times as $U^+_{i+1}$ in shifted walks of the form $W(T_i,U^+_{i+1})$. Finally, we suppose that the cluster $V$ was not visited in the initial series of shifted walks between exceptional vertices. Then we have added a shifted walk from some cluster $V'$ to $V$, $W(V',V)$, and another shifted walk, $W(V,V'')$, from $V$ to some cluster $V''$. Together these walks form a longer shifted walk from $V'$ to $V''$. We must add another occurrence of the vertex $V$ as an entry cluster here. So in total, each cluster has been used at most $$\sqrt{\varepsilon}m'/4+\sqrt{\varepsilon}m'/4+\sqrt{\varepsilon}m'/4+1 \leq \sqrt{\varepsilon}m'$$ times as an entry cluster. Similarly, we find that each vertex appears at most $\sqrt{\varepsilon}m'$ times as an exit cluster. Since we have constructed $W$ using shifted walks, $W$ uses all clusters lying on the same cycle an equal number of times, as required.

\vspace{6pt}

We now employ a \textquoteleft short-cutting' technique. We will fix edges in $G$ corresponding to those edges in $W$ which are not contained in a cycles of $F$:
\begin{itemize}
\item for each exceptional vertex $a_i$ we fix an edge of the form $a_it_i$ where $t_i \in T_i$ and an edge of the form $u_ia_i$ where $u_i \in U_i$;
\item for each edge $XY$ in $W$ where $X$ is an exit cluster and $Y$ is an entry cluster, we fix an edge of the form $xy$ where $x\in X$ and $y \in Y$.
\end{itemize}
Since no cluster appears more than $\sqrt{\varepsilon}m'$ times as an entry cluster or more than $\sqrt{\varepsilon}m'$ times as an exit cluster, we can choose these edges to be disjoint outside $V_0$.

For each cluster $A$, let $A_{exit}$ be the set of all vertices in $A$ which are the initial vertex of a fixed edge and $A_{entry}$ be the set of all vertices in $A$ which are the final vertex of a fixed edge. Observe that $A_{entry}$ and $A_{exit}$ are disjoint. We define a bipartite graph $G_A=(A_1,A_2)_G$ where $A_1 = A\setminus A_{exit}$ and $A_2=A^+\setminus A^+_{entry}$ and we consider $G_A$ as an undirected graph. As $W$ is made up of shifted walks, we see that $|A_1|=|A_2|$. Also $|A^+_{entry}|=|A_{exit}| \leq \sqrt{\varepsilon}m'$. So, using (a), we can apply \Cref{prop:subsets} to see that $G_A$ is $4\varepsilon$-regular with density $d'_A \in (d_A-4\varepsilon, d_A+ 4\varepsilon)$. We also have, by (b) and \Cref{removing}, that $G_A$ is $(\sqrt{2\varepsilon}, d_A-3\varepsilon-\sqrt{2\varepsilon})$-superregular. We conclude that $G_A$ is:
\begin{itemize}
\item[(c)]$(\sqrt{2\varepsilon}, d'_A-\sqrt[3]{\varepsilon})$-superregular and
\item[(d)]$\sqrt[3]{\varepsilon}$-regular.
\end{itemize}
Then we can apply \Cref{matchingundir} to see that $G_A$ contains a perfect matching. We will denote this perfect matching $M_A$.

\begin{figure}
\centering
\includegraphics[scale=0.55]{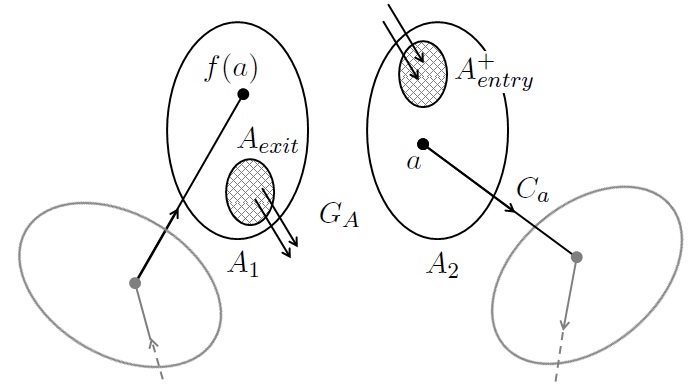}
\caption{The bipartite graph $G_A$ and the vertex $f(a)$ associated with the vertex $a$ in the construction of the digraph $J$.}\label{fig:GA}
\end{figure}

If we consider the set of edges in $M_A$ together with all of the fixed edges, we see that these form a $1$-factor $\mathcal{C}$ in $G$. We will show that we can modify $\mathcal{C}$ so that it becomes a Hamilton cycle.

\begin{claim}
For every cluster $A$, we can find a perfect matching $M'_A$ in $G_A$ such that, if we replace $M_A$ by $M'_A$ in $\mathcal{C}$, then all vertices of $G_A$ lie on a common cycle in the new $1$-factor.
\end{claim}

For each vertex $a \in A_2$ let $C_a$ be the cycle on which $a$ lies in $\mathcal{C}$ and let $f(a)$ be the first vertex encountered in $A_2$ when following the cycle $C_a$, starting from $a$. We define an auxiliary digraph $J$ with $V(J)=A_2$ and $E(J)=\{ab: a,b \in A_2 \text{ and } f(a)b\in E(G_A)\}$. In other words, we have an edge from $a$ to each of the outneighbours of $f(a)$ in $G_A$, that is, $N^+_J(a) \coloneqq N^+_{G_A}(f(A))$. We have that $e(J)=e(G_A)$ and so $J$ has density $d'_A$, the same as $J$.

By (c), we know that $\delta^+(J)>(d'_A-\sqrt[3]{\varepsilon})|A_2|=(d'_A-\sqrt[3]{\varepsilon})|J|$. By identifying each vertex $a$ with the vertex $f(a)$, we see that we also have  $\delta^-(J)>(d'_A-\sqrt[3]{\varepsilon})|A_2|=(d'_A-\sqrt[3]{\varepsilon})|J|$. So $\delta^0(J)>(d'_A-\sqrt[3]{\varepsilon})|J|$. If we choose subsets $X,Y \subseteq V(J)$ of size at least $\sqrt[3]{\varepsilon} |J|$, then, by considering the subsets $X'=\{f(a):a\in X\} \subseteq A_1$ and $Y \subseteq A_2$, we see that the regularity of $G_A$, (d), implies that $J$ is also $\sqrt[3]{\varepsilon}$-regular. So we have that $J$ is $[\sqrt[3]{\varepsilon}, d'_A-\sqrt[3]{\varepsilon}]$-superregular. Then, by \Cref{superham}, $J$ has a Hamilton cycle. This Hamilton cycle corresponds to the required $1$-factor $M'_A$ in $G_A$.

\vspace{6pt}

We apply the claim to every cluster and we will denote the resulting $1$-factor again by $\mathcal{C}$. Now, for each cluster $A$, we have that $A_{entry} \cap A_{exit} = \emptyset$ so we know that every vertex $x \in A$ is contained in at least one of $V(G_A)$ and $V(G_{A^-})$. Then, using that $V(G_A)\cap V(G_{A^-}) \neq \emptyset$, by the claim, all vertices contained in clusters that lie on the same cycle in $F$ will lie on the same cycle in $\mathcal{C}$. We also know that, since $W$ visits every cluster, all non-exceptional vertices must lie on the same cycle in $\mathcal{C}$. Finally, we observe that, since $V_0$ is an independent set in $\mathcal{C}$, each exceptional vertex must lie on cycle in $\mathcal{C}$ which also contains non-exceptional vertices. Therefore, $\mathcal{C}$ is a Hamilton cycle.
\end{proof}

Recall that in \Cref{robdegseq} we stated a degree sequence condition which implies a graph is a robust outexpander. So we can obtain an approximate proof of the conjecture of Nash-Williams, Conjecture \ref{nashconjecture}, as a corollary to \Cref{expander}.

\begin{corollary}
Let $n_0$ be a positive integer and $\tau, \eta$ be constants such that
$$1/n_0 \ll \tau \ll \eta <1.$$ Suppose that $G$ is a digraph on $n \geq n_0$ vertices satisfying
	\begin{enumerate}[(i)]
	\item $d^+_i \geq i+ \eta n$ or $d^-_{n-i-\eta n} \geq n-i$ and
	\item $d^-_i \geq i+ \eta n$ or $d^+_{n-i-\eta n} \geq n-i$
	\end{enumerate}
for all $i<n/2$.
Then $G$ contains a Hamilton cycle.
\end{corollary}

\chapter{Arbitrary Orientations of Hamilton Cycles}
\label{chap:arbitrary}

In the previous sections, we have looked for Hamilton cycles in digraphs and always assumed that these cycles are oriented in the standard way. In this chapter, we will instead consider what minimum semidegree will guarantee that we have, not only a standard Hamilton cycle, but \emph{any} orientation of a Hamilton cycle. For example, we could also require that a graph (on an even number of vertices) contains an \emph{anti-directed} Hamilton cycle, a Hamilton cycle in which the orientations of the edges alternate. In \cite{kelly}, Kelly gave a condition on the minimum semidegree which will ensure every orientation of a Hamilton cycle in any sufficiently large oriented graph.

\begin{theorem}[Kelly, \cite{kelly}]\label{anyorient}
For every $\alpha>0$ there exists an integer $n_0=n_0(\alpha)$ such that every oriented graph on $n\geq n_0$ vertices with $\delta^0(G) \geq (3/8+\alpha)n$ contains every orientation of a Hamilton cycle.
\end{theorem}

Recall from \Cref{thm:stdhamcycle} that any sufficiently large oriented graph with minimum semidegree at least $(3n-4)/8$ has a directed Hamilton cycle. We might then expect to be able to replace the bound in \Cref{anyorient} by $(3n-4)/8$. However, this minimum semidegree does not suffice when we look for any orientation of a Hamilton cycle. We show this in \Cref{prop:bestforham} when we construct an oriented graph which does not contain an anti-directed Hamilton cycle. Again, we refer to \Cref{fig:noantidir} for an illustration.

\begin{proposition}\label{prop:noantidir}
There are infinitely many oriented graphs $G$ with $\delta^0(G) = (3|G|-4)/8$ which do not contain an anti-directed Hamilton cycle.
\end{proposition}

\begin{proof}
Suppose that $m$ is a positive integer and let $n \eqqcolon 8m+4$. We will define an oriented graph $G$ on $n$ vertices as follows. Let $A$ and $C$ be regular tournaments on $2m+1$ vertices and let $B$ and $D$ each be sets of $2m+1$ vertices. Then $G$ is the disjoint union of $A$, $B$, $C$ and $D$ together with:
\begin{itemize}
\item all edges from $A$ to $B$, $B$ to $C$, $C$ to $D$ and $D$ to $A$;
\item all edges between $B$ and $D$, oriented to form a bipartite graph which is as regular as possible, so that the indegree and outdegree of each vertex differ by at most one.
\end{itemize} 
Then $d^+(x) = d^-(x) = 3m+1$ for all $x \in V(A),V(C)$ and $d^+(x) = d^-(x)\geq 3m+1$ for all $x \in B,D$. So $\delta^0(G) = 3m+1 = (3n-4)/8$. We will show that this graph does not contain an anti-directed Hamilton cycle.

Choose any vertex $v$ in $B$. We will try to construct an anti-directed Hamilton cycle, starting from this vertex. First we suppose we follow a forward oriented edge from $v$. We see that this edge must enter either $C$ or $D$. Then, since we are constructing an anti-directed cycle, the next edge must go backwards. From $C$ we either enter $B$ or remain in $C$ and from $D$ we move to either $B$ or $C$. So in both cases, we enter either $B$ or $C$. If we then follow a forward oriented edge from $B$ the situation repeats. A forward oriented edge from $C$ enters $D$ or remains in $C$. So we again see that we can repeat the previous arguments. But none of the anti-directed paths we have considered visit $A$. If we follow a backward oriented edge from $v$ instead, a similar argument shows that all anti-directed paths starting from $v$ with a backward edge avoid $C$. So any anti-directed cycle in $G$ which visits the vertex $v$ will either avoid all vertices in $A$ or all vertices in $C$. Therefore, $G$ does not contain an anti-directed Hamilton cycle. 
\end{proof}

We will ultimately prove the following generalisation of \Cref{anyorient} for robust $(\nu, \tau)$-outexpanders. The proof will closely follow that of Kelly's result.

\begin{theorem}\label{cor:anyorient2}
Let $n_0$ be a positive integer and $\nu, \tau, \eta$ be positive constants such that $1/n_0 \ll \nu \leq \tau \ll \eta <1$. Let $G$ be a digraph on $n \geq n_0$ vertices with $\delta^0(G) \geq \eta n$ and suppose $G$ is a robust $(\nu, \tau)$-outexpander. Then $G$ contains every orientation of a Hamilton cycle.
\end{theorem}

Recall from \Cref{lem:3/8} that any sufficiently large oriented graph $G$ of minimum semidegree at least $(3/8+\alpha)n$ is a robust outexpander. So we see that \Cref{cor:anyorient2} indeed implies \Cref{anyorient}. Similarly, we saw a degree sequence condition in \Cref{robdegseq} which implies that $G$ is a robust outexpander. So \Cref{cor:anyorient2} also implies the following result.

\begin{corollary}
Let $n_0$ be a positive integer and $\tau, \eta$ be constants such that
$$1/n_0 \ll \tau \ll \eta <1.$$ Suppose that $G$ is a digraph on $n \geq n_0$ vertices satisfying
	\begin{enumerate}[(i)]
	\item $d^+_i \geq i+ \eta n$ or $d^-_{n-i-\eta n} \geq n-i$ and
	\item $d^-_i \geq i+ \eta n$ or $d^+_{n-i-\eta n} \geq n-i$
	\end{enumerate}
for all $i<n/2$. Then $G$ contains every orientation of a Hamilton cycle.
\end{corollary}

\section{Some Useful Results and Techniques}

DeBiasio observed (in private correspondence) that robust outexpansion implies robust inexpansion.

\begin{proposition}[DeBiasio]\label{prop:inout}
Let $n_0$ be a positive integer and $\nu, \tau, \eta$ be positive constants such that $1/n_0 \ll \nu \leq \tau \ll \eta <1$. Let $G$ be a digraph on $n \geq n_0$ vertices with $\delta^0(G) \geq \eta n$. If $G$ is a robust $(\nu, \tau)$-outexpander, then $G$ is also a robust $(\nu^3,2\tau)$-inexpander.
\end{proposition}

\begin{proof}
Suppose that $G$ is not a robust $(\nu^3,2\tau)$-inexpander. Then there is a set $S\subseteq V(G)$ with $2\tau n <|S|<(1-2\tau)n$ such that $|RN^-_{\nu^3,G}(S)|<|S|+\nu^3n$. Let $T:=V(G)\setminus RN^-_{\nu^3,G}(S)$. Observe that
$$|S|\eta n \leq e(V(G), S) \leq |RN^-_{\nu^3,G}(S)||S|+(n-|RN^-_{\nu^3,G}(S)|)\nu^3n,$$
so $|RN^-_{\nu^3,G}(S)|> \eta n/2$.
Therefore,
$$\tau n < n-(1-2\tau+\nu^3)n<|T|< (1-\eta/2)n<(1-\tau)n.$$
By the definition of $T$, we have that $e(T,S)< |T|\nu^3n$ and so
$|RN^+_{\nu,G}(T)\cap S|<|T|\nu^2< \nu^2n$.
Hence,
\begin{align*}
|RN^+_{\nu,G}(T)|&=|RN^+_{\nu,G}(T)\setminus S|+|RN^+_{\nu,G}(T)\cap S|\\
&< n-|S|+\nu^2n \leq n-(|S|+\nu^3n)+\nu n\\
&< n-|RN^-_{\nu^3,G}(S)|+\nu n = |T|+\nu n
\end{align*}
and so $G$ is not a robust $(\nu, \tau)$-outexpander.
\end{proof}

We will actually prove the following result. Note that \Cref{prop:inout} gives that any robust $(\nu, \tau)$-outexpander is also a robust $(\nu^3, 2\tau)$-diexpander, so \Cref{cor:anyorient2} and \Cref{anyorient2} are equivalent.

\begin{theorem}\label{anyorient2}
Let $n_0$ be a positive integer and $\nu, \tau, \eta$ be positive constants such that $1/n_0 \ll \nu \leq \tau \ll \eta <1$. Let $G$ be a digraph on $n \geq n_0$ vertices with $\delta^0(G) \geq \eta n$ and suppose $G$ is a robust $(\nu, \tau)$-diexpander. Then $G$ contains every orientation of a Hamilton cycle.
\end{theorem}

We will split the proof of \Cref{anyorient2} into two cases based on how close the orientation of the Hamilton cycle, $C$, we wish to find is to the standard orientation. The following definition, allows us to compare any cycle to the standard orientation  

\begin{definition}
Suppose that $G$ is an oriented graph. The subgraph induced by distinct vertices $x,y,z \in V(G)$ is called a \emph{neutral pair} if $xy, zy \in E(G)$. We write $n(G)$ for the number of neutral pairs in $G$.
\end{definition}

We observe that every neutral pair $x,y,z$ in an arbitrarily oriented cycle $C$ must also have a corresponding \textquoteleft inverse' neutral pair $x',y',z'$ whose edges have the opposite direction, that is, $y'x', y'z' \in E(C)$. Then $n(C)$ indicates the number of changes of direction on the cycle.

In our proof of \Cref{anyorient2} we will need to use the Diregularity Lemma, \Cref{direg}. The reduced digraph inherits many properties of $G$. Recall that in \Cref{Routexpander}, we were able to show that the reduced digraph will also be a robust outexpander. By considering also the robust inneighbourhoods, it is easy to adapt the proof of \Cref{Routexpander} and obtain the following result.

\begin{lemma}\label{diexpander}
Let $k_0,n_0$ be positive integers and $\varepsilon, d, \eta, \nu, \tau$ be positive constants such that $$1/n_0 \ll 1/k_0, \varepsilon \ll d \ll \nu, \tau, \eta < 1.$$ Suppose that $G$ is a digraph on $n \geq n_0$ vertices with $\delta^0(G) \geq \eta n$ such that $G$ is a robust $(\nu, \tau)$-diexpander. Let $R$ be the reduced digraph of $G$ with parameters $\varepsilon, d$ and $k_0$. Then $\delta^0(R) \geq \eta |R| /2$ and $R$ is a robust $(\nu/2, 2\tau)$-diexpander.
\end{lemma}

Recall that if $G$ is a digraph, we define $(A,B)_G$ to be the oriented graph with all edges from $A$ to $B$. We will say $(A,B)_G$ is $(\varepsilon,d)^*$-superregular if the underlying undirected graph is $(\varepsilon,d)$-superregular and, additionally, $\varepsilon$-regular. We now state a result which tells that by removing a small number of vertices, we are able to ensure that all pairs of clusters corresponding to edges of a subgraph of $R$ with bounded maximum degree are superregular and that the other pairs remain regular. It is very similar to \Cref{superreg} for undirected graphs where the subgraph of interest was a path.

\begin{lemma}\label{lem:subgraphsuperreg}
Let $0<\varepsilon \ll d, 1/\Delta$ and let $R$ be the reduced digraph of $G$ as given by \Cref{direg}. Suppose that $S \subseteq R$ with $\Delta(S) = \Delta$. Then we can move exactly $2 \Delta \varepsilon |V_i|$ vertices from each cluster $V_i$ into $V_0$ so that each pair $(V_i,V_j)$ corresponding to an edge of $S$ becomes $(2 \varepsilon, d/2)^*$-superregular and each pair of clusters corresponding to an edge of $R \setminus S$ is $2\varepsilon$-regular with density at least $d-\varepsilon$.
\end{lemma}

\subsection{Consequences of the Blow-up Lemma}

We will also require some consequences of the Blow-up Lemma which allow us to embed a series of paths into our graph. The next result follows directly from \Cref{rpartblowup}.

\begin{lemma}\label{lemma11}
For every $d \in (0,1)$ and $p \geq 4$, there exists $\varepsilon_0 >0$ such that the following holds for all $0< \varepsilon \leq \varepsilon_0$. Let $U_1, \ldots, U_p$ be pairwise disjoint sets of size $m$. Suppose that $G$ is a graph on $U_1 \cup \ldots \cup U_p$ such that, for each $1 \leq i <p$, the pair $(U_i, U_{i+1})$ is $(\varepsilon, d)^*$-superregular. Suppose that $f: U_1 \rightarrow U_p$ is a bijective map. Then there are $m$ vertex disjoint paths from $U_1$ to $U_p$ so that for every $x \in U_1$ the path starting at $x$ ends at $f(x) \in U_p$.
\end{lemma}

\begin{proof}
Let $F$ be a cycle of length $p-1$. Apply the Blow-up Lemma, \Cref{rpartblowup}, with the parameters $d$, $\Delta = 2$ and $p-1$ to obtain $\varepsilon_0$. Suppose that $G$ is a graph as in the lemma and $f: U_1 \rightarrow U_p$ is a bijective map.

Merge the sets $U_1$ and $U_p$ to create a new set $U^* = \{(x,f(x)): x\in U_1\}$ where we associate each of the vertices $x \in U_1$ with $f(x) \in U_p$ to form a vertex $(x,f(x))$. Consider the graph $G'$ on $U^* \cup U_2 \cup \ldots \cup U_{p-1}$ such that
for each $i \geq 2$, $x \in U_i$ and $y\in U_{i+1}$, $xy \in E(G')$ if $xy \in E(G)$. We also add all edges $(x,f(x))y$ whenever $xy \in E(G)$ or $f(x)y \in E(G)$. Observe that the pairs $(U^*,U_2)_{G'}$, $(U_{p-1},U^*)_{G'}$ and $(U_i,U_{i+1})_{G'}$ for all $i\geq 2$ are	$(\varepsilon,d)^*$-superregular. 

Let $F' = F^m$, a blow-up of the cycle $F$. Then $G'$ is a spanning subgraph of $F'$. Let $H$ be a graph consisting of $m$ disjoint cycles of length $p-1$. Then $H$ is a subgraph of $F'$ with $\Delta(H)=2$, so we are able to apply \Cref{rpartblowup} to see that $G'$ contains a copy of $H$. This copy of $H$ corresponds to $m$ disjoint paths in $G$ such that the path starting at $x$ in $U_1$ finishes at $f(x)$.
\end{proof}

We will use \Cref{lemma11} to prove the following consequence of the Blow-up Lemma.

\begin{lemma}\label{lemma10}
Suppose that $$0<1/m \ll \varepsilon \ll d \ll 1$$ and the following hold:
\begin{itemize}
\item $G$ is a digraph on $U_1 \cup \ldots \cup U_k$, where $k \geq 6$ and $U_1, \ldots, U_k$ are pairwise disjoint sets of size $m$ such that each $(U_i, U_{i+1})_G$ is $(\varepsilon, d)^*$-superregular (by convention $U_{k+1}=U_1$);
\item $H$ is a vertex disjoint union of (not necessarily directed) paths of length at least $3$ on $A_1 \cup \ldots \cup A_k$, where $A_1, \ldots, A_k$ are pairwise disjoint sets of vertices with $m_i \coloneqq |A_i|$ satisfying $(1-\varepsilon)m \leq m_i \leq m$ and such that, for each $1 \leq i \leq k$, every edge in $H$ leaving $A_i$ ends in $A_{i+1}$;
\item $S_1 \subseteq U_1, \ldots, S_k \subseteq U_k$ are sets of size $|S_i|=m_i$;
\item For each path $P$ of $H$, we are given vertices $x_P, y_P \in V(G)$ such that if the initial vertex $a_P$ of $P$ lies in $A_i$ then $x_P \in S_i$ and if the final vertex $b_P$ of $P$ lies in $A_j$ then $y_P \in S_j$ and the vertices $x_P, y_P$ over all paths $P$ in $H$ are distinct.
\end{itemize}
Then there is an embedding of $H$ into $G_S \coloneqq G[\bigcup_{i=1}^k S_i]$ in which every path $P$ of $H$ is mapped to a path that starts at $x_P$ and ends at $y_P$.
\end{lemma}

In order to prove this result, we will also require the following lemma. Suppose that $(A,B)$ is a superregular pair with $|A|=|B|=m$. Then this lemma tells us that, with high probability, all new pairs created by a random partition of $(A,B)$ are also superregular. By \emph{high probability}, we mean that the probability tends to $1$ as $m$ tends to infinity.

\begin{lemma}\label{lemma12}
Let $0<\varepsilon\ll \theta<d<1/2$ and $k \geq 2$. For $1\leq i \leq k$, suppose that $a_i,b_i>\theta$ are constants satisfying $\sum_{i=1}^k a_i = \sum_{i=1}^k b_i = 1$. Let $G=(A,B)$ be an $(\varepsilon, d)^*$-superregular pair with $|A|=|B|\eqqcolon m$ sufficiently large. If $$A=A_1 \cup \ldots \cup A_k \text{ and } B = B_1 \cup \ldots \cup B_k$$ are partitions chosen uniformly at random with $|A_i| = a_i m$ and $|B_i| = b_i m$ for $1 \leq i \leq r$ then with high probability $(A_i,B_i)$ is $(\theta^{-1}\varepsilon,d/2)^*$-superregular for every $1 \leq i,j \leq k$.
\end{lemma}

We give a short sketch of the proof of \Cref{lemma12}. Let us choose partitions of $A$ and $B$ uniformly at random with $|A_i| = a_im$ and $|B_i|=b_im$. By \Cref{prop:subsets}, all pairs $(A_i, B_i)$ will be $\theta^{-1}\varepsilon$-regular so we just need to check that with positive probability, each vertex in $A_i$ has at least $db_im/2$ neighbours in $B_i$ and each vertex in $B_i$ has at least $da_im/2$ neighbours in $A_i$. Fix $i$ and let $x \in A_i$. The variable $|N_G(x) \cap B_i|$ has hypergeometric distribution. Using that $G$ is $(\varepsilon, d)^*$-superregular and a type of Chernoff bound similar to those we will introduce in \Cref{sub:approx}, we can show that only with small probability does $x$ have significantly fewer than the expected number of neighbours in $B_i$. We do the same for each vertex in $B_i$. A union bound gives that the probability that each pair $(A_i,B_i)$ is not $(\theta^{-1}\varepsilon,d/2)^*$-superregular is strictly less than one. Hence we are able to find a partition which satisfies the desired properties.

We will now apply these results to prove \Cref{lemma10}.

\begin{proof}[Proof of \Cref{lemma10}.]
Label the paths of $H$ by $P_1, P_2, \ldots, P_p$. Then divide each path $P_i$ into subpaths $P_{i,j}$ of length $3$, $4$ or $5$ so that $$P_i = P_{i,1}P_{i,2}\ldots P_{i,q_i}.$$ For each $1\leq i \leq p$, $1\leq j \leq q_i$, let $a_{i,j}$ denote the initial vertex of the path $P_{i,j}$ and $b_{i,j}$ the final vertex. We note that $a_{i,j}=b_{i,j-1}$ for all $2 \leq j \leq q_i$. For each $1 \leq s \leq k$, let $E_s$ be the set of all $a_{i,j}$ in $A_s$ and $F_s$ be the set of all $b_{i,j}$ in $A_s$. Choose distinct vertices $x_{i,j} \in S_s$ for each $a_{i,j} \in E_s$ and choose also distinct $y_{i,j} \in S_s$ for each $b_{i,j} \in F_s$ so that $x_{i,1}=x_{P_i}$ and $y_{i,q_i} = y_{P_i}$ and whenever $a_{i,j}=b_{i,j-1}$ we have that $x_{i,j}=y_{i,j-1}$. We will look for an embedding of $H$ which maps each path $P_{i,j}$ to a path from $x_{i,j}$ to $y_{i,j}$ in $G_S$.

We will describe the direction of the edges which make up each $P_{i,j}$, writing:
$$\mathtt{f} \text{ for an edge from some } A_\ell \text{ to } A_{\ell+1},$$
$$\mathtt{b} \text{ for an edge from some } A_\ell \text{ to } A_{\ell-1}.$$ Then we can encode each path $P_{i,j}$ using the letters $\mathtt{f}$ and $\mathtt{b}$, for example see \Cref{fig:pathencoding}. There are $2^3$ possible encodings for a path of length $3$, $2^4$ for a path of length $4$ and $2^5$ for a path of length $5$. In total, we obtain at most $2^3+2^4+2^5 = 56$ different encodings.

\begin{figure}[h]
\centering
\includegraphics[scale=0.35]{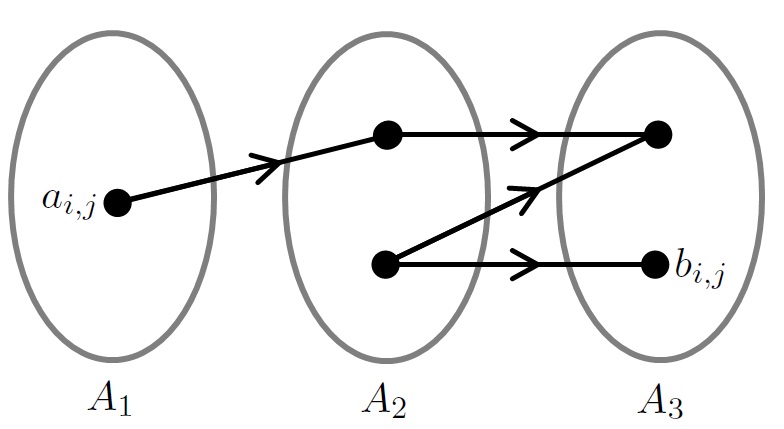}
\caption{We would encode the path $P_{i,j}$ shown by $\mathtt{ffbf}$.}\label{fig:pathencoding}
\end{figure}

For each $3 \leq \ell \leq 5$ we can describe each encoding of a path of length $\ell$ by a function $$t: \{0,1, \ldots, \ell\} \rightarrow \{-\ell, -\ell+1, \ldots, \ell\},$$ where $t(i)$ gives the index of the current cluster relative to the starting cluster. For example, in \Cref{fig:pathencoding}, we start in $A_1 = A_{1+0}$, then move to $A_{1+1}$, then to $A_{1+2}$, back to $A_{1+1}$ and finish in $A_{1+2}$. So we would encode this by $t:(0,1,2,3,4) \rightarrow (0,1,2,1,2)$. Then for each $1 \leq i \leq k$, $3 \leq \ell \leq 5$ and $t: \{0,1,\ldots, \ell\} \rightarrow \{-\ell, -\ell+1, \ldots, \ell\}$, we let $\mathcal{P}_{i,t}$ consist of all paths $P_{i,j}$ of length $\ell$, starting in $A_i$ and with encoding described by $t$. There are $56k$ sets $\mathcal{P}_{i,t}$ in total. Let
$$\mathcal{Q}_1 = \{\mathcal{P}_{i,t}: |\mathcal{P}_{i,t}|\leq d^2m\} \text{ and } \mathcal{Q}_2 = \{\mathcal{P}_{i,t}: |\mathcal{P}_{i,t}|> d^2m\}.$$

We will begin by greedily embedding each $\mathcal{P}_{i,t} \in \mathcal{Q}_1$. Note that  for all $1\leq i\leq k$, $|U_i \setminus S_i| \leq \varepsilon m$, so by \Cref{prop:subsets} and since $d-\varepsilon \geq d/2$, $(S_i,S_{i+1})$ is $(2\varepsilon, d/2)^*$-superregular. Choose each set $\mathcal{P}_{i,t} \in \mathcal{Q}_1$ in turn. Select any set of $|\mathcal{P}_{i,t}|$ vertices in $S_i$ as starting vertices. Next choose distinct neighbours of each of these vertices from either $S_{i-1}$ or $S_{i+1}$, as appropriate, and continue in this way until all paths in $\mathcal{P}_{i,t}$ have been constructed. Repeat this process for each $\mathcal{P}_{i,t} \in \mathcal{Q}_1$, each time selecting vertices that have not already been chosen. Now, the paths $P_{i,j}$ have length at most $5$ and all edges go from $S_i$ to $S_{i+1}$ so each $S_i$ can be used by at most $11 \times 56$ of the sets $\mathcal{P}_{i,t}$. Since each path uses at most $6$ vertices and we consider only sets $\mathcal{P}_{i,t}$ containing at most $d^2m$ paths, at any stage in this process, at most $6 \times 11 \times 56 \times d^2m \leq dm/4$ vertices in each $S_i$ have already been used. Then, by \Cref{prop:subsets} and noting that $d-\varepsilon -d/4\geq d/4$, the graphs $(S_i,S_{i+1})$ are still $(4 \varepsilon, d/4)^*$-superregular at any stage in this process, allowing us to construct the paths in this way.

We now consider each of the large sets $\mathcal{P}_{i,t} \in \mathcal{Q}_2$. Randomly split all of the remaining vertices so that we have sets $$S^0_{i,t} \subseteq S_{i+t(0) = i}, \; S^1_{i,t} \subseteq S_{i+t(1)},\; \ldots,\; S^\ell_{i,t} \subseteq S_{i+t(\ell)},$$ for each $\mathcal{P}_{i,t} \in \mathcal{Q}_2$, each of size $|\mathcal{P}_{i,t}|$. Provided that $m$ is sufficiently large, we can then apply \Cref{lemma12} to see that for all $\mathcal{P}_{i,t} \in \mathcal{Q}_2$ and for all $0 \leq r \leq \ell-1$, the pair $(S^r_{i,t}, S^{r+1}_{i,t})$ if $t(r+1) >t(r)$ or $(S^{r+1}_{i,t}, S^r_{i,t})$ if $t(r+1)<t(r)$ is $(4\varepsilon/d^2, d/8)^*$-superregular with high probability. So, we are able to choose a partition of the remaining vertices satisfying this property. Then, we are able to apply \Cref{lemma11} in order to embed each $\mathcal{P}_{i,t} \in \mathcal{Q}_2$ into the designated sets. Hence, we obtain the desired embedding of each of the paths of $H$ in $G$.
\end{proof} 

\subsection{An Approximate Embedding Lemma}\label{sub:approx}

In this section, we prove a result which will allow us to embed a collection of subpaths of the desired Hamilton cycle into the reduced digraph in such a way that each of the clusters receives a similar number of vertices. The proof will use probabilistic arguments, in particular, the following Chernoff-type bounds will be useful.

\begin{theorem}[Chernoff Bound 1]\label{chernoff1}
Suppose that $X_1, X_2, \ldots, X_n$ are independent, random variables with $\mathbb{P}(X_i = 1)=p$ and $\mathbb{P}(X_i = 0)=1-p$. Let $X= \sum_{i=1}^n X_i$. Then for any $0 \leq \lambda \leq np$,
$$\mathbb{P}(|X-\mathbb{E}(X)| > \lambda) < 2e^{-\lambda^2/3\mathbb{E}(X)}.$$
\end{theorem}

\begin{theorem}[Chernoff Bound 2]\label{chernoff2}
Let $X$ be a random variable determined by $n$  independent trials $X_1, X_2, \ldots , X_n$ such that changing the outcome of any one trial can affect $X$ by at most $c$. Then for any $\lambda \geq 0$,
$$\mathbb{P}(|X-\mathbb{E}(X)| > \lambda) < 2e^{-\lambda^2/2c^2n}.$$
\end{theorem}

We use \Cref{chernoff2} to prove the following lemma, which will allow us to find such an embedding in the reduced digraph. 

\begin{lemma}\label{lemma16}
Let $R$ be a digraph on $k$ vertices and suppose that $F = V_1 \ldots V_k$ is a Hamilton cycle in $R$. Let $\mathcal{P} = \{P_1, P_2, \ldots, P_s\}$ be a collection of arbitrarily oriented paths on $t$ vertices and $\mathcal{Q}$ be a collection of pairwise disjoint oriented subpaths of the $P_i$. Then, for any $\gamma >0$ and sufficiently large $s$, there exists a map $\phi: [s] \rightarrow V(R)$ such that if the paths in $\mathcal{P}$ are greedily embedded around $F$ with the embedding of each $P_i$ starting at $\phi(i)$, the following hold. Define $a(i)$ to be the number of vertices in $\bigcup \mathcal{P}$ assigned to $V_i$ by this embedding and $n(i, \mathcal{Q})$ to be the number of subpaths in $\mathcal{Q}$ starting at $V_i$. Then for all $V_i \in V(R)$
\begin{equation}\label{eq:1lemma16}
\left|a(i) - \frac{st}{k}\right| \leq \gamma st
\end{equation}
and
\begin{equation}\label{eq:2lemma16}
\left|n(i, \mathcal{Q}) - \frac{|\mathcal{Q}|}{k}\right| \leq \gamma st.
\end{equation}
\end{lemma}

In the statement of this lemma, \textquoteleft greedily embedding' a path $P$ around $F$ means that we start from the specified initial vertex $V_i$ and then choose the next vertex from $V_{i-1}$ or $V_{i+1}$ according to the orientation of the edge. We continue in this way until all vertices in the path have been embedded.

\begin{proof}
We will pick each $\phi(i)$ independently and uniformly at random. The paths in $\mathcal{P}$ all consist of $t$ vertices so each assignment of a path can affect the number of vertices assigned to any vertex in $R$ by at most $t$. We have that $\mathbb{E}(a(i))= st/k$ so we may apply \Cref{chernoff2} to see that
$$\mathbb{P}(|a(i)-st/k|>\gamma st) \leq 2e^{-(\gamma st)^2/2t^2s} = 2e^{-\gamma^2s/2}<1/2k$$
for $1/s \ll 1/k$. We also have that $\mathbb{E}(n(i, \mathcal{Q})) = |\mathcal{Q}|/k$ and so we may again apply \Cref{chernoff2} to obtain
$$\mathbb{P}(|n(i, \mathcal{Q})-|\mathcal{Q}|/k|>\gamma st) \leq 2e^{-(\gamma st)^2/2t^2s} = 2e^{-\gamma^2s/2}<1/2k$$
for $1/s \ll 1/k$.

We obtain that $$\mathbb{P}((|a(i)-st/k|>\gamma st)  \text{ \& } (|n(i, \mathcal{Q})-|\mathcal{Q}|/k|>\gamma s t)) <1/k$$ whenever $s$ is sufficiently large. Taking the sum of these probabilities over all clusters, we find that the probability that some $V_i \in E(R)$ has either $a(i)$ or $n(i,\mathcal{Q})$ far from the expected values is less than $1$. Then, with positive probability, the map $\phi$ constructed in this way satisfies properties \eqref{eq:1lemma16} and \eqref{eq:2lemma16} of the lemma and hence such a map exists.
\end{proof}

We will also frequently make use of the following lemma which will allow us to find a path isomorphic to any orientation of a short path between any given pair of vertices in the reduced graph.

\begin{lemma}\label{lem:isopath}
Let $1/n \ll \nu \leq \tau \ll \eta \ll 1$. Suppose that $G$ is a digraph on $n$ vertices and that $G$ is a robust $(\nu,\tau)$-diexpander with $\delta^0(G) \geq \eta n$.
Let $$\lceil 2/\nu \rceil \leq k \leq \nu n/4.$$ Let $x,y \in V(G)$ be distinct vertices. Then if $P$ is any orientation of a path of length $k$, there is a path in $G$ from $x$ to $y$ isomorphic to $P$.
\end{lemma}

\begin{proof}
Let us first show that we can find any (arbitrarily oriented) path $P$ of length $\ell \coloneqq \lceil 1/\nu \rceil$ between any two vertices $x,y \in V(G)$. If the first edge of $P$ is forward oriented then take $A_1 \coloneqq N^+(x)$ (if backward oriented then take $N^-(x)$ instead). We know that $|A_1|\geq \delta^0(G) \geq \eta n> \tau n$. Let $A_2=RN^+_{\nu,G}(A_1)$ or $RN^-_{\nu,G}(A_1)$ according to the orientation of the next edge. By the robust diexpansion property, we have that $|A_2| \geq |A_1| + \nu n \geq \eta n+\nu n$. Continue in this way, each time letting $A_i = RN^+_{\nu,G}(A_{i-1})$ or $RN^-_{\nu,G}(A_{i-1})$ as appropriate. (If at any stage $|A_i|>(1-\tau )n$, choose a subset of size $(1-\tau)n$ as $A_i$.)

Then, after $\ell-1$ steps, 
$$|A_{\ell-1}| \geq \eta n+(\ell-2)\nu n >(1-(\eta-\nu))n$$
and, since $\delta^0(G) \geq \eta n$, we have that $|N^+(y) \cap A_{\ell-1}| \geq \nu n$ and $|N^-(y) \cap A_{\ell-1}| \geq \nu n$. This means that from $y$ we have a choice of at least $\nu n$ neighbours in $A_{\ell-1}$ for the vertex in $P$ preceding $y$. Each of these vertices has at least $\nu n$ suitable neighbours in $A_{\ell-2}$ and we continue in this way to see that there are at least $(\nu n)^{\ell-1}$ walks from $x$ to $y$ with the same orientation as $P$. However, some of these walks may use some vertex more than once. At most $\ell^2n^{\ell-2}$ of these walks are not paths (consider all possible orderings of $\ell-1$ vertices, the middle portion of the walk, with at least one repeated vertex). But, since $n$ is sufficiently large, we have that
$$(\nu n)^{\ell-1}>\ell^2n^{\ell-2},$$
so at least one of the walks must be a path. Hence, we can find a path of length $\ell$ between $x$ and $y$ which is isomorphic to $P$.

Let us now suppose that $P$ is any orientation of a path of length $k$, where $\lceil 2/\nu \rceil \leq k \leq \nu n/4$. (The lower bound for $k$ is greater than $\ell$ to account for the change in parameters of the robust diexpander when we greedily embed the first portion of the path.) Starting at $x$, embed the first $k-\ell+1$ vertices greedily, letting $z$ denote the last vertex embedded, this is possible since $\delta^0(G) \geq \eta n$. Remove all of these vertices, except for $z$ from the graph. We have removed fewer than $\nu n/4$ vertices so, by \Cref{prop:removingrobexp}, the new graph $G'$ remains a robust $(\nu/2, 2\tau)$-diexpander with $\delta^0(G') \geq \eta n/2$. We can then apply the previous result to this graph to find a path in $G'$ from $z$ to $y$ which is isomorphic to the remainder of the path $P$.
\end{proof}

\subsection{Skewed Traverses and Shifted Walks}\label{sub:skewed}

We will introduce one final technique before commencing the proof of \Cref{anyorient2}. The two main problems that we will face when using an embedding in the reduced graph to find an embedding in $G$ will be incorporating the exceptional vertices and ensuring that each cluster has been assigned exactly $m$ vertices. This is when we will use skewed traverses and shifted walks.

Throughout this section we will assume that $F = V_1 V_2 \ldots V_k$ is a Hamilton cycle, with standard orientation, in the reduced digraph $R$ where each vertex corresponds to a cluster of size $m$. We define a further graph, $R^*$.  Let $c >0$. Then we obtain the graph $R^*$ by adding all exceptional vertices $v \in V_0$ to $V(R)$ and edges $vV_i$ if $|N^+_G(v) \cap V_i| \geq cm$ and $V_iv$ if $|N^-_G(v) \cap V_i| \geq cm$.

Suppose that $C$ is an arbitrarily oriented cycle and let $W$ be an assignment of $V(C)$ to $V(R^*)$ respecting edges. We will write $a(i)$ for the number of vertices assigned to the cluster $V_i$. 

\begin{definition}
Let $W$ be an assignment of $C$ to $R^*$. We say that $W$ is $\gamma$-\emph{balanced} if $\max_i|a(i)-m| \leq \gamma n$ and \emph{balanced} if $a(i) = m$ for all $1\leq i \leq k$.

We say that the assignment $(\gamma, \mu)$\emph{-corresponds} to $C$ if:
\begin{itemize}
\item $W$ is $\gamma$-balanced;
\item Each $v \in V_0$ has exactly one vertex of $C$ assigned to it;
\item In every $V_i\in V(R)$, at least $m-\mu n$ of the vertices of $C$ assigned to $V_i$ have both of their neighbours assigned to $V_{i-1} \cup V_{i+1}$. 
\end{itemize}
We will write that an assignment $\mu$\emph{-corresponds} to $C$ if it $(0, \mu)$-corresponds to $C$.
\end{definition}

Let us now define the skewed traverses and shifted walks which will help us adapt an assignment in order to find a closed walk in the graph $R^*$ which corresponds to the cycle $C$.

\begin{definition}
Let $A,B \in V(R)$. A \emph{skewed traverse}, $T(A,B)$, is a collection of edges of the form
$$T(A,B) = AV_{i_1}, V_{i_1-1}V_{i_2}, V_{i_2-1}V_{i_3}, \ldots, V_{i_t-1}B.$$ 
\end{definition}

The length of a skewed traverse is one less than the number of edges it contains, so in this case, $T(A,B)$ has length $t$. We will always assume that a skewed traverse has minimal length.

\begin{figure}[h]
\centering
\includegraphics[scale=0.36]{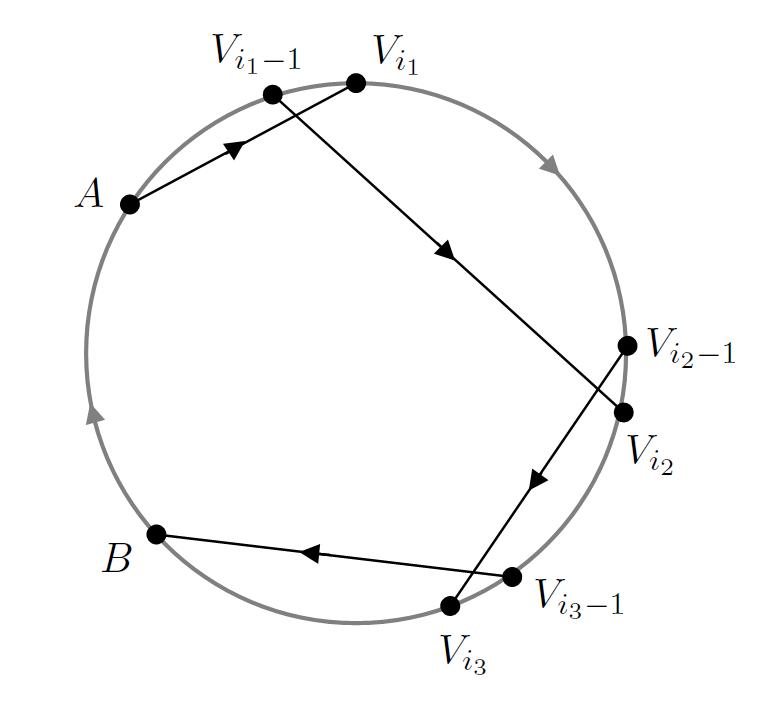}
\caption{A skewed traverse, $T(A,B)$, of length $3$.}\label{fig:skewed}
\end{figure}

Suppose that we have an assignment, $W,$ of $C$ to $R^*$, we can think of $W$ as a walk in $R^*$. Suppose that $W$ is not balanced. Then there are clusters $V_i$ and $V_j$ such that $a(i)>m$ and $a(j)<m$. Suppose further that each cluster $V_p$ has been assigned the initial vertex of many neutral pairs of $C$, and that many of these assignments have the form $V_pV_{p+1}V_p$.  If we have a skewed traverse
$$T(V_{i-1},V_j) \coloneqq V_{i-1}V_{i_1}, V_{i_1-1}V_{i_2}, V_{i_2-1}V_{i_3}, \ldots, V_{i_t-1}V_{j}$$
then we can replace sections of the walk $W$ which have been assigned neutral pairs and look like $$V_{i-1}V_iV_{i-1}, V_{i_1-1}V_{i_1}V_{i_1-1}, V_{i_2-1}V_{i_2}V_{i_2-1},\ldots, V_{i_t-1}V_{i_t}V_{i_t-1}$$ by the edges from $T(V_{i-1},V_j)$: $$V_{i-1}V_{i_1}V_{i-1}, V_{i_1-1}V_{i_2}V_{i_1-1}, V_{i_2-1}V_{i_3}V_{i_2-1}, \ldots, V_{i_t-1}V_{j}V_{i_t-1}.$$
This process:
\begin{itemize}
\item decreases $a(i)$ by one;
\item increases $a(j)$ by one and
\item does not change $a(p)$ for any $p \neq i,j$. 
\end{itemize}
So, if $C$ has many neutral pairs, we will be able to use this method to obtain a balanced assignment. 

We also define a shifted walk in the graph $R$. This definition is almost identical to the one given in \Cref{def:shifted}, but, to be consistent with our definition of a skewed traverse, we now include an additional initial edge.

\begin{definition}
Let $A, B \in V(R)$. A \emph{shifted walk}, $W(A,B)$, is a walk of the form
$$W(A,B) = AV_{i_1}FV_{i_1-1}V_{i_2}FV_{i_2-1} \ldots V_{i_t}F V_{i_t-1}B.$$
\end{definition}

We will say that $W(A,B)$ \emph{traverses} $F$ $t$ times. There is an obvious similarity between $T(A,B)$ and $W(A,B)$ and we observe that if we can find a skewed traverse $T(A,B)$ then we an also find a shifted walk from $A$ to $B$. We will again assume that $W(A,B)$ traverses $F$ as few times as possible which gives us that each vertex appears at most once as an entry cluster in $W(A,B)$. The main fact about shifted walks which we will use is that $W(A,B)\setminus \{A,B\}$ visits every vertex in $R$ the same number of times.

Let us again assume that we have an assignment, $W$, of $C$ to $R^*$. So there are clusters $V_i$ and $V_j$ such that $a(i)>m$ and $a(j)<m$. Suppose that most edges in $W$ correspond to an edge in $F$. If $C$ contains few neutral pairs then the assignment $W$ must contain some long, consistently oriented subwalks. Then we can replace a section of the assignment comprising of $\ell$ copies of $F$ by
$$W(V_{i-1}, V_j)W(V_j, V_{i+1})FV_{i-1}F \ldots FV_{i-1}$$
of the same length, $\ell k$, so that the total number of vertices assigned is the same. Noting that $W(V_{i-1}, V_j)\setminus \{V_{i-1}, V_j\}$ and $W(V_j,V_{i+1})\setminus \{V_j,V_{i+1}\}$ visit all clusters of $R$ the same number of times we see that the first section of this walk $W(V_{i-1}, V_j)W(V_j, V_{i+1})FV_{i-1}$ has length divisible by $k$. Adding as many further copies of $F$ as necessary, we see that we can indeed find a walk of the required form which starts and finishes at $V_{i-1}$. Making this substitution decreases $a(i)$ by one; increases $a(j)$ by one and leaves $a(p)$ unchanged for all $p\neq i, j$.

\begin{proposition}\label{prop:shortskewed}
Let $0<\nu\leq \tau \ll \eta <1$ and let $R$ be a digraph on $k$ vertices with $\delta^0(G)\geq \eta k$. Suppose that $R$ is a robust $(\nu, \tau)$-outexpander and $F = V_1V_2\ldots V_k$ is a directed Hamilton cycle in $R$. Define $r \coloneqq \lceil 1/\nu \rceil$. Then, for any pair of distinct vertices $A, B \in V(R)$ there exists:
\begin{itemize}
\item A skewed traverse $T(A,B)$ of length at most $r$ and
\item A shifted walk $W(A,B)$ traversing $F$ at most $r$ times.
\end{itemize}
\end{proposition}

\begin{proof}
Let $A=U_0$ and for each $i \geq 1$ let $U_i$ be the set of clusters that can be reached from $A$ by a skewed traverse of length at most ${i-1}$. We define 
$$U^-_i \coloneqq \{V_i: V_{i+1} \in U_i\},$$ the set of predecessors of the clusters in $U_i$. Note that $|U^-_i|= |U_i|$ for all $i\geq 1$.

We first note that $|U_1| \geq \delta^0(R) \geq \eta k$. If $|U^-_1|\leq (1-\tau)k$ then we can use that $R$ is a robust $(\nu, \tau)$-outexpander to see that $$|U_2| \geq |U_1|+\nu k \geq (\eta+\nu)k.$$ Continuing in this way we see that, as long as $|U^-_t| \leq (1-\tau)k$, we can reach $$|U_{t+1}| \geq (\eta + t\nu)k $$ clusters by  skewed traverse of length at most $t$.

Then $|U_{r}| > (1-\eta)k$ and, since $d^-_R(B) \geq \delta^0(R) \geq \eta k$, we have $N^-_R(B) \cap U^-_{r} \neq \emptyset$. Therefore there is a skewed traverse from $A$ to $B$ of length at most $r$. We can use this skewed traverse to find a shifted walk in $R$ of traversing at most $r$ cycles by going once around $F$ after each edge.
\end{proof}

\section{The Two Cases}

We will separate our argument into two cases. We first look at the case when our desired orientation of a Hamilton cycle $C$ is far from the standard orientation. In this case we have many changes of direction on the cycle and hence many neutral pairs. The second case deals with the situation when $C$ is close to the standard orientation, having long consistently oriented subpaths and few neutral pairs. Let us suppose that $C$ has $n(C) = \lambda n$ neutral pairs and let $\mathcal{Q}$ denote a maximal collection of neutral pairs all at a distance at least $3$ from each other. We will define a hierarchy of positive constants
$$0 < \varepsilon_1 \ll \varepsilon_2 \ll \varepsilon_3 \ll \varepsilon_4 \ll \varepsilon_5 \ll \varepsilon_6 \ll \varepsilon_7 \ll \nu.$$
Then, if:
\begin{itemize}
\item $\lambda > \varepsilon_4$, let $\varepsilon \coloneqq \varepsilon_3$, $\varepsilon_A \coloneqq \varepsilon_2$ and $\varepsilon^* \coloneqq \varepsilon_1$. \hfill (Case 1)
\item $\lambda \leq \varepsilon_4$, let $\varepsilon \coloneqq \varepsilon_7$, $\varepsilon_A \coloneqq \varepsilon_6$ and $\varepsilon^* \coloneqq \varepsilon_5$. \hfill (Case 2)
\end{itemize}

In either case, we must begin by partitioning the graph. We will  show that we can split our graph $G$ roughly in half, so that all vertices have around the expected number of neighbours in each half. We will use this partition to gain some control over the number of exceptional vertices obtained when we apply the Diregularity Lemma. If we were just to apply the lemma to the entire graph, we would obtain an exceptional set of size at most $\varepsilon n$	which could greatly exceed the cluster size. By splitting the graph first, we can apply the Diregularity Lemma to partition the vertices of the first subgraph $G_1$ and obtain an exceptional set. We then apply the lemma to the subgraph $G_2$ induced by the remaining vertices of the graph together with this exceptional set, this time with a much smaller \textquoteleft$\varepsilon$' in order to an exceptional set which is much smaller when compared to the size of the clusters in $G_1$.

We let $d, d', \varepsilon', \tau'$ and $\nu'$ be positive constants and $M'_A$, $M'_B$ and $k_0$ be positive integers satisfying
$$\varepsilon^* \ll \frac{1}{M'_A} \ll \varepsilon_A \ll \frac{1}{M'_B} \ll \varepsilon \ll d  \ll\nu'\ll \frac{1}{k_0} \ll \varepsilon' \ll d' \ll \nu \leq \tau \ll \tau' \ll \eta.$$
Note that in Case 1 we have that $\varepsilon \ll \lambda$ and in Case 2 we have $\lambda \ll \varepsilon^*$.

\begin{lemma}
There exists a subset $A \subseteq V(G)$ such that:
\begin{enumerate}[$(i)$]
\item $\left(\frac{1}{2}-\varepsilon\right)n \leq |A| \leq \left(\frac{1}{2}+\varepsilon\right)n$;
\item $\frac{d^+(x)}{n} - \frac{\eta}{10} \leq \frac{|N^+_A(x)|}{|A|} \leq \frac{d^+(x)}{n} + \frac{\eta}{10}$, for all $x \in V(G)$;
\item $\frac{d^-(x)}{n} - \frac{\eta}{10} \leq \frac{|N^-_A(x)|}{|A|} \leq \frac{d^-(x)}{n} + \frac{\eta}{10}$, for all $x \in V(G)$;
\item The graphs $G[A]$ and $G\setminus A$ are robust $(\nu',\tau')$-diexpanders.
\end{enumerate}
\end{lemma}

\begin{proof}
Apply the Diregularity Lemma, \Cref{direg}, with parameters $\varepsilon'$, $d'$ and $k_0$ to the graph $G$, to obtain a partition into clusters $V_1, \ldots, V_k$ of size $m$ and an exceptional set $V_0$.

We will show that we can find a set $A$ which satisfies $(i)$ and such that $\left| |N^+_A(x)|-d^+(x)/2\right| \leq \eta n/20$  and $\left| |N^-_A(x)|-d^-(x)/2\right| \leq \eta n/20$ for all $x \in V(G)$. Then this set $A$ satisfies property $(ii)$ as
$$\frac{n}{(1/2+\varepsilon)n}\left(\frac{d^+(x)}{2n} - \frac{\eta}{20}\right) \leq \frac{|N^+_A(x)|}{|A|} \leq \frac{n}{(1/2-\varepsilon)n}\left(\frac{d^+(x)}{2n} + \frac{\eta}{20}\right).$$
Property $(iii)$ follows similarly.

Let $V(G) = \{x_1, \ldots x_n\}$. Let $A$ be obtained by including each $x_i$ with probability $1/2$, independently of all $x_j$. For $1\leq i \leq n$, let $X_i$ be the event that $x_i$ is in $A$. Then $\mathbb{P}(X_i =1) = 1/2$, $X = \sum_{i=1}^n X_i = |A|$ and $\mathbb{E}(X) = n/2$. We apply the Chernoff bound of \Cref{chernoff1} to see that
\begin{equation}\label{eq:A1}
\mathbb{P}(||A|-n/2| > \varepsilon n) < 2e^{-2\varepsilon^2n/3}\leq \frac{1}{4n}
\end{equation}
for sufficiently large $n$.

Now consider the degree of each vertex of $G$ inside $A$. Fix some vertex $x \in V(G)$. For each $x_i \in N^+(x)$, let $X_{x,i}$ be the event that $x_i$ is in $N^+_A(x)$. Let $X_x= \sum_{x_i \in N^+(x)} X_{x,i}$. Then $X_x=|N^+_A(x)|$ and $\mathbb{E}(X_x) = d^+(x)/2$. We apply \Cref{chernoff1} to get
$$\mathbb{P}\left(\left| |N^+_A(x)|-\frac{d^+(x)}{2}\right| > \frac{\eta n}{20}\right) < 2e^{-\frac{2\eta^2n}{1200(3/8+\eta)}}\leq 2e^{-n}.$$
We can then sum these probabilities over all $x \in V(G)$ to see that
\begin{equation}\label{eq:A2}
\mathbb{P}\left(\left| |N^+_A(x)|-\frac{d^+(x)}{2}\right| > \frac{\eta n}{20} \; \text{ for all } x \in V(G)\right) <2ne^{-n}\leq \frac{1}{4n}
\end{equation}
for sufficiently large $n$.
In the same way, we show that 
\begin{equation}\label{eq:A3}
\mathbb{P}\left(\left| |N^-_A(x)|-\frac{d^-(x)}{2}\right| > \frac{\eta n}{20} \; \text{ for all } x \in V(G)\right) <2ne^{-n} \leq \frac{1}{4n}.
\end{equation}

Finally, we consider the robust diexpansion property. Recall that we have chosen an $\varepsilon'$-regular partition of $V(G)$ into $k$ clusters $V_1, \ldots V_k$ of size $m$ and exceptional set $V_0$. For each cluster $V_j$, let $V'_j = V_j \cap A$. We let $X_i$ be defined as before, then $|V'_j| = \sum_{x_i \in V_j}X_i$ and $\mathbb{E}(|V'_j|) = m/2$. We again use the Chernoff bound to see that
\begin{equation}\label{eq:A4}
\mathbb{P}\left(\left||V'_j|-\frac{m}{2}\right|>\varepsilon'm\right)<2e^{-2\varepsilon'^2m/3}<\frac{1}{4n}.
\end{equation}

Suppose that $A$ is a subset of $V(G)$ satisfying $(1/2-\varepsilon')m \leq |V'_j| \leq (1/2-\varepsilon')m$ for all clusters $V_j$. We will show that $A$ satisfies $(iv)$. Move at most $\varepsilon'm$ vertices from each cluster $V'_i$ into the exceptional set $V_0'$ so that each cluster has size $m'\coloneqq (1/2-\varepsilon')m$. Then $|V_0'|\leq 2\varepsilon' n$. Consider any set $S \subseteq A$ such that $\tau'|A| \leq |S| \leq (1-\tau')|A|$. We will say that $S$ \emph{intersects} $V'_i$ \emph{significantly} if
$$|V'_i \cap S| > \nu^2m.$$ Write $k^*$ for the number of clusters $S$ intersects significantly. Then $k^*$ satisfies
$$k^*m' + (k-k^*)\nu^2 m +2\varepsilon' n \geq |S|$$ which implies
$$k^* \geq |S|/m'-3\nu^2k \eqqcolon q.$$
Let $Q$ be a set of $q$ clusters $V_i$ such that $V'_i$ is intersected significantly by $S$. In particular, note that 
$$\tau'|A|/m' - 3\nu^2k \leq q \leq (1-\tau')|A|/m' - 3\nu^2k$$
and so
$$2\tau k \leq q \leq (1-2\tau)k.$$
We now recall that $G$ is a robust $(v,\tau)$-diexpander, and so $R$, the reduced digraph, is a robust $(v/2,2\tau)$-diexpander, by \Cref{diexpander}. Then $|N^+_R(Q)| \geq q+\nu k/2$ and $|N^-_R(Q)| \geq q+\nu k/2$. Now, each edge in $R$ corresponds to an $\varepsilon'$-regular pair of clusters of density at least $d'$. Suppose $V_j \in N^+_R(Q)$, then $V_iV_j \in E(R)$ for some $V_i \in Q$.
We can apply \Cref{prop:neighbours} to see that all but at most $\varepsilon'm$ vertices in $V_j$ (and hence all but at most $\varepsilon'm$ vertices in $V'_j$) have at least
$$(d'-\varepsilon')\nu^2m \geq \nu'|A|$$ inneighbours in $S$. Thus for each $V_j \in N_R^+(Q$, all but at most $\varepsilon'm$ vertices in $V_j$ lie in $RN^+_{\nu',G[A]}(S)$. Therefore,
\begin{equation*}
\begin{split}
|RN^+_{\nu',G[A]}(S)| &\geq \left( q +\nu k/2\right)(m'-\varepsilon'm) \\
&= \left( |S|/m' - 3\nu^2k +\nu k/2\right)(m'-\varepsilon'm) \\
&\geq |S| - |S|\varepsilon'm/m' - 3\nu^2km' +\nu k(m'-\varepsilon' m)/2 \\
&\geq |S|+\left(\nu/4 - 3\nu^2-3\varepsilon'\right)|A|\\
&\geq |S|+\nu'|A|
\end{split}
\end{equation*}
In a similar way, we see that $|RN^-_{\nu',G[A]}(S)|\geq |S|+\nu'|A|$ and so $G[A]$ is a robust $(\nu',\tau')$-diexpander. The same reasoning applied to the graph $G\setminus A$ allows us to conclude that $G\setminus A$ is also a robust $(\nu',\tau')$-diexpander. So we have satisfied property $(iv)$.

Together, equations \eqref{eq:A1}--\eqref{eq:A4} imply that with probability at most $1/n$ our randomly chosen set $A$ does not satisfy properties $(i)$--$(iv)$. Therefore, we will choose a set $A$ satisfying the properties with positive probability and hence such a set exists.
\end{proof}

Let us choose such a subset $A \subseteq V(G)$. We observe that
\begin{equation}\label{eq:g[a]mindeg}
\delta^0(G[A]) \geq \left( \frac{\delta^0(G)}{n} - \frac{\eta}{10}\right)|A| \geq \frac{9\eta}{10}|A|.
\end{equation}
We also note that
\begin{equation*}
\begin{split}
\frac{\delta^+(G\setminus A)}{|G\setminus A|}&\geq \frac{\delta^+(G) - (\delta^+(G)/n +\eta/10)(1/2+\varepsilon)n}{(1/2-\varepsilon)n} \\
&\geq \frac{\delta^+(G)}{n} -\frac{\eta}{9} \geq \frac{2\eta}{3}.
\end{split}
\end{equation*}
We get a similar bound for the minimum indegree and hence 
\begin{equation}\label{eq:g-amindeg}
\delta^0(G\setminus A)\geq \frac{2\eta}{3}|G\setminus A|.
\end{equation}

We now apply the Diregularity Lemma with parameters $\varepsilon^2$, $2d$ and $M'_B$	to the graph $G \setminus A$. We obtain a partitition of the vertices into $M_B \geq M'_B$ clusters,  $V_1, \ldots, V_{M_B}$, with $|V_1|=\ldots=|V_{M_B}| \eqqcolon m'_B$, and an exceptional set $V_0$. Let $R_B$ be the reduced  graph and $G'_B$ be the pure digraph. Then we use \Cref{Routexpander}, \eqref{eq:g-amindeg} and \Cref{diexpander} to see that:
\begin{itemize}
\item $R_B$ is a $(\nu'/2, 2\tau')$-diexpander and
\item $\delta^0(R_B) \geq \eta M_B/3$.
\end{itemize}
Therefore, $R_B$ contains a Hamilton cycle $F_B$, by \Cref{thm:stdhamcycle}. Relabel the vertices of $R_B$ if necessary so that $F_B= V_1V_2 \ldots V_{M_B}V_1$. We can then apply \Cref{lem:subgraphsuperreg} to show that by moving $4\varepsilon^2m'_B$ from each cluster into the exceptional set $V_0$ we may assume that each edge in $F_B$ corresponds to an $(\varepsilon,d)^*$-superregular pair and each edge of $R_B$ corresponds to an $\varepsilon$-regular pair of density at least $d$ (in $G'_B$). We continue to denote the clusters by $V_1, \ldots , V_{M_B}$ and let $m_B$ denote the new cluster size. The resulting exceptional set, which we will continue to call $V_0$, satisfies $$|V_0| \leq \varepsilon^2n+4\varepsilon^2m'_BM_B \leq \varepsilon n.$$ Let  $B \coloneqq \bigcup_{i=1}^{M_B} V_i$ and write $G^*_B$ for the graph $G'_B[B]$.

Let us now consider the graph $G[A \cup V_0]$. By our choice of $A$, all vertices $x \in V_0$ satisfy $|N^+_A(x)| \geq 9 \eta|A|/10$ and $|N^-_A(x)| \geq 9 \eta|A|/10$. Combining this with \eqref{eq:g[a]mindeg}, we get that $$\delta^0(G[A \cup V_0]) \geq \frac{2\eta}{3}|A \cup V_0|.$$ We have added at most $\varepsilon n \leq (\nu')^2 |A|$ vertices to $G[A]$, so we can apply \Cref{prop:addingrobexp} to see that $G[A \cup V_0]$ is still a robust $(\nu'/2,2\tau')$-diexpander. We now carry out a similar process again, this time applying the Diregularity Lemma with parameters $\varepsilon_A^2/4$, $2d$ and $M'_A$ to $G[A \cup V_0]$. We obtain a partition into $M_A \geq M'_A$ clusters, $V'_1, \ldots, V'_{M_A}$, with $|V'_1|=  \ldots = |V'_{M_A}|$ and an exceptional set $V'_0$. Let $A' \coloneqq \bigcup_{i=1}^{M_A} V'_i$, $R_A$ be the reduced digraph and $G'_A$ the pure digraph. As previously, we can show that:
\begin{itemize}
\item $R_A$ is a robust $(\nu'/4, 4\tau')$-diexpander and
\item $\delta^0(R_A) \geq \eta M_A/3.$
\end{itemize}
Then apply \Cref{thm:stdhamcycle} to find a Hamilton cycle $F_A$ in $R_A$. Again, we apply \Cref{lem:subgraphsuperreg} to show that by moving at most $\varepsilon_A^2|A\cup V_0|$ vertices into the exceptional set $V'_0$ we may assume that each edge in $F_A$ corresponds to an $(\varepsilon_A,d)^*$-superregular pair and each edge of $R_A$ corresponds to an $\varepsilon_A$-regular pair of density at least $d$ (in $G'_A$). Let $m_A$ denote the resulting cluster size, $G_B$ denote the graph $G[B\cup V'_0]$ and $n_B\coloneqq |G_B|$. The new exceptional set satisfies
\begin{equation}\label{eq:except1}
|V'_0| \leq (\varepsilon_A^2/4 + \varepsilon_A^2)|A \cup V_0|\leq \varepsilon_A|A \cup V_0|/2 < \varepsilon_An_B.
\end{equation}
The constants $M_A$ and $M_B$ satisfy
$$0<\varepsilon^* \ll 1/M_A \ll \varepsilon_A \ll 1/M_B \ll \varepsilon \ll d \ll \nu'.$$

\section{Case 1: $C$ Contains Many Neutral Pairs}

\subsection*{Step 1: Splitting up the Cycle}

Now that we have two subgraphs, $G[A']$ and $G_B$ which partition the vertex set of $G$, we want to also split the cycle $C$ into two subpaths and then embed one of the subpaths into each graph. Recall that $n(C) = \lambda n$ denotes the number of neutral pairs in $C$. In this section we will assume that $C$ contains many neutral pairs, i.e, that $\lambda>\varepsilon_4$, although the process will be very similar when we consider the case when $C$ is close to the standard orientation. We will begin by assigning vertices to the clusters in the reduced  graphs.

Recall that we defined $\mathcal{Q}$ to be a maximal collection of neutral pairs in $C$, all at a distance at least $3$ from each other. Let $v^*$ be a vertex such that both subpaths of $C$ of length $n/2$ which have $v^*$ as an endvertex contain at least $2/5$ of the elements of $\mathcal{Q}$.

We set $r \coloneqq \lceil 8/\nu'\rceil$. Recall that the reduced digraphs $R_A$ or $R_B$ are both robust $(\nu'/4,4\tau')$-diexpanders. Then, by \Cref{lem:isopath}, given any pair of vertices in $R_A$ or any pair in $R_B$, we can find any orientation of a path of length $r$ between them. Let
$$s \coloneqq \lfloor(\log n)^2\rfloor \;\text{ and } \;t\coloneqq \left\lfloor \frac{n-r(s+1)}{s+2}\right\rfloor.$$
We will divide $C$ into overlapping subpaths, sharing endvertices, so that
$$C = Q_1P_1Q_2P_2\ldots Q_sP_sQ^*P^*$$
where $Q_1$ starts at the vertex $v^*$ and $\ell(P_i) = t$, $\ell(Q_i) = \ell(Q^*) = r$ and $2t\leq \ell(P^*)<3t$.

Let $s_B$ be a positive integer such that
$$1<n_B - s_B(r+t)< \ell(P^*).$$
Then define the path
$$P_B \coloneqq P^*_BQ_1P_1\ldots Q_{s_B}P_{s_B}$$
where $P^*_B$ is an end segment of $P^*$ chosen so that $|P_B| = n_B$. Then
\begin{equation}\label{eq:nb}
n_B=s_B(r+t)+\ell(P^*_B)+1.
\end{equation}
Let $P_A$ be the remainder of $C$, that is
$$P_A \coloneqq Q'_1P'_1 \ldots Q'_{s_A}P'_{s_A} Q^*P^*_A$$
where $s_A\coloneqq s-s_B$, $Q'_i \coloneqq Q_{s_B+i}$, $P'_i \coloneqq P_{s_B+i}$ and $P^*_A$ is an initial segment of $P^*$ overlapping $P^*_B$ in exactly one place.

We will begin by assigning the vertices of $P_B$ to the vertices of $R_B$. Ideally, we would like each cluster to be assigned a similar number of vertices and a similar number of neutral pairs. Define a subset $\mathcal{Q}_B \subseteq \mathcal{Q}$ be the set of all neutral pairs which are contained in a $P_i$ for some $i$ and are at a distance at least $3$ from its ends. We can apply \Cref{lemma16} with $\gamma = \varepsilon^*/2$	to obtain an embedding of $P_1, \ldots, P_{s_B}$  in $R_B$ such that, for all $1\leq i \leq M_B$:
$$\left|a(i)-\frac{s_B(t+1)}{M_B}\right|\leq\frac{\varepsilon^*s_B(t+1)}{2} \;\text{ and }\; \left|n(i,\mathcal{Q}_B) - \frac{|\mathcal{Q}_B|}{M_B}\right| \leq \frac{\varepsilon^*s_B(t+1)}{2}$$
and hence
\begin{equation}\label{eq:rbassignment}
\left|a(i)-\frac{s_Bt}{M_B}\right|\leq\varepsilon^*s_Bt \;\text{ and }\; \left|n(i,\mathcal{Q}_B) - \frac{|\mathcal{Q}_B|}{M_B}\right| \leq \varepsilon^*s_Bt.
\end{equation}
Then
\begin{eqnarray}\label{eq:7}
|a(i)-m_B| &=& \left|a(i)-\frac{n_B-|V'_0|}{M_B}\right| \nonumber \\
&\stackrel{\eqref{eq:except1}}{\leq}& \left|a(i)-\frac{n_B}{M_B}\right|+2\varepsilon_Am_B \nonumber \\
&\stackrel{\eqref{eq:nb}}{\leq}& \left|a(i)-\frac{s_Bt}{M_B}\right|+ \left|\frac{s_Br+3t}{M_B}\right| +2\varepsilon_Am_B  \nonumber \\
& \stackrel{\eqref{eq:rbassignment}}{\leq}& \varepsilon^*s_Bt+ \varepsilon^*m_B+2\varepsilon_Am_B .
\end{eqnarray}

Let us now estimate $|\mathcal{Q}_B|$. Recall that $\mathcal{Q}$ is a maximal set of neutral pairs all at a distance at least $3$ from each other, so we note that $|\mathcal{Q}| \geq \lambda n/4$. We selected $P_B$ to contain at least $2|\mathcal{Q}|/5 \geq \lambda n/10$ neutral pairs. At most $s_Br+3t$ neutral pairs can be contained in the paths $Q_i$ and $P^*_B$. We will lose at most $4s_B$ neutral pairs which are in a $P_i$ but too close to its ends. Therefore,
$$|\mathcal{Q}_B| \geq \lambda n/10 - (s_Br+3t+4s_B)$$
and so we use \eqref{eq:rbassignment} to see that
\begin{equation}\label{eq:neutralpairs}
\begin{split}
n(i,\mathcal{Q}_B) &\geq \left(\frac{\lambda n}{10} - (s_Br+3t+4s_B)\right)\frac{1}{M_B} - \varepsilon^*s_Bt \\
& \geq \frac{\lambda n_B}{6M_B} - 2\varepsilon^*n_B \geq \frac{\lambda m_B}{7}.
\end{split}
\end{equation}

We now use \Cref{lem:isopath} to connect each pair $P_{i-1},P_i$ by a path in $R_B$ which is isomorphic to $Q_i$. We also greedily extend $P_1$ backwards by a path isomorphic to $P^*_BQ_1$. This will assign at most
\begin{equation}\label{eq:addingsome1}
s_Br+3t\leq \varepsilon^*m_B
\end{equation}
additional vertices to each $V_i$. Let us denote this assignment of $P_B$ to $R_B$ (which can be thought of as a walk in $R_B$) by $W_B$.

\subsection*{Step 2: Incorporating the Exceptional Vertices}

In this section, we will show that we can modify the walk $W_B$ so as to include all of the vertices in $V'_0$. Let us define an extended reduced graph $R^*_B$ to be the graph formed by the union of $R_B$ and the vertices in $V'_0$. For each $v\in V'_0$ and each $V_i \in V(R_B)$, we add the edge:
\begin{itemize}
\item $vV_i$ if $|N^+_G(V) \cap V_i| > \eta m_B/10$.
\item $V_iv$ if $|N^-_G(V) \cap V_i| > \eta m_B/10$
\end{itemize}
By our choice of $A$, we have that $v$ has at least $\eta n_B/3$ inneighbours in $B$ and so $v$ must have at least one inneighbour in $R^*_B$.  For each $v \in V'_0$, choose an inneighbour $V_i \in V(R_B)$ and change the assignment of one neutral pair mapped to $V_iV_{i+1}V_i$ to $V_ivV_i$. This process reduces $a(i+1)$ and $n(i, \mathcal{Q}_B)$ by one. We are able to choose a distinct neutral pair for each vertex in $V'_0$ since \eqref{eq:neutralpairs} implies that $n(i, \mathcal{Q}_B) \geq \lambda m_B/7 > |V'_0|$. After carrying out this process for all exceptional vertices we reduced $a(i)$ by  at most $|V'_0| < \varepsilon_An_B$ for each cluster and we use this together with \eqref{eq:7} and \eqref{eq:addingsome1}, to see that for all $V_i \in V(R_B)$ the modified $a(i)$ now satisfies
\begin{equation}\label{eq:ai}
|a(i)-m_B| \leq \varepsilon^*s_Bt + \varepsilon^*m_B +2\varepsilon_Am_B + \varepsilon^*m_B + |V'_0| < 4\varepsilon_An_B.
\end{equation}

We also note that for each $V_i$ we still have
$$n(i,\mathcal{Q}_B) \geq \frac{\lambda m_B}{7} - |V'_0| \geq \frac{\lambda m_B}{8}.$$
Each vertex assigned to a $V_i \in V(R_B)$ must have both of its neighbours in $V_{i-1} \cup V_{i+1}$, unless it is the assignment of a vertex in a $Q_i$ or was part of a neutral pair used to incorporate an exceptional vertex. So there can be at most $\varepsilon_An_B + 2|V'_0| \leq 3 \varepsilon_A n_B$ vertices assigned to some $V_i$ which do not have their neighbours in $V_{i-1} \cup V_{i+1}$. Therefore, we have found a $(4\varepsilon_A, 3\varepsilon_A)$-corresponding assignment of $P_B$ to $R^*_B$.

\subsection*{Step 3: Obtaining a Balanced Assignment}

We will now show that we can use skewed traverses to modify the assignment so that it becomes balanced, that is, every cluster is assigned exactly $m_B$ vertices. For each cluster $V_i \in V(R_B)$ with $a(i) > m_B$, choose a cluster $V_j$ with $a(j)<m_B$. By \Cref{prop:shortskewed}, we can find a skewed traverse $T(V_{i-1},V_j)$ of length $\ell \leq \lceil 4/\nu' \rceil$ in $R_B$, say, 
$$T(V_{i-1},V_j) = V_{i-1}V_{i_1}, V_{i_1-1}V_{i_2},V_{i_2-1}V_{i_3}, \ldots, V_{i_\ell-1}V_j.$$
As detailed in \Cref{sub:skewed}, we can use this skewed traverse to reduce $a(i)$ by one and increase $a(j)$ by one. We do this by replacing sections of $W_B$ which correspond to the assignment of a neutral pair in $C$ and are of the form $$V_{i-1}V_iV_{i-1}, V_{i_1-1}V_{i_1}V_{i_1-1}, V_{i_2-1}V_{i_2}V_{i_2-1},\ldots, V_{i_\ell-1}V_{i_\ell}V_{i_\ell-1}$$ by edges from $T(V_{i-1},V_j)$: $$V_{i-1}V_{i_1}V_{i-1}, V_{i_1-1}V_{i_2}V_{i_1-1}, V_{i_2-1}V_{i_3}V_{i_2-1}, \ldots, V_{i_\ell-1}V_{j}V_{i_\ell-1}.$$   The number of vertices assigned to all other clusters in $R_B$ remains unchanged. Since $$n(i, \mathcal{Q}_B) \geq \lambda m_B/8 > 4\varepsilon_AM_Bn_B \geq \sum_{i=1}^{M_B}|a(i)-m_B|$$ we are able to carry out this process for all clusters $V_i$ with $a(i) > m_B$ to obtain a balanced embedding which we continue to call $W_B$.

We can check that each cluster $V_i \in V(R_B)$ has now been assigned at most $3 \varepsilon_An_B+ 8\varepsilon_AM_Bn_B< 9\varepsilon_AM_Bn_B$ vertices which do not have both of their neighbours in $V_{i-1} \cup V_{i+1}$. So we now have that $W_B$ is a $9\varepsilon_AM_B$-corresponding embedding of $P_B$ into $R^*_B$.

\subsection*{Step 4: Finding a Copy of $P_B$ in $G_B$}

Now that we have found a $9\varepsilon_AM_B$-corresponding assignment of $P_B$ to $R^*_B$, we will show that we can use this to find an embedding, $W'_B$, of $P_B$ in the graph $G_B$. This embedding will have the following properties:

\begin{itemize}
\item Each vertex of $P_B$ assigned to some $v \in V'_0$ by $W_B$ is also assigned to $v$ by $W'_B$;
\item For each $V_i \in V(R_B)$, each occurrence of $V_i$ in $W_B$ is replaced by a unique vertex in $V_i$;
\item Each edge of $W_B$ which does not lie on an edge of $F_B$ is mapped to an edge in $G_B$.
\end{itemize}

Define $W_{B,1}$ to be the digraph consisting of all maximal walks, $u_{i,1}u_{i,2}\ldots u_{i, \ell_i}$,
in $W_B$ of length at least $3$ and with all their edges lying on $F_B$. Let $W_{B,2}$ be the digraph $W_B \setminus W_{B,1}$. Then $W_{B,2}$ is a union of walks of the form $v_{i,1}v_{i,2}\ldots v_{i,k_i}$ where $u_{i,1} = v_{i-1,k_{i-1}}$ and $u_{i,\ell_i} = v_{i,1}$.

A walk might be in $W_{B,2}$ because of any of the following three of our earlier techniques:
\begin{enumerate}
\item Incorporating an exceptional vertex; \label{1}
\item Embedding the paths $Q_i$ and $P^*_B$; \label{2}
\item Using skewed traverses to correct imbalances. \label{3}
\end{enumerate}
These are the only situations in which a walk in $W_{B,2}$ can arise, since we asked that all neutral pairs in $\mathcal{Q}$ were separated by at least $3$ edges.

Let us embed the paths of $W_{B,2}$ greedily in the following way. If the walk is of type \ref{1}, then let $v \in V'_0$ be the exceptional vertex and $V_i$ the cluster adjacent to $v$ in $W_B$. Recall that we chose $V_i$ so that $|N^-_G(v) \cap V_i|>\eta m_B/10$. Since $|V'_0| < \varepsilon_An_B \leq \varepsilon m_B/3 $, we can choose distinct vertices $x,y \in N^-_G(v) \cap V_i$ for each such path. If the walk is of type \ref{2}, then we will again embed the walk greedily so that its image is a path, this time in the graph $G^*_B \subseteq G_B$. We use  that each edge of $R_B$ corresponds to an $\varepsilon$-regular pair of density at least $d$ and that the total length of paths of type \ref{2} is at most $s_Br+3t \leq \varepsilon m_B/3$ (together with \Cref{prop:neighbours} and \Cref{prop:subsets}). The final type, type \ref{3}, comes from a skewed traverse and will be of the form $V_iV_jV_i$. We must embed at most $3(9\varepsilon_AM_Bn_B)\leq \varepsilon m_B/3$ vertices into paths of this type, so we can again assign these vertices greedily. We have now embedded all walks in $W_{B,2}$ so that the image of each of the walks is a path.

For each $V_i$ let $S_i \subseteq V_i$ be the set consisting of all vertices which have not yet been assigned a vertex or have been assigned an endvertex in $W_{B,2}$. We then apply \Cref{lemma10} to the graph $G^*_B$ with $H \coloneqq W_{B,1}$ and with  $x_{P_i}$ defined to be the vertex in $G^*_B$ assigned $u_{i,1}$ and $y_{P_i}$ the vertex assigned $u_{i,\ell_i}$. We obtain an embedding of $W_{B,1}$ into $G^*_B[\bigcup S_i]$. Together with the embedding of $W_{B,2}$, we obtain an embedding of $P_B$ in $G_B$ which satisfies all of the desired properties. 

\subsection*{Step 5: Finding a Copy of $C$ in $G$}
Let $u, v \in V(G)$ be the vertices to which we assigned the endvertices of $P_B$. We must now find an embedding of 
$$P_A = Q'_1P'_1 \ldots Q'_{s_A}P'_{s_A} Q^*P^*_A$$
in the graph $G_A \coloneqq G[A' \cup \{u,v\}]$ which starts at $v$ and ends at $u$. We will follow an almost identical method to that for embedding $P_B$, the main difference will be that this time we only have two exceptional vertices, $u$ and $v$.

Define $\mathcal{Q}_A \subseteq \mathcal{Q}$ to be the set of all neutral pairs which are contained in a $P'_i$ and are at a distance at least $3$ from its ends. Then embed the $P'_i$ using \Cref{lemma16} and use \Cref{lem:isopath} to embed $Q'_i$ for each $i>1$, as in Step 1. We must also greedily embed $P^*_A$ and again use \Cref{lem:isopath} connect this path to $P'_{s_A}$ by a path isomorphic to $Q^*$. In Step 2, we only have to consider the exceptional set $\{u,v\}$ and then we repeat Steps 3 and 4 to correct imbalances and then find a copy of $P_A$ in $G_A$ from $v$ to $u$. All of the equations follow through after replacing $m_B, M_B, s_B$ by $m_A, M_A, s_A$. We combine these embeddings to obtain a copy of $C$ in $G$.

\section{Case 2: $C$ Contains Few Neutral Pairs}
In this case, the cycle $C$ is close to the standard orientation. A lot of the methods used will be similar to those in Case 1 but, since $C$ now contains few neutral pairs, we will need a new way to incorporate exceptional vertices and correct imbalances. Unless otherwise stated, the notation used in this section is as defined for Case 1.

\subsection*{Step 1$'$: Splitting up the Cycle}
Let us define 
$$\ell_B \coloneqq \left\lceil \frac{4}{\nu'} \right\rceil M_B.$$ Since $R_B$ is a robust $(\nu'/4, 4\tau')$-outexpander, \Cref{prop:shortskewed} implies that $\ell_B$ is the maximum number of cycles that must be traversed by a shifted walk between any two vertices in $R_B$. Split the cycle into subpaths $P_A$ and $P_B$ as in Step 1 (this time $v^*$ can be taken to be any vertex in $C$). We define a  \emph{long run} to be a directed subpath of $C$ of length $3\ell_B$. Let $\mathcal{Q}'_B$ be a maximal collection of long runs in $P_B$ all oriented in the same direction and at a distance at least $3$ from each other. We observe that
$$|\mathcal{Q}'_B| \geq \left\lfloor \frac{n_B-1}{3\ell_B +3} \right\rfloor - 2\lambda n \geq \frac{\nu' n_B}{50M_B}.$$ 
We have subtracted $2\lambda n$ since a neutral pair or an inverse neutral pair can both prevent a long run.

We let $\mathcal{Q}_B \subseteq \mathcal{Q}'_B$ be a maximal subset containing long runs which are contained in a $P_i$, and at a distance at least $4$ from its ends, and are all oriented in the same direction. We will assume that all of the long runs in $\mathcal{Q}_B$ are oriented in the same direction as $F_B$. Since $\ell(Q_i)+8 = r+8 < 3\ell_B$, the path formed by $Q_i$ together with the last $4$ edges of $P_{i-1}$ and the initial $4$ edges of $P_i$ can intersect at most $2$ long runs in $Q'_B$. The path $Q^*P^*_B$ extended backwards by $4$ edges has length at most $r+3t+4$ so it can also intersect at most $4t/3\ell_B$ long runs. So we lose at most $2s_B+4t/3\ell_B$ long runs in $\mathcal{Q}'$ because they are intersected by a $Q_i$ or $Q^*P^*_B$ or are too close to a $Q_i$.  We lose at most half of the remaining long runs by selecting a maximal set oriented in the same direction. Hence
$$|\mathcal{Q}_B| \geq \frac{1}{2}\left(\frac{\nu' n_B}{50M_B} - (2s_B+4t/3\ell_B) \right) \geq \frac{\nu' n_B}{110M_B}.$$

We apply \Cref{lemma16} to $R_B$ with $P_1, \ldots, P_{s_B}$, $\mathcal{Q}_B$ and $\varepsilon^*/2$ to obtain an embedding of the $P_i$ with
\begin{equation*}
\left|a(i)-\frac{s_Bt}{M_B}\right|\leq\varepsilon^*s_Bt \;\text{ and }\; n(i,\mathcal{Q}_B) \geq \frac{|\mathcal{Q}_B|}{M_B} - \varepsilon^*s_Bt \geq \frac{\nu'n_B}{120M_B^2}.
\end{equation*}

Exactly as in Step 1, we use \Cref{lem:isopath} to connect each pair $P_{i-1},P_i$ by a path isomorphic to $Q_i$. We also greedily extend $P_1$ backwards by a path isomorphic to $P^*_BQ_1$. This process assigns at most $s_Br+3t\leq \varepsilon^*m_B$ additional vertices to each $V_i$. Let us denote this assignment of $P_B$ to $R_B$ by $W_B$. As before, we will view $W_B$ as a walk in $R_B$.

\subsection*{Step 2$'$: Incorporating the Exceptional Vertices}

Suppose that $v \in V'_0$. Then let $V_i, V_j \in V(R_B)$ be such that $V_iv, vV_j \in E(R_B^*)$, we can find such clusters by the definition of $R_B^*$ and our choice of $A$. Choose a long run in $\mathcal{Q}_B$ whose initial vertex is assigned to $V_i$. Note that since $M_B$ divides $3\ell_B$, the subwalk of $W_B$ corresponding to this long run also ends at $V_i$. We will show that we can replace the assignment of this long run in $R_B$ by an extended shifted walk $W'$ in order to incorporate $v$ without significantly changing the number of vertices assigned to any cluster. Let $W'$ be the following walk of length $3\ell_B$
$$W' \coloneqq V_ivW(V_j, V_{i+3})F_BF_B\ldots F_BV_i.$$
Notice that, since the walk $W(V_j, V_{i+3}) \setminus \{V_j,V_{i+3}\}$ visits all clusters in $R_B$ the same number of times, we have that $M_B$ divides $V_ivW(V_j,V_{i+3})F_BV_i$. As $M_B$ divides $3\ell_B$, by adding as many extra copies of $F_B$ as necessary, we can indeed find a walk of the form $W'$  which finishes at $V_i$.
This replacement causes:
\begin{itemize}
\item $a(i+1)$ and $a(i+2)$ to decrease by one;
\item $a(j)$ to increase by one.
\end{itemize}
If we carry out this process for all exceptional vertices then, as in Case 1, we have that for all $V_i\in V(R_B)$
$$|a(i)-m_B| \leq \varepsilon^*s_Bt + \varepsilon^*m_B +2\varepsilon_Am_B + \varepsilon^*m_B + |V'_0| < 4\varepsilon_An_B.$$
We observe that for each $V_i$
$$n(i,\mathcal{Q}_B) \geq \frac{\nu'n_B}{120M_B^2} - |V'_0| \geq \frac{\nu'n_B}{150M_B^2}.$$
We also have at most 
$\varepsilon_An_B + 4|V'_0| \leq 5\varepsilon_An_B$
vertices assigned to $V_i$ which do not have their neighbours assigned to $V_{i-1} \cup V_{i+1}$. These vertices arose when we embedded the $Q_i$ (at most $\varepsilon_An_B$ vertices) and when we replaced the assignment of a long run in order the incorporate an exceptional vertex. Each of the replacement walks creates at most $4$ of these vertices - the clusters adjacent to the exceptional vertices and at either end of the shifted walk.
So our current assignment of $P_B$ in $R_B^*$ is a $(4\varepsilon_A, 5\varepsilon_A)$-corresponding assignment.

\subsection*{Step 3$'$: Obtaining a Balanced Assignment}
We must now modify our assignment to obtain a balanced assignment. We will do this using shifted walks. Suppose that $V_i \in V(R_B)$ with $a(i)>m_B$, then we can find $V_j \in V(R_B)$ such that $a(j)<m_B$. We replace a subwalk in $W_B$ corresponding to the assignment of a long run which starts at $V_{i-1}$ by a walk
$$W(V_{i-1},V_j)W(V_j,V_{i+1})F_B \ldots F_BV_{i-1}$$ of length $3\ell_B$. We are able to find a walk of this length since $M_B$ divides $3\ell_B$. As discussed in \Cref{sub:skewed}, this decreases $a(i)$ by one, increases $a(j)$ by one and leaves $a(p)$ unchanged for all $p \neq i,j$. Since
$$n(i,\mathcal{Q}_B) \geq \frac{\nu'n_B}{150M_B^2} > 4\varepsilon_An_B,$$
for all $V_i \in V(R_B)$, we can obtain a balanced assignment by repeating this process as many times as necessary.

There are now at most 
$$5\varepsilon_An_B+4(4\varepsilon_An_B) = 21 \varepsilon_An_B$$
vertices assigned to each $V_i$ which do not have their neighbours assigned to $V_{i-1} \cup V_{i+1}$. So we currently have a $21\varepsilon_A$-corresponding assignment of $P_B$ to $R^*_B$.

\subsection*{Step 4$'$: Finding a Copy of $P_B$ in $G_B$}
We define $W_{B,1}$ and $W_{B,2}$ exactly as in Step 4. Since the total length of the paths in $W_{B,2}$ is at most $21\varepsilon_An_BM_B< \varepsilon m_B$, we can follow the same process as in Step 4 to find a copy of $P_B$ in $G_B$. 

\subsection*{Step 5$'$: Finding a Copy of $C$ in $G$}
As in Step 5, we repeat Steps 1$'$ to 4$'$ to find a copy of $P_A$ in $G_A$, starting and finishing at the required vertices. This embedding, together with the embedding of $P_B$ in $G_B$, gives the desired cycle $C$ in $G$. This completes the proof of \Cref{anyorient2}.

\chapter{Conclusion}
\label{chap:conclusion}

Szemer\'edi's Regularity Lemma, \Cref{reglemma}, has been fundamental throughout this project. Roughly speaking, it tells us that we can approximate any sufficiently large graph by a random-like graph. It allows us to partition the vertices of the graph into a bounded number of clusters so that most of the clusters induce $\varepsilon$-regular pairs. Whilst it would be very difficult to embed straight into the graph $G$, it is a much more manageable task to begin by embedding an often simpler structure into the reduced graph. Then, using the Key Lemma, \Cref{keylem}, or some form of the Blow-up Lemma of Koml\'os, S\'ark\"ozy and Szemer\'edi, we find an embedding of the desired subgraph in $G$. This method works for digraphs as well as undirected graphs since there is analogue of the Regularity Lemma for digraphs due to Alon and Shapira, the Diregularity Lemma, \Cref{direg}.

Using the regularity method, we were able to prove some well-known results in extremal graph theory including the Erd\H{o}s-Stone theorem. The Erd\H{o}s-Stone theorem has a famous corollary which determines asymptotically the number of edges required in a graph $G$ to force a copy of any non-bipartite graph $H$ as a subgraph. We also considered an application to Ramsey theory. Recall that the Ramsey number $R(H)$ is defined to be the smallest natural number such that any colouring of the edges of a complete graph on $R(H)$ vertices using two colours yields a monochromatic copy of $H$. In general these numbers are very difficult to calculate. Using the Regularity Lemma we proved \Cref{ramseyresult}, giving a bound on $R(H)$ which is linear in $|H|$ for graphs $H$ of bounded maximum degree. This is a significant improvement on the usual exponential bound if the maximum degree of $H$ is not bounded. 

A key aim in extremal graph theory is often to find a spanning structure in a graph. As we saw, this presents a new problem, it is relatively easy to find a structure in the reduced graph but we are now forced to find a way to incorporate the exceptional vertices as well. For instance, the third application of the Regularity Lemma which we met involved finding perfect $F$-packings. Recall that a perfect $F$-packing in a graph $G$ is a spanning subgraph of $G$ consisting of vertex disjoint copies of $F$. Tutte's theorem completely characterises those graphs which contain a perfect matching but Hell and Kirkpatrick \cite{hell} showed that the decision problem of whether a graph $G$ contains a perfect $F$-packing is NP-complete if $F$ contains a component on at least three vertices. So instead of aiming for a characterisation of all graphs containing a perfect $F$-packing, it makes sense to look for sufficient minimum degree conditions. In particular, we used the regularity method to find a minimum degree condition for a sufficiently large graph to contain a perfect $C_6$-packing. In \cite{ksosz2}, Koml\'os, S\'ark\"ozy and Szemer\'edi, give a minimum degree condition for general graphs $F$ based on the chromatic number. K\"uhn and Osthus, in \cite{ko-packing}, improved this to a result which is best possible up to a constant by using a refinement of the chromatic number.

Hamilton cycles were the main focus for the remainder of this project. Again, since the Hamilton cycle problem in NP-complete, it is unlikely that we can completely describe all graphs which have a Hamilton cycle so instead we look for sufficient conditions. These conditions will often involve the minimum (semi)degree, for example Dirac's theorem  for graphs \cite{dirac} or Ghouila-Houri's theorem for digraphs \cite{ghouila}, or the degree sequence of the graph or digraph. If $G$ is an undirected graph then Chv\'atal's theorem, \Cref{chvatal}, describes all degree sequences which guarantee a Hamilton cycle. Nash-Williams conjectured that an analogue of this result holds for digraphs. This conjecture remains an open problem, in fact, it is still unknown whether the degree sequence conditions of the conjecture even guarantee that every pair of vertices lie on a cycle. 

\newtheorem*{again}{Conjecture \ref{nashconjecture}}

\begin{again}[Nash-Williams, \cite{nash}]
Suppose that $G$ is a strongly connected digraph on $n \geq 3$ vertices such that 
\begin{enumerate}[(i)]
	\item $d^+_i \geq i+1$ or $d^-_{n-i} \geq n-i$ and
	\item $d^-_i \geq i+1$ or $d^+_{n-i} \geq n-i$
	\end{enumerate}
for all $i<n/2$. Then $G$ contains a Hamilton cycle.
\end{again}

In \Cref{chap:digraphs} we defined a robust outexpander, originally introduced by K\"uhn, Osthus and Treglown in \cite{kot}. Robust outexpanders have proved to be very useful in the study of Hamilton cycles. Informally, a graph is said to be a robust $(\nu, \tau)$-outexpander if when we consider any subset $S \subseteq V(G)$ which is neither too small nor too large, the set of vertices having at least $\nu n$ inneighbours in $S$ is slightly larger than $S$. So the expansion property of $G$ is resilient in that it cannot be destroyed by deleting a small number of vertices or edges. We are interested in such graphs because they are fairly common, for example, any sufficiently large oriented graph with minimum semidegree at least $(3/8+\alpha)n$ is a robust outexpander. We also showed that satisfying degree sequence conditions which are slightly stronger than those in Conjecture \ref{nashconjecture} implies robust outexpansion in \Cref{robdegseq}. In \Cref{expander} we showed that we can find a Hamilton cycle in a robust outexpander of linear minimum degree and hence we were able to prove an approximate version of Nash-Williams' conjecture.

In \Cref{thm:stdhamcycle}, Keevash, K\"uhn and Osthus prove that a minimum semidegree of $(3n-4)/8$ guarantees a Hamilton cycle in a sufficiently large oriented graph. Interestingly, we saw that this bound no longer suffices when we instead ask for Hamilton cycles of all possible orientations. In \Cref{anyorient}, Kelly showed that a minimum semidegree of $(3/8+\alpha)n$ guarantees any orientation of a Hamilton cycle in any sufficiently large oriented graph.
\newtheorem*{again1}{\Cref{anyorient}}

\begin{again1}[Kelly, \cite{kelly}]
For every $\alpha>0$ there exists an integer $n_0=n_0(\alpha)$ such that every oriented graph on $n\geq n_0$ vertices with $\delta^0(G) \geq (3/8+\alpha)n$ contains every orientation of a Hamilton cycle.
\end{again1}

\noindent Whether we can improve on this result to obtain an exact bound is an open problem. To attempt to answer this question for all possible orientations of a Hamilton cycle might be unrealistic but it would become more approachable if we were to restrict ourselves to finding, say, anti-directed Hamilton cycles.

Since we know that a graph satisfying the conditions of \Cref{anyorient} is a robust outexpander, it seemed natural to try to generalise this result for robust outexpanders. We obtained \Cref{cor:anyorient2}, proving that every sufficiently large robust outexpander of linear minimum semidegree contains any orientation of a Hamilton cycle. 

\newtheorem*{again2}{\Cref{cor:anyorient2}}

\begin{again2}
Let $n_0$ be a positive integer and $\nu, \tau, \eta$ be positive constants such that $1/n_0 \ll \nu \leq \tau \ll \eta <1$. Let $G$ be a digraph on $n \geq n_0$ vertices with $\delta^0(G) \geq \eta n$ and suppose $G$ is a robust $(\nu, \tau)$-outexpander. Then $G$ contains every orientation of a Hamilton cycle.
\end{again2}

\noindent We note that \Cref{anyorient} follows from \Cref{cor:anyorient2} and we also obtain a degree sequence condition guaranteeing any orientation of a Hamilton cycle as a corollary.


\begin{thebibliography}{999}


\bibitem{alon} N. Alon and D. Shapira.
\newblock Testing Subgraphs in Directed Graphs,
\newblock {\em Journal of Computer and System Sciences} \textbf{69} (2004) 354--382.

\bibitem{bollo} B. Bollob\'as.
\newblock Modern Graph Theory.
\newblock Graduate Texts in Mathematics; 184, Springer (1998).

\bibitem{bondy} J.A. Bondy and U.S.R. Murty.
\newblock Graph Theory,
\newblock Graduate Texts in Mathematics; 244, Springer (2008).

\bibitem{ckko} D. Christofides, P. Keevash, D. K\"uhn and D. Osthus.
\newblock A Semi-exact Degree Condition for Hamilton Cycles in Digraphs,
\newblock {\em SIAM Journal Discrete Math.} \textbf{24} (2010) 709--756.

\bibitem{diestel} R. Diestel.
\newblock Graph Theory,
\newblock Graduate Texts in Mathematics; 173, Springer-Verlag (2005, 3rd edition).

\bibitem{dirac} G.A. Dirac.
\newblock Some Theorems on Abstract Graphs,
\newblock {\em Proc. London Math. Soc.} \textbf{2} (1952), 69--81.

\bibitem{frieze} A. Frieze and M. Krivelevich.
\newblock On Packing Hamilton Cycles in Epsilon-Regular Graphs,
\newblock {\em J. Combinatorial Theory B} \textbf{94} (2005) 159--172.

\bibitem{ghouila} A. Ghouila-Houri.
\newblock Une Condition Suffisante d'Existence d'un Circuit Hamiltonien,
\newblock {\em C.R. Acad. Sci. Paris} \textbf{25} (1960), 495--497.

\bibitem{haggthom} R. H\"aggkvist and A. Thomason.
\newblock Oriented Hamilton Cycles in Oriented Graphs,
\newblock {\em Combinatorics, Geometry and Probability}, Cambridge University Press (1997) 339--353.

\bibitem{hell} P. Hell and D.G. Kirkpatrick.
\newblock On the Complexity of General Graph Factor Problems,
\newblock {\em SIAM J. Computing} \textbf{83} (2009) 601--609.

\bibitem{keevko} P. Keevash, D. K\"uhn and D. Osthus.
\newblock An Exact Minimum Degree Condition for Hamilton Cycles in Oriented Graphs,
\newblock {\em Journal of the London Math. Soc.} \textbf{79} (2009) 144--166.

\bibitem{kelly} L. Kelly.
\newblock Arbitrary Orientations of Hamilton Cycles in Oriented Graphs,
\newblock {\em The Electronic Journal of Combinatorics} \textbf{18} (2011).

\bibitem{kko} L. Kelly, D. K\"uhn and D. Osthus.
\newblock A Dirac Type Result on Hamilton Cycles in Oriented Graphs,
\newblock {\em Combin. Probab. Comput.} \textbf{17} (2008) 689--709.

\bibitem{ksosz} J. Koml\'os, G.N. S\'ark\"ozy and E. Szemer\'edi.
\newblock Blow-up Lemma, {\em Combinatorica} \textbf{17} (1997) 109--123.

\bibitem{ksosz2} J. Koml\'os, G.N. S\'ark\"ozy and E. Szemer\'edi.
\newblock Proof of the Alon-Yuster Conjecture, {\em Discrete Math.} \textbf{235} (2001) 255--260.

\bibitem{ko-packing} D. K\"uhn and D. Osthus.
\newblock The Minimum Degree Threshold for Perfect Graph Packings,
\newblock {\em Combinatorica} \textbf{29} (2009) 65--107.

\bibitem{ko1} D. K\"uhn and D. Osthus.
\newblock Embedding Large Subgraphs into Dense Graphs,
\newblock {\em Surveys in Combinatorics, London Math. Soc. Lecture Notes 365}, Cambridge University Press (2009), 137--167.

\bibitem{ko2} D. K\"uhn and D. Osthus.
\newblock A Survey on Hamilton Cycles in Directed Graphs,
\newblock {\em European J. Combinatorics} \textbf{33} (2012), 750--766.

\bibitem{kot} D. K\"uhn, D. Osthus and A. Treglown.
\newblock Hamilton Degree Sequences in Digraphs,
\newblock {\em J. Combinatorial Theory B} \textbf{100} (2012) 367--380.

\bibitem{kellyexact78} D. K\"uhn and D. Osthus.
\newblock Hamilton Decompositions of Regular Expanders: A Proof of Kelly's Conjecture for Large Tournaments,
\newblock {\em Advances in Math.} \textbf{237} (2013) 62--146.

\bibitem{molloy} M. Molloy and B. Reed.
\newblock Graph Colouring and the Probabilistic Method.
\newblock Springer (2002).

\bibitem{nash} C.St.J.A. Nash-Williams.
\newblock Hamiltonian Circuits,
\newblock {\em Studies in Math.} \textbf{12} (1975) 301--360.

\bibitem{posa} L. P\'osa.
\newblock A Theorem Concerning Hamiltonian Lines,
\newblock {\em Magyar Tud. Akad. Mat. Fiz. Oszt. Kozl.} \textbf{7} (1962), 225--226.

\bibitem{schr} A. Schrijver.
\newblock A Pythagoras Proof of Szemer\'edi's Regularity Lemma,
\newblock {\em ArXiv: 1212.3499v1 [math.CO]} (2012).

\bibitem{szem} E. Szemer\'edi.
\newblock Regular Partitions of Graphs,
\newblock {\em Probl\'emes Combinatoires et Th\'eorie des Graphes Colloques Internationaux CNRS} \textbf{260} (1978) 399--401.




\end{thebibliography}
\end{document}